\documentclass[final]{dmtcs-episciences}
\usepackage[utf8]{inputenc}
\usepackage{subfigure}

\usepackage{xcolor}
\usepackage{colortbl}

\usepackage{comment}

\usepackage{color}
\usepackage{polynom}
\usepackage{amsthm}
\usepackage{mathtools}
\usepackage{amsmath}
\usepackage{amsfonts}
\usepackage{amssymb}
\usepackage{graphicx}
\usepackage{hyperref}
\usepackage{appendix}
\usepackage[shortlabels]{enumitem}
\usepackage{transparent}

\def\Z{{\mathbb Z}}

\def\Aut{\textup{Aut}}
\def\span{\textup{Span}}

\def\1sat{$1$-saturated}
\def\isat{$i$-saturated}

\newtheorem{Thm}{Theorem}[section]
\newtheorem{Lemma}[Thm]{Lemma}
\newtheorem{Cor}[Thm]{Corollary}

\newtheorem{Prop}[Thm]{Proposition}
\newtheorem{Conj}{Conjecture}
\newtheorem{Question}{Question}

\theoremstyle{definition}
\newtheorem{Def}[Thm]{Definition}
\newtheorem{Eg}{Example}

\theoremstyle{remark}
\newtheorem*{remark}{Remark}

\newcommand{\defref}[1]{\hyperref[#1]{Definition~\ref*{#1}}}
\newcommand{\factref}[1]{\hyperref[#1]{Fact~\ref*{#1}}}
\newcommand{\conjref}[1]{\hyperref[#1]{Conjecture~\ref*{#1}}}
\newcommand{\thmref}[1]{\hyperref[#1]{Theorem~\ref*{#1}}}
\newcommand{\lemmaref}[1]{\hyperref[#1]{Lemma~\ref*{#1}}}
\newcommand{\corref}[1]{\hyperref[#1]{Corollary~\ref*{#1}}}
\newcommand{\propref}[1]{\hyperref[#1]{Proposition~\ref*{#1}}}
\newcommand{\secref}[1]{\hyperref[#1]{Section~\ref*{#1}}}
\newcommand{\refref}[1]{\hyperref[#1]{[\ref*{#1}]}}
\newcommand{\egref}[1]{\hyperref[#1]{Example~\ref*{#1}}}
\newcommand{\problemref}[1]{\hyperref[#1]{Problem~(\ref*{#1})}}
\newcommand{\figref}[1]{\hyperref[#1]{Figure~\ref*{#1}}}

\allowdisplaybreaks

\author[Aaron Potechin and Joseph Tsang]{
Aaron Potechin\affiliationmark{1}\thanks{Supported by NSF grant CCF: 2008920}
\and Hing Yin Tsang\thanks{Work done when the author was at the University of Chicago}
}
\title[On induced subgraphs of $H(n,3)$ with maximum degree $1$]{On induced subgraphs of $H(n,3)$ with maximum degree $1$}

\affiliation{The University of Chicago, Chicago, USA} 

\begin{document}
\publicationdata{vol. 28:2}{2026}{5}{10.46298/dmtcs.15440}{2025-03-31; 2025-03-31; 2025-12-09}{2025-12-28}
\maketitle

\begin{abstract}
    In this paper, we consider induced subgraphs of the Hamming graph $H(n,3)$. We show that if $U \subseteq \mathbb{Z}_3^n$ and $U$ induces a subgraph of $H(n,3)$ with maximum degree at most $1$ then
    \begin{enumerate}
    \item If $U$ is disjoint from a maximum size independent set of $H(n,3)$ then $|U| \leq 3^{n-1}+1$. Moreover, all such $U$ with size $3^{n-1}+1$ are isomorphic to each other.
    \item For $n \geq 6$, there exists such a $U$ with size $|U| = 3^{n-1}+18$ and this is optimal for $n = 6$.
    \item If $U \cap \{x, x+e_1, x+2e_1\} \ne \phi$ for all $x \in \mathbb{Z}
    _3^n$ then $|U| \le 3^{n-1} +81$.
    \end{enumerate}
\end{abstract}

\section{Introduction}

The Sensitivity Theorem \cite{NS94, Huang19} asserts that the sensitivity of a Boolean function is polynomially related to its decision tree complexity. The theorem was first conjectured by Nisan and Szegedy in the 90s \cite{NS94} and it was settled in 2019 by Huang \cite{Huang19} using a beautiful spectral argument. Huang's theorem shows that any induced subgraph on more than half of the vertices of a Boolean hypercube has maximum degree at least $\sqrt{n}$. This bound is tight \cite{CFGS88} and it is known to imply the Sensitivity Theorem \cite{GL92}.

Since then there has been a considerable amount of work trying to extend Huang's result. Huang's argument has been generalized to Cartesian products of directed $l$-cycles \cite{Tika22} and Cartesian products of paths \cite{ZH23}, as well as other products of graphs \cite{HLL20}. Alon and Zheng \cite{AZ20} considered arbitrary Cayley graphs over $\Z_2^n$ and showed that Huang's theorem implies any induced subgraph on more than half of the vertices must have maximum degree $\sqrt{d}$ where $d$ is the degree of the Cayley graph. Potechin and Tsang \cite{PT20} showed that this can be generalized to any abelian Cayley graph. Recently, Ansensio, Garc\'{i}a-Marco, and Knauer \cite{AGK24} conjectured that the Sensitivity Theorem can be generalized to $m$-ary functions.

For graphs which do not contain induced subcubes, Frankl and Kupavskii \cite{FK20}, and Chau et al. \cite{CEFL23} studied Kneser graphs and proved a relationship between the size and maximum degree of the induced subgraphs that is closely related to Huang's result on the Boolean hypercube. García-Marco and Knauer \cite{GK22} showed that analogues of Huang's result hold for several families of non-abelian Cayley graphs including the point-line incidence graph of the projective plane of order $q$, Kronecker double covers of the Ramanujan graphs of Lubotzky, Phillips, and Sarnak \cite{LPS88}, and random Cayley graphs with a sufficiently large number of generators.

On the other hand, it was shown that a number of graphs have low degree induced subgraphs on more than half of the vertices. Bishnoi and Das \cite{BD20} observed that the odd graphs with vertices $\{V \subseteq [2n+1]: |V| = n\}$ and edges $\{\{U,V\}: U \cap V = \phi\}$ contain 1-regular induced subgraphs on more than half of the vertices. This gives a family of vertex-transitive graphs of unbounded degree that contains large 1-regular induced subgraphs. Lehner and Verret \cite{LV20} constructed three families of graphs that contain 1-regular induced subgraphs on more than half of the vertices, including the odd graphs and families of non-abelian Cayley graphs with constant and unbounded degrees, refuting a conjecture of Potechin and Tsang \cite{PT20}. García-Marco and Knauer \cite{GK22} constructed further examples of infinite families of Cayley graphs that contain 1-regular induced subgraphs on more than half of the vertices.

Dong \cite{Dong21} and Tanday \cite{Tand22} studied the maximum degree of induced subgraphs of Hamming graphs $H(n, k)$. Given a graph $G$, let $\alpha(G)$ denote the maximum size of an independent set of $G$. Dong \cite{Dong21} constructed a bipartite induced subgraph on more than $\alpha(H(n,k))$ vertices with maximum degree at most $\lceil \sqrt{n} \rceil$. For $k \geq 3$, this result was strengthened by \cite{Tand22} who gave a construction of an induced subgraph on $\alpha(H(n,k))+1$ vertices with maximum degree 1. \cite{GK22} independently observed that $H(n, 3)$ has a 1-regular induced subgraph on $3^{n-1}+1$ vertices.

In this paper, we further investigate the Hamming graph $H(n,3)$ (i.e., the graph with vertices $\Z_3^{n}$ and edges $\{\{x, y\} \in \Z_3^n \times \Z_3^n : d_H(x, y) = 1\}$ where $d_H$ is the Hamming distance). In particular, we investigate the following question:
\begin{Question} 
    Let $U \subseteq \Z_3^n$. If the induced subgraph of $H(n,3)$ on $U$ has maximum degree at most 1, how large can $U$ be?
\end{Question}
\begin{remark}
It is not hard to show that if $G$ is a $d$-regular graph on $N$ vertices than any induced subgraph of $G$ with maximum degree $k$ has at most $\frac{dN}{2d-k}$ vertices. While this bound is tight for some graphs, for $H(n,3)$ this bound is worse than the upper bound of $\frac{3^{n}-1}{2}$ which can be obtained from Huang's result \cite{PT20}. For this problem, we are not aware of better bounds.
\end{remark}
\subsection{Our Results}
We know that $|U|$ can be at least $3^{n-1}+1$ according to \cite{GK22, Tand22}. Our first result shows that this is the largest possible under an extra assumption on $U$.

\begin{Thm}
    Let $U \subseteq \Z_3^n$. If $U$ is disjoint from a maximum size independent set of $H(n,3)$ and the induced subgraph of $H(n,3)$ on $U$ has maximum degree at most $1$ then $|U| \le 3^{n-1}+1$.
\end{Thm}

However, without the assumption of being disjoint from a maximum size independent set, $U$ can be somewhat larger. We construct examples of size $3^{n-1}+K$ where $K = 2, 6, 18$ for $n = 4, 5, 6$ respectively. Our example for $n = 6$ can be extended to give an example of size $3^{n-1}+18$ for any $n \geq 6$.

\begin{Thm}
    For $n \ge 6$, there exists $U \subseteq \Z_3^n$ such that $|U| = 3^{n-1}+18$ and $U$ induces a subgraph of $H(n,3)$ with maximum degree 1.
\end{Thm}

It can be shown that our examples for $n \in [6]$ are the largest possible. Interestingly, they all share a common property that there exists $i \in [n]$ such that every line along direction $i$ intersects with them. We say that the set is $i$-saturated if it satisfies this intersection property with lines along direction $i$. Our main result is an upper bound on the size of $U$ when $U$ is $i$-saturated for some $i \in [n]$.

\begin{Thm}
\label{SAT_upper_bound}
    For all $U \subseteq \Z_3^n$ and $i \in [n]$, if $U$ is $i$-saturated and $U$ induces a subgraph of $H(n,3)$ with maximum degree at most $1$ then $|U| \le 3^{n-1}+81$.
\end{Thm}

One lemma in our proof, Lemma \ref{SAT_line_extend}, is verified by a SAT solver. In Section \ref{sec:lineextensionlemma}, we present a proof of a weaker version of this lemma (cf. Lemma \ref{line_extend}) which is sufficient to show the weaker upper bound that all $i$-saturated subsets of $\Z_3^n$ with induced degree $1$ have size at most $3^{n-1}+729$.

\subsection{Organization}
In Section \ref{section:pre}, we describe the notation we will use for our arguments and for illustrating our constructions. In Section \ref{section:canonical}, we define canonical sets and prove that they are the unique extremal induced degree 1 subsets that are disjoint from a maximum size independent set. Canonical sets and their properties will play a crucial role in our upper bound on $1$-saturated subsets. We construct examples of large induced degree 1 subsets in Section \ref{section:lower} and prove our size upper bound on $1$-saturated induced degree 1 subsets in Section \ref{section:upper}.

\section{Preliminaries}
\label{section:pre}
\begin{Def}[Hamming graph $H(n,3)$]
The Hamming graph $H(n,3)$ is the graph with vertices $\Z_3^n$ and edges $\{\{x, y\} \in \Z_3^n \times \Z_3^n : d_H(x, y) = 1\}$ where $d_H$ is the Hamming distance.
\end{Def}
\begin{Def}[Standard basis for $\Z_3^n$]
We use the standard basis $e_1,\ldots,e_n$ for $\Z_3^n$. In other words, we take $e_i \in \Z_3^n$ to be the point which is $1$ in coordinate $i$ and $0$ in the other coordinates.
\end{Def}
\begin{Def}[All zero point]
We define $0^n$ to be the point $(0,\ldots,0) \in \Z_3^n$ whose coordinates are all $0$.
\end{Def}
\begin{Def}[Induced degree]
    We say that $U \subseteq \Z_3^n$ has \emph{induced degree} $d$ if $U$ induces a subgraph of $H(n,3)$ with maximum degree $d$.
\end{Def}
Throughout the paper, we will draw diagrams representing subsets $U \subseteq \Z_3^n$. For these diagrams, we will use the following conventions. 
\begin{enumerate}
\item Each square represents a subset of $\Z_3^k$ for some $k \geq 0$. When $k = 0$, we represent a point of $U$ by $\bullet$ and we represent the empty set $\phi$ by an empty space.
\item For each $3 \times 3$ block, the first column corresponds to $x_{k+1} = 0$, the second column corresponds to $x_{k+1} = 1$, and the third column corresponds to $x_{k+1} = 2$. Similarly, if there is more than one row, the first row corresponds to $x_{k+2} = 0$, the second column corresponds to $x_{k+2} = 1$, and the third column corresponds to $x_{k+2} = 2$.
\item When there is more than one block, each additional direction represents an additional coordinate.
\end{enumerate}
\begin{Eg}
The following diagram shows the set of points $U = \{(0,0), (0,2), (1,1), (2,1)\}$. Note that $|U| = 4$ and $U$ induces a subgraph of $H(n,3)$ where every vertex has degree $1$.
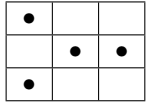
\[
\begin{tabular}{|c|c|c|}
 \hline
 $\bullet$  &   &  \\
 \hline
  & $\bullet$ & $\bullet$ \\
 \hline
 $\bullet$ &  &  \\
 \hline
\end{tabular}
\]
\end{Eg}

We now describe some basic facts about $H(n,3)$ and maximum size independent sets of $H(n,3)$.
\color{black}
\begin{Prop}
The automorphisms of the Hamming graph $H(n,3)$ are generated by the following operations.
\begin{enumerate}
    \item Permuting the coordinates.
    \item Multiplying a coordinate by $2$.
    \item Adding $x \in \Z_3^n$ to every point.
\end{enumerate}
\end{Prop}

\begin{proof}
We claim that the automorphisms of $H(n, 3)$ are precisely the maps of the form
\[
\sigma(x_1,\ldots,x_n) = ({c_1}x_{\pi^{-1}(1)} + y_1,\ldots,{c_n}x_{\pi^{-1}(n)} + y_n)
\]
where $y = (y_1,\ldots,y_n) \in \Z_3^n$, $\pi \in S_n$, and $c_1,\ldots,c_n \in \{1,2\}$. Note that these maps can be generated using the three operations above as follows:
\begin{enumerate}
\item Apply the permutation $\pi \in S_n$ to the coordinates.
\item For each $i \in [n]$, multiply coordinate $i$ by $c_i \in \{1,2\}$.
\item Add $y = (y_1,\ldots,y_n) \in \Z_3^n$ to each point.
\end{enumerate}
To see why any automorphism $\sigma \in Aut(H(n, 3))$ must be of this form, consider an arbitrary $\sigma \in Aut(H(n, 3))$ and observe the following:
\begin{enumerate}
\item $\sigma(0^n) \in \Z_3^n$ so we can take $y = \sigma(0^n)$.
\item For each $i \in [n]$, $\sigma(0^n)$ and $\sigma(e_i)$ have Hamming distance $1$ so $\sigma(e_i) = y + ce_{j}$ for some $j \in [n]$ and $c \in \{1,2\}$. Thus, we can set $\pi(i) = j$ and set $c_j = c$.
\item For each pair of distinct indices $i,i' \in [n]$, $\sigma(e_i)$ and $\sigma(e_{i'})$ have Hamming distance $2$ so we cannot have that $\pi(i') = \pi(i)$. This implies that for each $j \in [n]$, there is exactly one $i \in [n]$ such that $\pi(i) = j$ so $\pi$ will be a permutation of $[n]$ and the coefficients $c_1,\ldots,c_n \in \{1,2\}$ are well-defined as for each $j \in [n]$, $c_j$ is set exactly once.
\end{enumerate}
We claim that for all $x = (x_1,\ldots,x_n) \in \Z_3^n$, $\sigma(x) = y + \sum_{i=1}^{n}{{c_{\pi(i)}}{x_i}e_{\pi(i)}}$. We show this by induction on the number of nonzero coordinates of $x$. When $x = 0^{n}$, $\sigma(x) = \sigma(0^n) = y$. When $x = e_i$ for some $i \in [n]$, $\sigma(x) = \sigma(e_i) = y + c_{\pi(i)}e_{\pi(i)}$. When $x = 2e_i$ for some $i \in [n]$, since the Hamming distance between $\sigma(2e_i)$ and $\sigma(0^n)$ is $1$ and the Hamming distance between $\sigma(2e_i)$ and $\sigma(e_i)$ is $1$, we must have that $\sigma(x) = \sigma(2e_i) = y + 2c_{\pi(i)}e_{\pi(i)}$. Thus, in both cases we have that $\sigma(x) = y + \sum_{i=1}^{n}{{c_{\pi(i)}}{x_i}e_{\pi(i)}}$.

For the inductive step, assume that $\sigma(x) = y + \sum_{i=1}^{n}{{c_{\pi(i)}}{x_i}e_{\pi(i)}}$ for all $x \in \Z_3^n$ such that $x$ has at most $k-1$ nonzero coordinates for some $k \geq 2$ and consider an $x \in \Z_3^n$ with $k$ nonzero coordinates. Observe that for all $j \in [n]$ such that $x_j \neq 0$, 
\begin{enumerate}
\item By the inductive hypothesis, $\sigma(x_1,\ldots,x_{j-1},0,x_{j+1},\ldots,x_n) = y + \sum_{i \in [n] \setminus \{j\}}{{c_{\pi(i)}}{x_i}e_{\pi(i)}}$.
\item $\sigma(x)$ and $\sigma(x_1,\ldots,x_{j-1},0,x_{j+1},\ldots,x_n)$ have Hamming distance $1$.
\end{enumerate}
Since $k \geq 2$, this implies that $\sigma(x) = y + \sum_{i=1}^{n}{{c_{\pi(i)}}{x_i}e_{\pi(i)}}$, as needed.
\end{proof}

\begin{remark}
    When we say two subsets of $\Z_3^n$ are isomorphic, we mean that there is an automorphism of the Hamming graph $H(n,3)$ mapping one subset to the other.
\end{remark}

\begin{Def}
Given a graph $G$, let $\alpha(G)$ denote the maximum size of an independent set of $G$. 
\end{Def}
For the Hamming graph $H(n,3)$, we have that $\alpha(H(n,3)) = 3^{n-1}$ and the maximum size independent sets of $H(n,3)$ are hyperplanes of the vector space $\Z_3^n$.

\begin{Def}
Let $n \geq 2$ be a natural number. Given a set of vertices $S \subseteq \Z_{3}^{n-1}$, for each $c \in \Z_3^n$, we define $(S,c) = S \times \{c\} \subseteq \Z_3^n$ to be the set of points 
\[
(S,c) = \{(x_1,\ldots,x_n) \in \Z_3^n: (x_1,\ldots,x_{n-1}) \in S, x_n = c\}
\]
\end{Def}
\begin{Prop}
\label{coset_char}
$I$ is a maximum size independent set of $H(n,3)$ if and only if there exist $b \in \{1, 2\}^n$ and $c \in \Z_3$ such that $b_1 = 1$ and $I = \{x \in \Z_3^n : \sum_{i = 1}^{n}{b_i{x_i}} \equiv c \mod 3\}$. Moreover, for each $I$ of this form, the only independent sets of $H(n,3)$ of size $3^{n-1}$ which are disjoint from $I$ are the independent sets $I' = \{x \in \Z_3^n : \sum_{i = 1}^{n}{b_i{x_i}} \equiv c+1 \mod 3\}$ and $I'' = \{x \in \Z_3^n : \sum_{i = 1}^{n}{b_i{x_i}} \equiv c+2 \mod 3\}$.
\end{Prop}
\begin{proof}
We prove this by induction. For the base case $n = 1$, the only independent sets of size $1$ are $\{0\}$, $\{1\}$, and $\{2\}$ which are given by $\{x \in \Z_3: x \equiv 0 \mod 3\}$, $\{x \in \Z_3: x \equiv 1 \mod 3\}$, and $\{x \in \Z_3: x \equiv 2 \mod 3\}$ respectively. All three of these indepdendent sets are disjoint.

For the inductive step, assume that the statement is true for $n \leq k$. By the inductive hypothesis, maximum size independent sets of $H(k,3)$ have size $3^{k-1}$ so the only way to obtain an independent set $I$ of $H(k+1,3)$ of size $3^{k}$ is if $I = (I_0,0) \cup (I_1,1) \cup (I_2,2)$ where $I_0,I_1,I_2$ are disjoint independent sets of $Z_{3}^{k}$ of size $3^{k-1}$. By the inductive hypothesis, $I_0 = \{x \in \Z_3^k : \sum_{i = 1}^{k}{b_i{x_i}} \equiv c \mod 3\}$ for some $b \in \{1,2\}^k$ and $c \in \Z_3$ such that $b_1 = 1$. Moreover, since the only independent sets of $H(k,3)$ of size $3^{k-1}$ which are disjoint from $I_0$ are $I' = \{x \in \Z_3^k : \sum_{i = 1}^{k}{b_i{x_i}} \equiv c+1 \mod 3\}$ and $I'' = \{x \in \Z_3^k : \sum_{i = 1}^{k}{b_i{x_i}} \equiv c+2 \mod 3\}$, we must either have that $I_1 = I'$ and $I_2 = I''$ or $I_1 = I''$ and $I_2 = I'$. In the first case, we have that $I = \{x \in \Z_3^k : \sum_{i = 1}^{k}{b_i{x_i}} + 2x_{k+1} \equiv c \mod 3\}$. In the second case, we have $I = \{x \in \Z_3^k : \sum_{i = 1}^{k}{b_i{x_i}} + x_{k+1} \equiv c \mod 3\}$.

To show the moreover statement, observe that if $b,b' \in \{1, 2\}^n$, $b_1 = b'_1 = 1$ and $b' \neq b$ then for any $c,c' \in \Z_3$, the linear equations $\sum_{i = 1}^{n}{b_i{x_i}} \equiv c \mod 3$ and $\sum_{i = 1}^{n}{b'_i{x_i}} \equiv c' \mod 3$ have $3^{n-2}$ common solutions.
\end{proof}

For our analysis, it is useful to fix three disjoint maximum size independent sets of $H(n,3)$.
\begin{Def}\label{ABCn}
For each $n \in \mathbb{N}$, we define $A_n = \{x \in \Z_3^n : \sum_{i = 1}^n x_i  \equiv 0 \mod 3\}$, $B_n = \{x \in \Z_3^n : \sum_{i = 1}^n x_i \equiv 1 \mod 3\}$ and $C_n = \{x \in \Z_3^n : \sum_{i = 1}^n x_i  \equiv 2 \mod 3\}$. 
\end{Def}
The following recursive definition of $A_n$, $B_n$, and $C_n$ is useful.
\begin{Prop} $A_1 = \{0\}$, $B_1 = \{1\}$ and $C_1 = \{2\}$ and for all $n \in \mathbb{N}$,
\begin{enumerate}
    \item $A_{n+1} = (A_{n},0) \cup (C_{n},1) \cup (B_{n},2)$. We can represent this visually as  $A_{n+1} = $
\begin{tabular}{|c|c|c|}
 \hline
  $A_{n}$ & $C_{n}$ & $B_{n}$ \\
 \hline
\end{tabular}.
    \item $B_{n+1} = (B_{n},0) \cup (A_{n},1) \cup (C_{n},2)$. We can represent this visually as  $B_{n+1} = $
\begin{tabular}{|c|c|c|}
 \hline
  $B_{n}$ & $A_{n}$ & $C_{n}$ \\
 \hline
\end{tabular}.
    \item $C_{n+1} = (C_{n},0) \cup (B_{n},1) \cup (A_{n},2)$. We can represent this visually as  $C_{n+1} = $
\begin{tabular}{|c|c|c|}
 \hline
  $C_{n}$ & $B_{n}$ & $A_{n}$ \\
 \hline
\end{tabular}.
\end{enumerate}
\end{Prop}

\begin{figure}
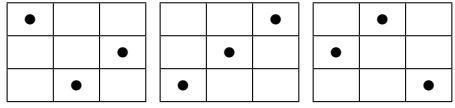

    \centering
\begin{tabular}{|c|c|c|}
 \hline
 $\bullet$  &  &  \\
 \hline
   &   & $\bullet$ \\
 \hline
  & $\bullet$ &  \\
 \hline
\end{tabular} \ 
\begin{tabular}{|c|c|c|}
 \hline
 &  & $\bullet$ \\
 \hline
   & $\bullet$ &  \\
 \hline
 $\bullet$ &   &  \\
 \hline
\end{tabular} \ 
\begin{tabular}{|c|c|c|}
 \hline
   & $\bullet$ &  \\
 \hline
$\bullet$ &   &  \\
 \hline
  &  & $\bullet$ \\
 \hline
\end{tabular}
    \caption{This figure shows the independent set $A_3$. Note that the first block of $A_3$ is $A_2$, the second block of $A_3$ is $C_2$, and the third block of $A_3$ is $B_2$.}
    \label{fig:independentset}
\end{figure}

For an illustration of the independent set $A_3$, see Figure \ref{fig:independentset}.

\subsection{Collapsing one dimension of $\Z_3^n$}
For our analysis, it is very useful to collapse one dimension of $\Z_3^n$. To do this, we represent a subset $U \subseteq \Z_3^n$ as a function $U_f : \Z_3^{n-1} \to \mathcal{P}(Z_3)$.
\begin{Def}
    Let $U$ be a subset of $\Z_3^n$ and $\mathcal{P}(\Z_3)$ denote the power set of $\Z_3^n$. We define $U_f$ to be the function $U_f : \Z_3^{n-1} \to \mathcal{P}(Z_3)$ such that $U = \cup_{x \in \Z_3^{n-1}}{(U_f(x) \times \{x\})}$.
\end{Def}

We will focus on subsets $U$ with induced degree at most 1. For these subsets, the corresponding function $U_f$ has range $\mathcal{P}(\Z_3) \setminus \{\Z_3\}$ so the following definition is useful.

\begin{Def}
    We define $A = A_1 = \{0\}$, $B = B_1 = \{1\}$, $C = C_1 = \{2\}$, $X = \Z_3 \setminus A = \{1,2\}$, $Y = \Z_3 \setminus B = \{0,2\}$, and $Z = \Z_3 \setminus C = \{0,1\}$.
\end{Def}
For an illustration of $A$, $B$, $C$, $X$, $Y$, and $Z$, see Figure \ref{fig:ABCXYZ}.

\begin{figure}[!hbt]
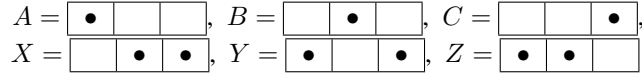

    \centering
$A = $
\begin{tabular}{|c|c|c|}
 \hline
  $\bullet$ & ${\transparent{0} \bullet}$ & ${\transparent{0} \bullet}$ \\
 \hline
\end{tabular}, \ 
$B = $
\begin{tabular}{|c|c|c|}
 \hline
  ${\transparent{0} \bullet}$ & $\bullet$ & ${\transparent{0} \bullet}$ \\
 \hline
\end{tabular}, \ 
$C = $
\begin{tabular}{|c|c|c|}
 \hline
  ${\transparent{0} \bullet}$ & ${\transparent{0} \bullet}$ & $\bullet$ \\
 \hline
\end{tabular}, \\
$X = $
\begin{tabular}{|c|c|c|}
 \hline
  ${\transparent{0} \bullet}$ & $\bullet$ & $\bullet$ \\
 \hline
\end{tabular}, \ 
$Y = $
\begin{tabular}{|c|c|c|}
 \hline
  $\bullet$ & ${\transparent{0} \bullet}$ & $\bullet$ \\
 \hline
\end{tabular}, \ 
$Z = $
\begin{tabular}{|c|c|c|}
 \hline
  $\bullet$ & $\bullet$ & ${\transparent{0} \bullet}$ \\
 \hline
\end{tabular} \ 

    \caption{This figure shows $A$, $B$, $C$, $X$, $Y$, and $Z$.}
    \label{fig:ABCXYZ}
\end{figure}

\begin{Prop}
\label{ABCXYZcharacterization}
Let $U = \cup_{x \in \Z_3^{n-1}}U_f(x) \times \{x\} \subseteq \Z_3^n$. Then $U$ has maximum induced degree at most $1$ if and only if the following conditions are satisfied for all $x \in \Z_3^{n-1}$ and $N(x) = \{y \in \Z_3^{n-1} : d_H(x, y) = 1\}$:
\begin{enumerate}
\item $U_f(x) \ne \Z_3$.
\item If $U_f(x) = A$, then $U_f(y) = A$ for at most one $y \in N(x)$ and $U_f(y) \notin \{Y, Z\}$ for all $y \in N(x)$.
\item If $U_f(x) = B$, then $U_f(y) = B$ for at most one $y \in N(x)$ and $U_f(y) \notin \{X, Z\}$ for all $y \in N(x)$.
\item If $U_f(x) = C$, then $U_f(y) = C$ for at most one $y \in N(x)$ and $U_f(y) \notin \{X, Y\}$ for all $y \in N(x)$.
\item If $U_f(x) = X$, then $U_f(y) \in \{\phi, A\}$ for all $y \in N(x)$.
\item If $U_f(x) = Y$, then $U_f(y) \in \{\phi, B\}$ for all $y \in N(x)$.
\item If $U_f(x) = Z$, then $U_f(y) \in \{\phi, C\}$ for all $y \in N(x)$.
\end{enumerate}
\end{Prop}

\begin{Def}[$i$-saturated]
Given $U \subseteq \Z_3^n$ and $i \in [n]$, we say that $U$ is \emph{{\isat}} if $U \cap \{x,x+e_i,x+2e_i\} \neq \phi$ for all $x \in \Z_3^n$. In other words, $U$ is $i$-saturated if and only if every line in direction $e_i$ contains at least one point of $U$.
\end{Def}

In terms of the function representation, $U$ is 1-saturated if and only if $U_f(x) \ne \phi$ for all $x \in \Z_3^{n-1}$. Together with Proposition \ref{ABCXYZcharacterization}, we have the following:

\begin{Prop}
If $U \subseteq \Z_3^{n}$, $U$ is \1sat and $U$ induces a subgraph of maximum degree at most $1$ then for all $x \in \Z_3^{n-1}$, $U_f(x) \in \{A,B,C,X,Y,Z\}$.
\end{Prop}

\section{Canonical Sets}
\label{section:canonical}
We first consider subsets of $\Z_3^n$ that have induced degree $1$ and are disjoint from a maximum size independent set of $H(n,3)$. It was known that such subsets can have size $3^{n-1}+1$ \cite{GK22}. In this section, we show that this is the maximum possible size for such a subset. Specifically, if $U \subseteq \Z_3^n$ has induced degree $1$ and $U$ is disjoint from a maximum size independent set of $H(n,3)$, then $|U| \le 3^{n-1}+1$. Moreover, up to isomorphism, there is only one such subset of size $3^{n-1}+1$.

\begin{Def}
    We say that $U \subseteq \Z_3^n$ is a \emph{canonical set} if $|U| \geq 3^{n-1}+1$, $U$ is disjoint from a maximum size independent set of $H(n,3)$, and $U$ has induced degree at most $1$.
\end{Def}
\subsection{Definition of $D_n$ and facts about $D_n$}
Up to isomorphism, the only canonical set in $\Z_3^n$ is the set $D_n$ which is defined as follows.
\begin{Def}
\label{D_n}

We define $D_n$ recursively.
\begin{enumerate}
\item $D_1 = \{1, 2\}$
\item For all $n \in \mathbb{N}$, $D_{n+1} = (D_{n},0) \cup (A_{n},1) \cup (A_{n},2)$.
\end{enumerate}
\end{Def}

\begin{figure}[!hbt]
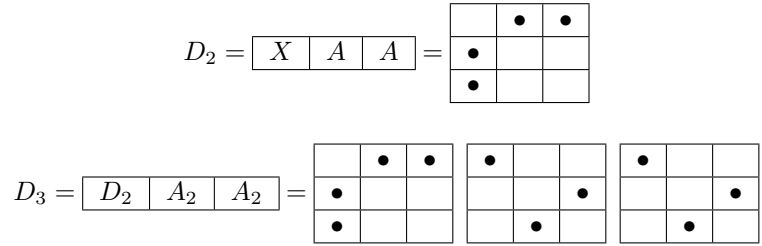

    \centering
$D_2 = $
\begin{tabular}{|c|c|c|}
 \hline
 $X$ & $A$ & $A$ \\
 \hline
\end{tabular} $=$
\begin{tabular}{|c|c|c|}
 \hline
  & $\bullet$ & $\bullet$ \\
 \hline
 $\bullet$  &   &  \\
 \hline
 $\bullet$ &  &  \\
 \hline
\end{tabular}\\
\vspace{5mm}
    $D_3 = $
\begin{tabular}{|c|c|c|}
 \hline
 $D_2$ & $A_2$ & $A_2$ \\
 \hline
\end{tabular} 
$=$
\begin{tabular}{|c|c|c|}
 \hline
  & $\bullet$ & $\bullet$ \\
 \hline
 $\bullet$  &   &  \\
 \hline
 $\bullet$ &  &  \\
 \hline
\end{tabular} \ 
\begin{tabular}{|c|c|c|}
 \hline
 $\bullet$  &  &  \\
 \hline
   &  & $\bullet$ \\
 \hline
  & $\bullet$ &  \\
 \hline
\end{tabular} \ 
\begin{tabular}{|c|c|c|}
 \hline
 $\bullet$  &  &  \\
 \hline
   &  & $\bullet$ \\
 \hline
  & $\bullet$ &  \\
 \hline
\end{tabular}

    \caption{This figure shows $D_2$ and $D_3$.}
    \label{fig:canonicalset}
\end{figure}

\begin{Prop}
For all $n \in \mathbb{N}$, $A_n \cap D_n = \phi$ and all vertices in $D_n$ have degree $1$.
\end{Prop}

\begin{proof}
We prove this by induction. For $n = 1$, $D_1 = \{1,2\}$ so both of the vertices in $D_1$ have degree $1$. Since $A_1 = \{0\}$, $D_1 \cap A_1 = \phi$, as needed. 

For the inductive step, assume that all vertices in $D_n$ have degree $1$ and $A_n \cap D_n = \phi$. Recall that 
\begin{enumerate}
\item $A_{n+1} = $
\begin{tabular}{|c|c|c|}
 \hline
  $A_{n}$ & $C_{n}$ & $B_{n}$ \\
 \hline
\end{tabular}
\item $D_{n+1} = $
\begin{tabular}{|c|c|c|}
 \hline
  $D_{n}$ & $A_{n}$ & $A_{n}$ \\
 \hline
\end{tabular}
\end{enumerate}
Since $A_n \cap D_n = \phi$, $A_n \cap B_n = \phi$, and $A_n \cap C_n = \phi$, $D_{n+1} \cap A_{n+1} = \phi$. To see that all vertices in $D_{n+1}$ have degree $1$, observe the following:
\begin{enumerate}
\item By the inductive hypothesis, $A_n \cap D_n = \phi$ and each vertex in $D_{n}$ has degree $1$.
\item Each vertex in one of the copies of $A_n$ is only adjacent to the same vertex in the other copy of $A_n$.
\end{enumerate}
\end{proof}

\begin{Prop}
\label{prop:numberoflines}
For all $n \in \mathbb{N}$,
\begin{enumerate}
    \item There is exactly one line of the form $\{x,x+e_1,x+2e_1\}$ which contains two points of $D_n$.
    \item For all $k \in [2,n]$, there are exactly $3^{k-2}$ lines of the form $\{x,x+e_k,x+2e_k\}$ which contain two points of $D_n$.
\end{enumerate}
\end{Prop}
\begin{proof}
We prove this by induction. For $n = 1$, there is only one line $\{0,e_1,2e_1\}$ and this line contains two points of $D_1$. For the inductive step, assume that the result is true for $D_n$ and consider $D_{n+1} = $
\begin{tabular}{|c|c|c|}
 \hline
  $D_{n}$ & $A_{n}$ & $A_{n}$ \\
 \hline
\end{tabular}.
Observe that since $A_n$ is an independent set of $\Z_3^n$, for all $k \in [n]$, a line of the form $\{x, x+e_k, x + 2e_k\}$ in $\Z_3^{n+1}$ contains two points of $D_{n+1}$ if and only if $x_{n+1} = 0$ and the line $\{x, x+e_k, x + 2e_k\}$ contains two points of $D_n$. Thus, by the inductive hypothesis, there is exactly one line of the form $\{x,x+e_1,x+2e_1\}$ which contains two points of $D_{n+1}$ and for all $k \in [2,n]$, there are exactly $3^{k-2}$ lines of the form $\{x,x+e_k,x+2e_k\}$ which contain two points of $D_{n+1}$.

For $k = n+1$, since $A_n \cap D_n = \phi$, a line of the form $\{x, x+e_{n+1}, x + 2e_{n+1}\}$ contains two points of $D_{n+1}$ if and only if $(x_1,\ldots,x_n) \in A_n$. Since $|A_n| = 3^{n-1}$, there are exactly $3^{n-1} = 3^{k-2}$ such lines.
\end{proof}

\begin{Cor}
    For all $n \geq 2$, $D_n$ is \1sat and $2$-saturated but is not {\isat} for any $i \geq 3$.
\end{Cor}
\begin{proof}
Since $|D_{n}| = 3^{n-1} + 1$ and $D_n$ has maximum induced degree $1$, for all $k \in [n]$, the number of lines of the form $\{x,x+e_k,x+2e_k\}$ which contain zero points of $D_n$ is equal to the number of lines of the form $\{x,x+e_k,x+2e_k\}$ which contain two points of $D_n$ minus one. By Proposition \ref{prop:numberoflines}, this is $0$ for $k \leq 2$ and $3^{k-2} - 1 > 0$ when $k > 2$.
\end{proof}

We now observe that for all $n \in \mathbb{N}$, there is exactly one other maximum size independent set of $H(n,3)$ which is disjoint from $D_n$.
\begin{Def}
Define 
\[
A'_n = \left\{x \in \Z_3^n : \left(\sum_{i = 1}^{n-1}{x_i}\right) + 2x_n \equiv 0 \mod 3\right\} = (A_{n-1},0) \cup (B_{n-1},1) \cup (C_{n-1},2).
\]
\end{Def}
\begin{Prop}
\label{unique_max_ind}
For all $n \in \mathbb{N}$, if $I$ is a maximum size independent set of $H(n,3)$ and $I \cap D_n = \phi$ then $I = A_n$ or $I = A'_n$.
\end{Prop}

\begin{proof}
We prove this for $n+1$ rather than $n$ to make the diagrams nicer.

Recall that $D_{n+1} = $
\begin{tabular}{|c|c|c|}
 \hline
  $D_{n}$ & $A_{n}$ & $A_{n}$ \\
 \hline
\end{tabular}
and let $I$ be an independent set of $H(n+1,3)$ of maximum size which is disjoint from $D_{n+1}$. Since the only independent sets of $H(n,3)$ of maximum size which are disjoint from $A_{n}$ are $B_{n}$ and $C_{n}$, $I$ must be equal to one of the following two possibilities 
\begin{enumerate}
\item $I = $
\begin{tabular}{|c|c|c|}
 \hline
  $I'$ & $B_{n}$ & $C_{n}$ \\
 \hline
\end{tabular}
 \item $I = $
\begin{tabular}{|c|c|c|}
 \hline
  $I'$ & $C_{n}$ & $B_{n}$ \\
 \hline
\end{tabular}
\end{enumerate}
for some independent set $I'$ of $H(n,3)$ of maximum size. Since $A_{n}$ is the only independent set of 
$H(n,3)$ of maximum size which is disjoint from $B_{n}$ and $C_{n}$, we must have $I' = A_{n}$. This implies that $I = A_{n+1}$ or $I = A'_{n+1}$, as needed.
\end{proof}

Before proving our uniqueness theorem, we need one more fact.
\begin{Lemma}\label{lem:canonicalsetintersection}
If $U \subseteq \Z_3^n$ is isomorphic to $D_n$ and $A_n \cap U = \phi$ then $|U \cap B_n| = |U \cap C_n| = \frac{|U|}{2}$.
\end{Lemma}
\begin{proof}
Recall that every vertex of $D_n$ has degree $1$ so there is a matching $M$ of size $\frac{3^{n-1}+1}{2}$ between the vertices of $U$. For each edge $\{u,v\} \in M$, at most one of $u$ and $v$ are in $B_n$ as $B_n$ is an independent set. Similarly, at most one of $u$ and $v$ are in $C_n$. Since $u,v \notin A_n$ as $U \cap A_n = \phi$, either $u$ in $B_n$ and $v \in C_n$ or $v$ in $B_n$ and $u \in C_n$. Since this is true for all edges $\{u,v\} \in M$, the result follows.
\end{proof}
\subsection{Uniqueness of canonical sets}
We are now ready to prove our uniqueness theorem.

\begin{Thm}
\label{uniqueDn}
Let $U \subseteq \Z_3^n$. If $|U| \geq 3^{n-1}+1$, $U$ has induced degree 1 and $U$ is disjoint from a maximum size independent set of $H(n,3)$, then there exists $\sigma \in \Aut(H(n,3))$ such that $\sigma(U) = D_{n}$.
\end{Thm}

\begin{proof}
We prove this by induction on $n$. For $n = 1$ and $n = 2$, it can be checked directly that up to isomorphism, $D_n$ is the unique subset of $\Z_3^n$ of size at least $3^{n-1} + 1$ which has induced degree $1$.

For the inductive step, assume the result is true for $\Z_3^{n+1}$ and let $U$ be a subset of $\Z_3^{n+2}$ such that $|U| \geq 3^{n+1}+1$, $U$ has induced degree $1$, and $U$ is disjoint from a maximum size independent set of $H(n+2,3)$.

Writing $U = (U_0,0) \cup (U_1,1) \cup (U_2,2)$, at least one of $U_0$, $U_1$, and $U_2$ must have size larger than $3^{n}$. By applying an appropriate translation, we can assume that $|U_0| > 3^{n}$. Since $U_0$ is disjoint from an independent set, we can apply an automorphism of $H(n+1,3)$ so that $U_0 = D_{n+1}$. After we do this, we have that\\
\begin{center}
 $U=$
\begin{tabular}{|c|c|c|}
 \hline
 $D_n$ & $A_n$ & $A_n$ \\
 \hline
 $U_{01}$  & $U_{11}$ & $U_{21}$ \\
 \hline
$U_{02}$  & $U_{12}$ & $U_{22}$ \\
 \hline
\end{tabular}
\end{center}\ \\
for some subsets $U_{01}, U_{11}, U_{21}, U_{02}, U_{12}, U_{22}$ of $\Z_3^{n}$.
Following similar logic as before and swapping the second and third rows and/or the second and third columns if needed, we can assume that $U$ is disjoint from the independent set
\begin{center}
$A_{n+2} = $
\begin{tabular}{|c|c|c|}
 \hline
 $A_{n}$ & $C_{n}$ & $B_{n}$ \\
 \hline
 $C_{n}$  & $B_{n}$ & $A_{n}$ \\
 \hline
 $B_{n}$ & $A_{n}$ & $C_{n}$ \\
 \hline
\end{tabular} \ 
\end{center}
We now make the following observations:
\begin{enumerate}
\item Since $|U| \geq 3^{n+1} + 1$, $|U_{01}| + |U_{11}| + |U_{21}| + |U_{02}| + |U_{12}| + |U_{22}| \geq 6 \cdot 3^{n-1}$.
\item $|U_{01}| + |U_{02}| \leq 2 \cdot 3^{n-1}$ as otherwise the first column would be disjoint from an independent set and would have more than $3^{n}+1$ points, which would contradict the inductive hypothesis.
\item $U_{11} \subseteq C_n$ as $U_{11}$ is disjoint from both $A_n$ and $B_n$.
\item $U_{22} \subseteq B_n$ as $U_{22}$ is disjoint from both $A_n$ and $C_n$.
\item $|U_{12}| \leq 3^{n-1}+1$ and if $|U_{12}| = 3^{n-1}+1$ then $|U_{11}| \leq 3^{n-1} - \frac{|U_{12}|}{2}$ and $|U_{22}| \leq 3^{n-1} - \frac{|U_{12}|}{2}$. To see this, observe that if $|U_{12}| > 3^{n-1}$ then since $U_{12}$ is disjoint from $A_n$, $U_{12}$ is a canonical set. By the inductive hypothesis, $|U_{12}| = 3^{n-1}+1$ and every vertex of $U_{12}$ has degree $1$. This implies that $U_{12} \cap U_{11} = \phi$ and $U_{12} \cap U_{22} = \phi$. By Lemma \ref{lem:canonicalsetintersection}, $|U_{12} \cap C_n| = \frac{|U_{12}|}{2}$ so since $U_{11} \subseteq C_n$ and $U_{11} \cap U_{12} = \phi$, $|U_{11}| \leq 3^{n-1} - \frac{|U_{12}|}{2}$
\item Following similar logic, $|U_{21}| \leq 3^{n-1}+1$ and if $|U_{21}| = 3^{n-1}+1$ then $|U_{11}| \leq 3^{n-1} - \frac{|U_{21}|}{2}$ and $|U_{22}| \leq 3^{n-1} - \frac{|U_{21}|}{2}$.
\end{enumerate}
Combining these observations, there are two possibilities for the sizes of $U_{01}, U_{11},U_{21}, U_{02}, U_{12}, U_{22}$.
\begin{enumerate}
\item $|U_{01}| + |U_{02}| = 2 \cdot 3^{n-1}$ and $|U_{11}| = |U_{12}| = |U_{21}| = |U_{22}| = 3^{n-1}$.
\item $n = 1$, $|U_{01}| + |U_{02}| = 2$, $|U_{12}| = |U_{21}| = 2$, and $|U_{11}| = |U_{22}| = 0$.
\end{enumerate}
If $|U_{01}| + |U_{02}| = 2 \cdot 3^{n-1}$ and $|U_{11}| = |U_{12}| = |U_{21}| = |U_{22}| = 3^{n-1}$, we have that
\\
\begin{center}
 $U=$
\begin{tabular}{|c|c|c|}
 \hline
 $D_n$ & $A_n$ & $A_n$ \\
 \hline
 $U_{01}$  & $C_n$ & $U_{21}$ \\
 \hline
$U_{02}$  & $U_{12}$ & $B_n$ \\
 \hline
\end{tabular}
\end{center}\ \\
We make the following further observations:
\begin{enumerate}
\item $U_{12} = B_n$ or $U_{12} = C_n$. To see this, observe that $U_{12} \cap A_n = \phi$ and $U_{12}$ must be an independent set as for every vertex $u \in U_{12}$, either $u \in U_{11} = C_n$ or $u \in U_{22} = B_n$ so if $u$ had degree $1$ in $U_{12}$ then $U$ would have a degree $2$ vertex.
\item Following similar logic, $U_{21} = B_n$ or $U_{21} = C_n$.
\item $U_{01} = U_{02} = A_n$. To see this, observe that by the first two observations, $U_{12} = C_n$ and $U_{21} = B_n$ or $U_{12} = B_n$ and $U_{21} = C_n$. In either case, $U_{01}$ and $U_{02}$ must be disjoint from both $B_n$ and $C_n$ as $U$ is disjoint from $A_{n+2}$.
\end{enumerate}
Combining these observations, we have that \\
\begin{center}
 $U=$
\begin{tabular}{|c|c|c|}
 \hline
 $D_n$ & $A_n$ & $A_n$ \\
 \hline
 $A_n$  & $C_n$ & $B_n$ \\
 \hline
$A_n$  & $C_n$ & $B_n$ \\
 \hline
\end{tabular}
or 
\begin{tabular}{|c|c|c|}
 \hline
 $D_n$ & $A_n$ & $A_n$ \\
 \hline
 $A_n$  & $C_n$ & $C_n$ \\
 \hline
$A_n$  & $B_n$ & $B_n$ \\
 \hline
\end{tabular} \\
\end{center} 
The first case is $D_{n+2}$ and the second case can be transformed into $D_{n+2}$ by swapping coordinates $n+1$ and $n+2$ (which corresponds to swapping the rows and columns.)

If $n = 1$, $|U_{01}| + |U_{02}| = 2$, $|U_{12}| = |U_{21}| = 2$, and $|U_{11}| = |U_{22}| = 0$ then we must have that\\
\begin{center}
 $U=$
\begin{tabular}{|c|c|c|}
 \hline
 $X$ & $A$ & $A$ \\
 \hline
 $A$  & $ $ & $X$ \\
 \hline
$A$  & $X$ & $ $ \\
 \hline
\end{tabular} $\simeq$ \begin{tabular}{|c|c|c|}
 \hline
  & $\bullet$ & $\bullet$ \\
 \hline
 $\bullet$  &   &  \\
 \hline
 $\bullet$ &  &  \\
 \hline
\end{tabular} \ 
\begin{tabular}{|c|c|c|}
 \hline
 $\bullet$  &  &  \\
 \hline
   &   & $\bullet$ \\
 \hline
  & $\bullet$ &  \\
 \hline
\end{tabular} \ 
\begin{tabular}{|c|c|c|}
 \hline
 $\bullet$  &  &  \\
 \hline
   &   & $\bullet$ \\
 \hline
  & $\bullet$ &  \\
 \hline
\end{tabular}.\\
\end{center}
Note that the right hand side is obtained from the left hand side by replacing $A$ with $\{0\}$ and $X = \{1,2\}$ where the replacement is done in the third coordinate rather than the first coordinate. As shown by Figure \ref{fig:canonicalset}, the right hand side is $D_3$ so we have that $U \simeq D_3$, as needed.
\end{proof}

We now make some useful observations about the structure of canonical sets.
\begin{Def}[Popular direction]
    Let $U$ be a canonical subset $U$ of $\Z_3^n$. We define the \emph{popular direction} of $U$ to be the unique $i \in [n]$ such that there are exactly $3^{n-2}$ lines of the form $\{x, x+e_i, x+2e_i\}$ which contains two points of $U$.
\end{Def}
\begin{Def}[Extra point]
Let $U \subseteq \Z_3^n$ be a \1sat canonical subset $U$ of $\Z_3^n$. Writing $U = \cup_{x \in \Z_3^{n-1}}{(U_f(x) \times \{x\})}$, we define the \emph{extra point} $x_U$ of $U$ to be the unique $x_U \in \Z_3^{n-1}$ such that $U_f(x_U)$ is $X$, $Y$, or $Z$.
\end{Def}

\begin{Def}[Affine subsets]
We define an affine subset of $\Z_3^{n}$ to be a subset $H \subseteq \Z_3^{n}$ of the form
\[
H = \{x \in \Z_3^{n}: \text{For all } i \in R_H, x_i = c_i\}
\]
for some $R_H \subseteq [n]$ and elements $\{c_i: i \in R_H\}$ where each $c_i \in \Z_3$.

We say that an affine subset $H$ contains direction $e_i$ if $i \notin R_H$ (i.e., the value of $x_i$ is not restricted by $H$).  
\end{Def}

\begin{Cor}\label{cor:canonicalsetrestriction}
If $U$ is a \1sat canonical subset $U$ of $\Z_3^n$ then for all affine subsets $H \subseteq \Z_3^n$ containing the direction $e_1$, if $H$ contains the extra point of $U$ then $U \cap H$ is a canonical set for $H$.
\end{Cor}
\begin{proof}
We prove this statement by induction. The base case $n = 1$ is trivial. For the inductive step, assume that the result is true when $n = k$ and consider the case when $n = k+1$.

Since $U$ is a \1sat canonical set, there is a popular direction $i \in [k+1]$ for $U$ which is not $1$. Let $U_0 = \{x \in U: x_i = 0\}$, $U_1 = \{x \in U: x_i = 1\}$, and $U_2 = \{x \in U: x_i = 2\}$. Similarly, let $H_0 = \{x \in H: x_i = 0\}$, $H_1 = \{x \in H: x_i = 1\}$, and $H_2 = \{x \in H: x_i = 2\}$. We must have that when we ignore coordinate $i$, two of $U_0$, $U_1$, and $U_2$ are copies of the same independent set $I$ while the third is a canonical set which is disjoint from $I$. Without loss of generality, we can assume that when we ignore coordinate $i$, $U_1 = U_2 = I$ and $U_0$ is a canonical set which is disjoint from $I$. There are two cases to consider.
\begin{enumerate}
\item $H$ contains the direction $e_i$. In this case, $H_0$ contains the extra point of $U$ which is also the extra point of $U_0$. By the inductive hypothesis, $U_0 \cap H_0$ is a canonical set for $H_0$. When we ignore coordinate $i$, $U_1 = U_2 = I$ and $U_0 \cap I = \phi$ so we must have that $U_1 \cap H_1$ and $U_2 \cap H_2$ are two copies of the same independent set which are disjoint from $U_0 \cap H_0$. This implies that $U \cap H$ is a canonical set.
\item $H$ does not contain the direction $e_i$. In this case, since $H$ contains the extra point of $U$, we must have that $H = H_0$. Since $H = H_0$ contains the extra point of $U_0$ and $U_0$ is a canonical set, by the inductive hypothesis, $U \cap H = U_0 \cap H_0$ is a canonical set.
\end{enumerate}
Thus, in both cases we have that $U \cap H$ is a canonical set, as needed.
\end{proof}

\subsection{Canonical Paths}
For our analysis, it is useful to define canonical paths on $U_f$. We will show that if $U$ is a \1sat subset of $\Z_3^n$ which induces a subgraph of maximum degree $1$ then for each $x \in \Z_3^{n-1}$ such that $U_f(x) \in \{X,Y,Z\}$, there must be a large number of $x' \in \Z_3^{n-1}$ such that the canonical path starting at $x'$ ends at $x$. This implies that there cannot be too many $x \in \Z_3^{n-1}$ such that $U_f(x) \in \{X,Y,Z\}$.
\begin{Def}
Let $U$ be a \1sat subset of $\Z_3^n$ which induces a subgraph of maximum degree at most $1$ and write $U = \cup_{x \in \Z_3^{n-1}}{(U_f(x) \times \{x\})}$. For each $x \in \Z_3^{n-1}$, we define the canonical path $P_x$ starting at $x$ as follows.

If there is a direction $i \in [n-1]$ such that $U_f(x + e_i) = U_f(x)$ or $U_f(x + 2e_i) = U_f(x)$ then we do the following:
\begin{enumerate}
\item If $U_f(x + e_i) = U_f(x)$ then we take the point $y = x + 2e_i$ and then take the canonical path $P_y$ starting from $y$.
\item If $U_f(x + 2e_i) = U_f(x)$ then we take the point $y = x + e_i$ and then take the canonical path $P_y$ starting from $y$.
\end{enumerate}

If there is no direction $i \in [n-1]$ such that $U_f(x + e_i) = U_f(x)$ or $U_f(x + 2e_i) = U_f(x)$ then we end $P_x$ at $x$.
\end{Def}
A key observation for our upper bound in Section \ref{section:upper} is that on \1sat canonical sets $U$, all canonical paths end at the extra point of $U$.
\begin{Lemma}
For all $n \in \mathbb{N}$, writing $D_n = \{D_{n,f}(x) \times x : x \in \Z_3^{n-1}\}$, we have that for all $x \in \Z_3^{n-1}$, the canonical path $P_x$ starting at $x$ ends at the extra point $x_{D_n} = 0^{n-1}$ of $D_n$.
\end{Lemma}
\begin{proof}
    We prove this statement by induction. The base case $n = 1$ is trivial. For the inductive step, assume that the result holds for $D_n$ and consider $D_{n+1} = $
\begin{tabular}{|c|c|c|}
 \hline
  $D_{n}$ & $A_{n}$ & $A_{n}$ \\
 \hline
\end{tabular}. If $x \in D_n$ then the result follows from the inductive hypothesis. If $x$ is in one of the copies of $A_n$ then the next point $y$ will be in $D_n$ so the result follows from the inductive hypothesis.
\end{proof}
\begin{Eg}
Observe that 
\[
D_4 = \begin{tabular}{|c|c|c|}
    \hline
    $D_3$ & $A_3$ & $A_3$\\
    \hline
\end{tabular}
= 
\begin{tabular}{|c|c|c|}
 \hline
 $X$ & $A$ & $A$ \\
 \hline
 $A$ & $C$ & $B$ \\
 \hline
 $A$ & $C$ & $B$ \\
 \hline
\end{tabular} \ 
\begin{tabular}{|c|c|c|}
 \hline
 $A$ & $C$ & $B$ \\
 \hline
 $C$ & $B$ & $A$ \\
 \hline
 $B$ & $A$ & $C$ \\
 \hline
\end{tabular} \ 
\begin{tabular}{|c|c|c|}
 \hline
 $A$ & $C$ & $B$ \\
 \hline
 $C$ & $B$ & $A$ \\
 \hline
 $B$ & $A$ & $C$ \\
 \hline
\end{tabular} \ 
\]
If we start with the $A$ at $x = (2,1,2)$ (the middle right point of the right block), the next points of $P_x$ are the $B$ at $(2,1,0)$, the $A$ at $(2,0,0)$, and the $X$ at $(0,0,0)$, i.e., the middle right, upper right, and upper left points of the left block.
\end{Eg}

\begin{Def}
Given an affine subset $H$ of $\Z_3^n$ which contains direction $e_1$, we define $H_{red}$ to be the affine subset of $\Z_3^n$ such that $H = \Z_3 \times H_{red}$.
\end{Def}

\begin{Lemma}
\label{canonical_path}
Let $U$ be a \1sat subset of $\Z_3^{n}$ which has induced degree at most $1$ 
and let $H' = \Z_3 \times H'_{red}$ and $H'' = \Z_3 \times H''_{red}$ be two affine subsets of $\Z_3^{n}$ containing the direction $e_1$. Taking $U' = U \cap H'$ and $U'' = U \cap H''$ and writing $U' = \cup_{x \in H'_{red}}{(U'_f(x) \times \{x\})}$ and $U'' = \cup_{x \in H''_{red}}{(U''_f(x) \times \{x\})}$,
\begin{enumerate}
    \item If $U'$ is a canonical set for $H'$ then for all $x \in H'_{red}$, the canonical path $P_x$ starting at $x$ ends at the extra point $x_{U'}$ of $U'$.
    \item If $U'$ is a canonical set for $H'$, $U''$ is a canonical set for $H''$, and $H' \cap H'' \neq \phi$ then $x_{U'} = x_{U''}$ and $U' \cap U''$ is a canonical set for $H' \cap H''$.
\end{enumerate}
\end{Lemma}
\begin{proof}
We prove the first statement by induction. For the first statement, the base case $n=1$ is trivial. For the inductive step, assume that the first statement is true when $n \leq k$ and consider the case when $U \subseteq \Z_3^{k+1}$.

Observe that if $U'$ is a canonical set, since $U$ and thus $U'$ are \1sat, there is a popular direction $i \in [k+1]$ for $U'$ which is not $1$. Let $U'_0 = \{x \in U': x_i = 0\}$, $U'_1 = \{x \in U': x_i = 1\}$, and $U'_2 = \{x \in U': x_i = 2\}$. Observe that $U'_0 = U_0 \cap H'_0$, $U'_1 = U_1 \cap H'_1$, and $U'_2 = U_2 \cap H'_2$ where $U_0 = \{x \in U: x_i = 0\}$, $U_1 = \{x \in U: x_i = 1\}$, $U_2 = \{x \in U: x_i = 2\}$, $H'_0 = \{x \in H': x_i = 0\}$, $H'_1 = \{x \in H': x_i = 1\}$, and $H'_2 = \{x \in H': x_i = 2\}$.

We must have that when we ignore coordinate $i$, two of $U'_0$, $U'_1$, and $U'_2$ are copies of the same independent set $I$ while the third is a canonical set which is disjoint from $I$. Without loss of generality, we can assume that when we ignore coordinate $i$, $U'_1 = U'_2 = I$ and $U'_0$ is a canonical set which is disjoint from $I$. By the inductive hypothesis, $(H'_0)_{red}$ contains the extra point $x_{U'}$ of $U'$ and all canonical paths starting in $(H'_0)_{red}$ end at $x_{U'}$. Since all canonical paths starting in $(H'_1)_{red}$ or $(H'_2)_{red}$ reach $U'_0$ after their first step, all canonical paths starting in $H'_{red}$ must end at $x_{U'}$.

We now show the second statement. Observe that by the first statement, for all $x \in H'_{red}$, $P_x$ ends at $x_{U'}$. Similarly, for all $x \in H''_{red}$, $P_x$ ends at $x_{U''}$. Thus, we must have that $x_{U'} = x_{U''}$. Since $x_{U'} \in H'_{red}$ and $x_{U''} \in H''_{red}$, $x_{U'} = x_{U''} \in (H' \cap H'')_{red}$. By Corollary \ref{cor:canonicalsetrestriction}, $U' \cap U'' = U' \cap (H' \cap H'') = U'' \cap (H' \cap H'')$ is a canonical set for $H' \cap H''$.
\end{proof}
\begin{remark}
The fact that $U' \cap U''$ is a canonical set for $H' \cap H''$ depends on the fact that $U$ is \1sat and $H$ and $H'$ contain the direction $e_1$. For an example where two canonical sets have an intersection which is not a canonical set, see Appendix \ref{sec:strangeintersection}.
\end{remark}

\section{Size lower bound}
\label{section:lower}
Recall that $\alpha(H(n,3)) = 3^{n-1}$ is the size of the largest independent set of $H(n,3)$. For $n = 1$ and $2$, the largest induced degree 1 subset of $\Z_3^n$ has size $\alpha(H(n,3))+1$ and these subsets are unique up to isomorphism. For $n = 3$, it is not hard to show that the largest induced degree 1 subset still has size $\alpha(H(n,3))+1$ but there are two non-isomorphic subsets. For $n = 4$, it is possible to have an induced degree 1 subset of size strictly greater than $\alpha(H(n,3))+1$. The following set has size $\alpha(H(n,3))+2 = 29$.

 
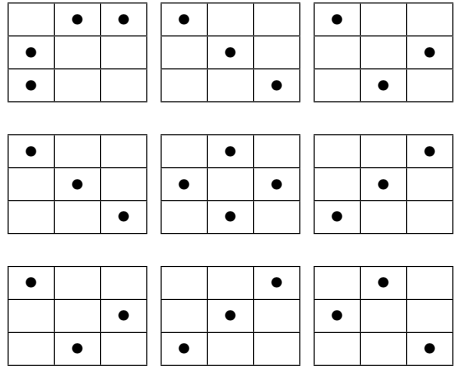
\[
    \begin{tabular}{c}
    
    \begin{tabular}{|c|c|c|}
        \hline
        \ \ \  & $\bullet$ & $\bullet$ \\
        \hline
        $\bullet$ & \ \ \  & \\
        \hline
        $\bullet$ &  & \ \ \ \\
        \hline
    \end{tabular} \ 
    \begin{tabular}{|c|c|c|}
        \hline
        $\bullet$ & \ \ \  &  \\
        \hline
         & $\bullet$ & \ \ \ \\
        \hline
        \ \ \  &  & $\bullet$ \\
        \hline
    \end{tabular} \ 
    \begin{tabular}{|c|c|c|}
        \hline
        $\bullet$ & \ \ \  &  \\
        \hline
        \ \ \  &  & $\bullet$ \\
        \hline
         & $\bullet$ & \ \ \ \\
        \hline
    \end{tabular} \\ 
    \\
    \begin{tabular}{|c|c|c|}
        \hline
        $\bullet$ & \ \ \  &  \\
        \hline
         & $\bullet$ & \ \ \ \\
        \hline
        \ \ \  &  & $\bullet$ \\
        \hline
    \end{tabular} \ 
    \begin{tabular}{|c|c|c|}
        \hline
        \ \ \  & $\bullet$ &  \\
        \hline
        $\bullet$ & \ \ \  & $\bullet$\\
        \hline
         & $\bullet$ & \ \ \ \\
        \hline
    \end{tabular} \ 
    \begin{tabular}{|c|c|c|}
        \hline
        \ \ \  &  & $\bullet$ \\
        \hline
         & $\bullet$ & \ \ \ \\
        \hline
        $\bullet$ & \ \ \  &  \\
        \hline
    \end{tabular} \\
    \\
    \begin{tabular}{|c|c|c|}
        \hline
        $\bullet$ & \ \ \  &  \\
        \hline
        \ \ \  &  & $\bullet$ \\
        \hline
         & $\bullet$ & \ \ \ \\
        \hline
    \end{tabular} \ 
    \begin{tabular}{|c|c|c|}
        \hline
        \ \ \  &  & $\bullet$ \\
        \hline
         & $\bullet$ & \ \ \ \\
        \hline
        $\bullet$ & \ \ \  &  \\
        \hline
    \end{tabular} \ 
    \begin{tabular}{|c|c|c|}
        \hline
         \ \ \  & $\bullet$  & \ \ \ \\
        \hline
         $\bullet$  & \ \ \ & \ \ \ \\
        \hline
         \ \ \  &  \ \ \  & $\bullet$ \\
        \hline
    \end{tabular}
    
    \end{tabular}
\]

It can be shown that not only is this set the largest possible subset of $\Z_3^4$ with induced degree 1, but it is also the only one up to isomorphism (see Appendix \ref{unique4}). This set demonstrates some interesting properties. Denote the set as $X_4$ and let $N(x)$ denote the set of neighbors of $x$: 
\begin{enumerate}
\item There exists a subset $O \subseteq \Z_3^4$ such that $\cup_{x \in O} N(x) \subseteq X_4$. Moreover, the only vertices with degree 1 in the induced subgraph are the vertices in $\cup_{x \in O} N(x)$. Note that for $X_4$, $O = \{(0,0,0,0),(1,1,1,1)\}$.
\item $X_4$ is $i$-saturated for all $i \in [4]$.
\end{enumerate}
Almost all of the extremal subsets for $n \le 4$ satisfy both properties, with $D_3$ the canonical set of dimension 3 being the only exception. $D_3$ satisfies neither of the two properties, but it satisfies a weaker form of property 2: it is $i$-saturated for some $i \in [3]$.

All the extremal subsets for $n \le 4$ are $i$-saturated for some $i$. By permutating the coordinates if needed, we can assume without loss of generality that they are $1$-saturated and hence we can represent them by functions $U_f : \Z_3^n \to \{A, B, C, X, Y, Z\}$ such that $U = \cup_{x \in \Z_3^{n-1}}{(U_f(x) \times \{x\})}$. The extremal subset $X_4$ for $\Z_3^4$ can be written as follows
\[
\begin{tabular}{|c|c|c|}
    \hline
    $X$ & $A$ & $A$\\
    \hline
    $A$ & $B$ & $C$\\
    \hline
    $A$ & $C$ & $B$\\
    \hline
\end{tabular} \ 
\begin{tabular}{|c|c|c|}
    \hline
    $A$ & $B$ & $C$\\
    \hline
    $B$ & $Y$ & $B$\\
    \hline
    $C$ & $B$ & $A$\\
    \hline
\end{tabular} \ 
\begin{tabular}{|c|c|c|}
    \hline
    $A$ & $C$ & $B$\\
    \hline
    $C$ & $B$ & $A$\\
    \hline
    $B$ & $A$ & $C$\\
    \hline
\end{tabular}.
\]
\subsection{Finding subsets via SAT solvers}
By Proposition \ref{ABCXYZcharacterization}, subsets $U$ with induced degree at most 1 can be characterized by a set of conditions on the adjacent blocks. If in addition $U$ is also 1-saturated, then it can be characterized as a solution to a certain CNF formula. Specifically, Proposition \ref{ABCXYZcharacterization} can be rephrased as the following fact regarding the function representation.

\begin{Prop}
    \label{constraints}
    Let $U_f : \Z_3^{n-1} \to \{A, B, C, X, Y, Z\}$ and $U = \cup_{x \in \Z_3^{n-1}}{(U_f(x) \times \{x\})}$. Then $U$ has induced degree 1 if and only if for all $x \in \Z_3^{n-1}$ and distinct $y, z \in N(x)$, $U_f$ satisfies the following constraints:
    \begin{enumerate}
        \item $U_f(x), U_f(y)$ and $U_f(z)$ are not all equal.

        \item If $U_f(x) = X$, then $U_f(y) = A$.

        \item If $U_f(x) = Y$, then $U_f(y) = B$.

        \item If $U_f(x) = Z$, then $U_f(y) = C$.
    \end{enumerate}
\end{Prop}
To construct our SAT instance, we create a variable $v_{x, E}$ for each $x \in \Z_3^{n-1}$ and $E \in \{A, B, C, X, Y, Z\}$ that indicates if $U_f(x) = E$. We have the following formula for each $x \in \Z_3^{n-1}$.
\[
Assign_x(v) = \Big(\bigvee_{E \in \{A, B, C, X, Y, Z\}}v_{x, E} \Big) \wedge \Big(\bigwedge_{E, F \in \{A, B, C, X, Y, Z\}, E \ne F} \neg v_{x,E} \vee \neg v_{x, F} \Big).
\]

For each $x \in \Z_3^{n-1}$ and distinct $y, z \in N(x)$, the first constraint in Proposition \ref{constraints} can be represented by the CNF
\[
NAE_{x, y, z}(v) = \bigwedge_{E \in \{A, B, C, X, Y, Z\}} (\neg v_{x, E} \vee \neg v_{y, E} \vee \neg v_{z, E}).
\]
The remaining constraints can be represented by 
\[
Disj_{x, y}(v) = (\neg v_{x, X} \vee v_{y, A}) \wedge (\neg v_{x, Y} \vee v_{y, B}) \wedge (\neg v_{x, Z} \vee v_{y, C}).
\]
Thus $U_f$ corresponds to a satisfiable assignment to
\[
\bigwedge_{x \in \Z_3^{n-1}, y, z \in N(x), y \ne z} Assign_x(v) \wedge NAE_{x, y, z}(v) \wedge Disj_{x, y}(v).
\]

Note that we do not try to minimize the size of the formula and there are redundant clauses.

\subsection{Examples of subsets with 6 and 18 extra points}
With the help of a SAT solver, we found an induced degree 1 subset in $\Z_3^5$ of size $\alpha(H(5, 3))+6$ and an induced degree 1 subset in $\Z_3^6$ of size $\alpha(H(6, 3))+18$. We illustrate them below in terms of $U_f$ using $A, B, C, X, Y, Z$ blocks. Note that each block with a value in $\{X, Y, Z\}$ (i.e., values in the the gray cells) contributes one additional point to the set. We say that a subset of $\Z_3^n$ has $K$ extra points if $|U|-\alpha(H(n, 3)) = K$. For \1sat subsets the number of extra points is precisely the number of $X, Y, Z$ blocks.\\

\noindent
\textbf{Example of 6 extra points in $\Z_3^5$.}
\[
\begin{tabular}{c}
    \begin{tabular}{|c|c|c|}
        \hline
        \cellcolor{lightgray}$X$ & $A$ & $A$\\
        \hline
        $A$ & $C$ & $B$\\
        \hline
        $A$ & $B$ & $C$\\
        \hline
    \end{tabular} \ 
    \begin{tabular}{|c|c|c|}
        \hline
        $A$ & $B$ & $C$\\
        \hline
        $B$ & $C$ & $A$\\
        \hline
        $C$ & $A$ & $B$\\
        \hline
    \end{tabular} \ 
    \begin{tabular}{|c|c|c|}
        \hline
        $A$ & $C$ & $B$\\
        \hline
        $C$ & \cellcolor{lightgray}$Z$ & $C$\\
        \hline
        $B$ & $C$ & $A$\\
        \hline
    \end{tabular}\\
    \\
    \begin{tabular}{|c|c|c|}
        \hline
        $A$ & $B$ & $C$\\
        \hline
        $C$ & $A$ & $B$\\
        \hline
        $B$ & $C$ & $A$\\
        \hline
    \end{tabular} \ 
    \begin{tabular}{|c|c|c|}
        \hline
        $B$ & \cellcolor{lightgray}$Y$ & $B$\\
        \hline
        $A$ & $B$ & $C$\\
        \hline
        $C$ & $B$ & $A$\\
        \hline
    \end{tabular} \ 
    \begin{tabular}{|c|c|c|}
        \hline
        $C$ & $B$ & $A$\\
        \hline
        $B$ & $C$ & $A$\\
        \hline
        $A$ & $A$ & \cellcolor{lightgray}$X$\\
        \hline
    \end{tabular}\\
    \\
    \begin{tabular}{|c|c|c|}
        \hline
        $A$ & $C$ & $B$\\
        \hline
        $B$ & $B$ & \cellcolor{lightgray}$Y$\\
        \hline
        $C$ & $A$ & $B$\\
        \hline
    \end{tabular} \ 
    \begin{tabular}{|c|c|c|}
        \hline
        $C$ & $B$ & $A$\\
        \hline
        $C$ & $A$ & $B$\\
        \hline
        \cellcolor{lightgray}$Z$ & $C$ & $C$\\
        \hline
    \end{tabular} \ 
    \begin{tabular}{|c|c|c|}
        \hline
        $B$ & $A$ & $C$\\
        \hline
        $A$ & $C$ & $B$\\
        \hline
        $C$ & $B$ & $A$\\
        \hline
    \end{tabular}
    
    \end{tabular}
\]

\noindent
\textbf{Example of 18 extra points in $\Z_3^6$.}
\[
    \begin{tabular}{c c}
    \begin{tabular}{c}
    \begin{tabular}{|c|c|c|}
        \hline
        \cellcolor{lightgray}$X$ & $A$ & $A$\\
        \hline
        $A$ & $C$ & $B$\\
        \hline
        $A$ & $B$ & $C$\\
        \hline
    \end{tabular} \ 
    \begin{tabular}{|c|c|c|}
        \hline
        $A$ & $B$ & $C$\\
        \hline
        $B$ & $C$ & $A$\\
        \hline
        $C$ & $A$ & $B$\\
        \hline
    \end{tabular} \ 
    \begin{tabular}{|c|c|c|}
        \hline
        $A$ & $C$ & $B$\\
        \hline
        $C$ & \cellcolor{lightgray}$Z$ & $C$\\
        \hline
        $B$ & $C$ & $A$\\
        \hline
    \end{tabular}\\
    \\
    \begin{tabular}{|c|c|c|}
        \hline
        $A$ & $B$ & $C$\\
        \hline
        $C$ & $A$ & $B$\\
        \hline
        $B$ & $C$ & $A$\\
        \hline
    \end{tabular} \ 
    \begin{tabular}{|c|c|c|}
        \hline
        $B$ & \cellcolor{lightgray}$Y$ & $B$\\
        \hline
        $A$ & $B$ & $C$\\
        \hline
        $C$ & $B$ & $A$\\
        \hline
    \end{tabular} \ 
    \begin{tabular}{|c|c|c|}
        \hline
        $C$ & $B$ & $A$\\
        \hline
        $B$ & $C$ & $A$\\
        \hline
        $A$ & $A$ & \cellcolor{lightgray}$X$\\
        \hline
    \end{tabular}\\
    \\
    \begin{tabular}{|c|c|c|}
        \hline
        $A$ & $C$ & $B$\\
        \hline
        $B$ & $B$ & \cellcolor{lightgray}$Y$\\
        \hline
        $C$ & $A$ & $B$\\
        \hline
    \end{tabular} \ 
    \begin{tabular}{|c|c|c|}
        \hline
        $C$ & $B$ & $A$\\
        \hline
        $C$ & $A$ & $B$\\
        \hline
        \cellcolor{lightgray}$Z$ & $C$ & $C$\\
        \hline
    \end{tabular} \ 
    \begin{tabular}{|c|c|c|}
        \hline
        $B$ & $A$ & $C$\\
        \hline
        $A$ & $C$ & $B$\\
        \hline
        $C$ & $B$ & $A$\\
        \hline
    \end{tabular}
    \end{tabular}
    &
    \begin{tabular}{c}
    \begin{tabular}{|c|c|c|}
        \hline
        $A$ & $B$ & $C$\\
        \hline
        $B$ & $A$ & $C$\\
        \hline
        $C$ & $C$ & \cellcolor{lightgray}$Z$\\
        \hline
    \end{tabular} \ 
    \begin{tabular}{|c|c|c|}
        \hline
        $B$ & $C$ & $A$\\
        \hline
        \cellcolor{lightgray}$Y$ & $B$ & $B$\\
        \hline
        $B$ & $A$ & $C$\\
        \hline
    \end{tabular} \ 
    \begin{tabular}{|c|c|c|}
        \hline
        $C$ & $A$ & $B$\\
        \hline
        $B$ & $C$ & $A$\\
        \hline
        $A$ & $B$ & $C$\\
        \hline
    \end{tabular}\\
    \\
    \begin{tabular}{|c|c|c|}
        \hline
        $C$ & $A$ & $B$\\
        \hline
        $A$ & \cellcolor{lightgray}$X$ & $A$\\
        \hline
        $B$ & $A$ & $C$\\
        \hline
    \end{tabular} \ 
    \begin{tabular}{|c|c|c|}
        \hline
        $C$ & $B$ & $A$\\
        \hline
        $B$ & $A$ & $C$\\
        \hline
        $A$ & $C$ & $B$\\
        \hline
    \end{tabular} \ 
    \begin{tabular}{|c|c|c|}
        \hline
        \cellcolor{lightgray}$Z$ & $C$ & $C$\\
        \hline
        $C$ & $A$ & $B$\\
        \hline
        $C$ & $B$ & $A$\\
        \hline
    \end{tabular}\\
    \\
    \begin{tabular}{|c|c|c|}
        \hline
        $B$ & $C$ & $A$\\
        \hline
        $C$ & $A$ & $B$\\
        \hline
        $A$ & $B$ & $C$\\
        \hline
    \end{tabular} \ 
    \begin{tabular}{|c|c|c|}
        \hline
        $A$ & $A$ & \cellcolor{lightgray}$X$\\
        \hline
        $B$ & $C$ & $A$\\
        \hline
        $C$ & $B$ & $A$\\
        \hline
    \end{tabular} \ 
    \begin{tabular}{|c|c|c|}
        \hline
        $C$ & $B$ & $A$\\
        \hline
        $A$ & $B$ & $C$\\
        \hline
        $B$ & \cellcolor{lightgray}$Y$ & $B$\\
        \hline
    \end{tabular}
    \end{tabular}
    \end{tabular}
\]
\[
    \begin{tabular}{c}
    \begin{tabular}{|c|c|c|}
        \hline
        $A$ & $C$ & $B$\\
        \hline
        $C$ & $B$ & $A$\\
        \hline
        $B$ & $A$ & $C$\\
        \hline
    \end{tabular} \ 
    \begin{tabular}{|c|c|c|}
        \hline
        $C$ & $A$ & $B$\\
        \hline
        $B$ & $A$ & $C$\\
        \hline
        $A$ & \cellcolor{lightgray}$X$ & $A$\\
        \hline
    \end{tabular} \ 
    \begin{tabular}{|c|c|c|}
        \hline
        $B$ & $B$ & \cellcolor{lightgray}$Y$\\
        \hline
        $A$ & $C$ & $B$\\
        \hline
        $C$ & $A$ & $B$\\
        \hline
    \end{tabular}\\
    \\
    \begin{tabular}{|c|c|c|}
        \hline
        $B$ & $C$ & $A$\\
        \hline
        $B$ & $A$ & $C$\\
        \hline
        \cellcolor{lightgray}$Y$ & $B$ & $B$\\
        \hline
    \end{tabular} \ 
    \begin{tabular}{|c|c|c|}
        \hline
        $A$ & $B$ & $C$\\
        \hline
        $C$ & $C$ & \cellcolor{lightgray}$Z$\\
        \hline
        $B$ & $A$ & $C$\\
        \hline
    \end{tabular} \ 
    \begin{tabular}{|c|c|c|}
        \hline
        $C$ & $A$ & $B$\\
        \hline
        $A$ & $B$ & $C$\\
        \hline
        $B$ & $C$ & $A$\\
        \hline
    \end{tabular}\\
    \\
    \begin{tabular}{|c|c|c|}
        \hline
        $C$ & \cellcolor{lightgray}$Z$ & $C$\\
        \hline
        $A$ & $C$ & $B$\\
        \hline
        $B$ & $C$ & $A$\\
        \hline
    \end{tabular} \ 
    \begin{tabular}{|c|c|c|}
        \hline
        $B$ & $C$ & $A$\\
        \hline
        $A$ & $B$ & $C$\\
        \hline
        $C$ & $A$ & $B$\\
        \hline
    \end{tabular} \ 
    \begin{tabular}{|c|c|c|}
        \hline
        $A$ & $C$ & $B$\\
        \hline
        \cellcolor{lightgray}$X$ & $A$ & $A$\\
        \hline
        $A$ & $B$ & $C$\\
        \hline
    \end{tabular}
    \end{tabular}
\]\\

Each \1sat induced degree 1 subset $U$ has a natural extension to higher dimensions that  preserves the number of extra points. This allows us to generate an example with 18 extra points for all $n \ge 6$.

\begin{Lemma}
    \label{lift}
    Let $U \subseteq \Z_3^n$ be a set corresponding to $U_f : \Z_3^{n-1} \to \{A, B, C, X, Y, Z\}$. If $U$ has induced degree 1 then there exists $V \subseteq \Z_3^{n+1}$ such that $|V|-\alpha(H(n+1,3)) = |U|-\alpha(H(n,3))$ and $V$ has induced degree 1.
\end{Lemma}

\begin{proof}
Define $V_f : \Z_3^{n} \to \{A, B, C, X, Y, Z\}$ as follows. For each $y \in \Z_3^{n-1}$,
\begin{enumerate}
    \item If $U_f(y) = X$ then $V_f(0, y) = X$, $V_f(1, y) = A$, and $V_f(2, y) = A$.
    
    \item If $U_f(y) = Y$ then $V_f(0, y) = Y$, $V_f(1, y) = B$, and $V_f(2, y) = B$.
    
    \item If $U_f(y) = Z$ then $V_f(0, y) = Z$, $V_f(1, y) = C$, and $V_f(2, y) = C$.
    
    \item If $U_f(y) = A$ then $V_f(0, y) = A$, $V_f(1, y) = C$, and $V_f(2, y) = B$.
    
    \item If $U_f(y) = B$ then $V_f(0, y) = B$, $V_f(1, y) = A$, and $V_f(2, y) = C$.
    
    \item If $U_f(y) = C$ then $V_f(0, y) = C$, $V_f(1, y) = B$, and $V_f(2, y) = A$.
\end{enumerate}
It is straightforward to verify that $V_f$ satisfies the constraints in Proposition \ref{constraints} and hence the corresponding subset $V = \cup_{x \in \Z_3^{n-1}}{(V_f(x) \times \{x\})}$ has induced degree 1. Since $|V|-\alpha(H(n+1,3)) = |\{x \in \Z_3^n : V_f(x) \in \{X, Y, Z\}\}|$ and $|\{x \in \Z_3^n : V_f(x) \in \{X, Y, Z\}\}| = |\{x \in \Z_3^{n-1} : U_f(x) \in \{X, Y, Z\}\}|$ by construction, the result follows.
\end{proof}

\begin{Cor}
    For all $n \ge 6$, there exists $U \subseteq \Z_3^n$ such that $U$ has induced degree 1 and $|U| = \alpha(H(n,3))+18$.
\end{Cor}

Note that our example of a subset of $\Z_3^5$ with 6 extra points contains a subset that is isomorphic to the extremal set $X_4$ in $\Z_3^4$, i.e., it is an extension of $X_4$. Similarly, our example of a subset of $\Z_3^6$ with 18 extra points is an extension of the example of a subset of $\Z_3^5$ with 6 extra points. However, the SAT solver determined that there is no \1sat subset of $\Z_3^7$ with induced degree 1 and size greater than $\alpha(H(n,7))+18$ which extends our example of a set in $\Z_3^6$ with 18 extra points.

\section{Size upper bound}
\label{section:upper}
As induced degree is non-increasing under restrictions, the extremal subset of $\Z_3^4$ we presented in the previous section implies an upper bound of $6$ and $18$ extra points for $\Z_3^5$ and $\Z_3^6$. Hence both examples we showed are the largest possible in the corresponding dimensions. Moreover, since the extremal subset of $\Z_3^4$ is unique up to isomorphism and it is $i$-saturated for all $i \in [4]$, we have

\begin{Prop}
    For $n \in [6]$, if $U \subseteq \Z_3^n$ has induced degree 1 and has the maximum size, then $U$ is $i$-saturated for some $i \in [n]$.
\end{Prop}

In this section we prove Theorem \ref{SAT_upper_bound} which we restate as follows.

\begin{Thm}
\label{constant_extra}
    Let $U \subseteq \Z_3^n$. If $U$ has induced degree 1 and $U$ is $1$-saturated, then $|U| \le 3^{n-1} + 81$.
\end{Thm}

Recall that \1sat subsets with induced degree 1 can be represented by functions $U_f : \Z_3^{n-1} \to \{A, B, C, X, Y, Z\}$ which satisfy the conditions in Proposition \ref{ABCXYZcharacterization}. We will use the function representation extensively in this section. In order to simplify the notation, we will consider functions with domain $\Z_3^{n}$ instead of $\Z_3^{n-1}$ and note that these functions correspond to subsets of $\Z_3^{n+1}$.

All subsets of $\Z_3^{n+1}$ we consider in this section are \1sat and it will be convenient to extend our definitions for subsets to their corresponding functions.

\begin{Def}[Canonical functions]
We say that a function $U_f : \Z_3^n \to \{A, B, C, X, Y, Z\}$ is \emph{canonical} if the corresponding subset is canonical.
\end{Def}
\begin{Def}[Induced degree]
    The \emph{induced degree} of a function $U_f : \Z_3^n \to \{A, B, C, X, Y, Z\}$ is defined as the maximum degree of the subgraph induced by the subset $U = \cup_{x \in \Z_3^{n}}{(U_f(x) \times \{x\})}$ corresponding to $U_f$.
\end{Def}

\begin{Def}
    We say that that two functions $U_f, V_f : \Z_3^n \to \{A, B, C, X, Y, Z\}$ are \emph{isomorphic} if their corresponding subsets are isomorphic, i.e., there exists $\sigma \in \Aut(H(n+1, 3))$ such that $\sigma(U) = V$ where $U, V$ are the corresponding subsets of $U_f$ and $V_f$.
\end{Def}

\subsection{Proof strategy for our size upper bound}
The main observation is that if $U_f$ has induced degree 1 and $U_f(x) \in \{X, Y, Z\}$, then not only does it determine the values of $U_f(y)$ for all $y \in N(x)$, it also imposes a strong restriction on what the values can be on $U_f(z)$ for $z$ where $d_H(x, z) = 2$. Specifically, we will show that for $n \ge 6$, if $U_f$ has induced degree 1 and $U_f(x) \in \{X, Y, Z\}$, then for all $i \in [n]$ there exists $j \in [n]\setminus \{i\}$ such that either
\begin{eqnarray*}
U_f(x+e_j) &=& U_f(x+2e_j),\\
U_f(x+e_i+e_j) &=& U_f(x+e_i+2e_j),\\
U_f(x+2e_i+e_j) &=& U_f(x+2e_i+2e_j).
\end{eqnarray*}
or
\begin{eqnarray*}
U_f(x+e_i) &=& U_f(x+2e_i),\\
U_f(x+e_j+e_i) &=& U_f(x+e_j+2e_i),\\
U_f(x+2e_j+e_i) &=& U_f(x+2e_j+2e_i).
\end{eqnarray*}

For illustration, suppose $U_f(x) = X$ and let $i \in [n]$, then there must be a different direction $j \in [n] \setminus \{i\}$ such that the function looks like one of the following:
\[
\begin{tabular}{c}
    \begin{tabular}{c c c}
         & $\overset{i}{\rightarrow}$ &
    \end{tabular}\\
    \begin{tabular}{c c c}
        $j \downarrow$ & \begin{tabular}{|c|c|c|} \hline $X$ & $A$ & $A$\\ \hline $A$ & $E$ & $F$\\ \hline $A$ & $E$ & $F$\\ \hline \end{tabular}  & $\cdots$
    \end{tabular}\\
    \begin{tabular}{c c c}
         & $\vdots$ &
    \end{tabular}
\end{tabular} \text{ or }
\begin{tabular}{c}
    \begin{tabular}{c c c}
         & $\overset{i}{\rightarrow}$ &
    \end{tabular}\\
    \begin{tabular}{c c c}
        $j \downarrow$ & \begin{tabular}{|c|c|c|} \hline $X$ & $A$ & $A$\\ \hline $A$ & $E$ & $E$\\ \hline $A$ & $F$ & $F$\\ \hline \end{tabular}  & $\cdots$
    \end{tabular}\\
    \begin{tabular}{c c c}
         & $\vdots$ &
    \end{tabular}
\end{tabular}
\]
where $E, F \in \{B, C\}$ and $E \ne F$. In either case, we find a canonical set containing $x$. We then show that it is possible to find a much larger canonical set by choosing different directions $i$. Specifically, we will show that there exists a $d$-dimensional canonical set containing $x$ where $d \ge n-4$. By constructing a large canonical set containing each $x$ for which $U_f(x) \in \{X, Y, Z\}$ and showing that canonical sets containing different $x$ are disjoint, we conclude that $|\{x \in \Z_3^n : U_f(x) \in \{X, Y, Z\}\}| \le 3^4 = 81$. Since the size of $U$ is precisely $3^{n} + |\{x \in \Z_3^n : U_f(x) \in \{X, Y, Z\}\}|$, the theorem follows.

\subsection{The line extension lemma}\label{sec:lineextensionlemma}

The first step of our proof is the following lemma.

\begin{Lemma}[Line extension lemma]
\label{SAT_line_extend}
Let $n \ge 6$ and $U_f : \Z_3^n \to \{A, B, C, X, Y, Z\}$. If $U_f$ has induced degree 1 and $U_f(0^n) = X$, then for all $i \in [n]$ there exists $j \in [n] \setminus \{i\}$ such that the restriction of $U_f$ to $\span(e_i, e_j)$ is canonical.
\end{Lemma}

The following definition will be useful.

\begin{Def}[$i$-skew functions]
We say that a function $U_f : \Z_3^n \to \{A, B, C, X, Y, Z\}$ is \emph{$i$-skew} if $U_f$ satisfies the following:
    \begin{enumerate}
        \item $U_f$ has induced degree 1. 
    
        \item $U_f(0^n) = X$.

        \item For all $j \in [n]\setminus\{i\}$, the restriction of $U_f$ to $\span(e_i, e_j)$ is not canonical.
    \end{enumerate}
\end{Def}

In other words, Lemma \ref{SAT_line_extend} asserts that if a function is $i$-skew for some $i \in [n]$, then $n \le 5$. Using ideas which are similar to the ideas discussed in Section \ref{section:lower}, deciding if an $i$-skew function exists in dimension $n$ can be formulated as a SAT formula. Using a SAT solver, we verified that there are no $i$-skew functions $U_f : \Z_3^6 \to \{A, B, C, X, Y, Z\}$. 

Note that this is the only place where we use a computer-assisted argument. In the rest of this subsection, we present a proof of Lemma \ref{SAT_line_extend} with $n \ge 8$ instead of $n \ge 6$ that does not require a SAT solver. More precisely, we prove the following:

\begin{Lemma}[Weak line extension lemma]
\label{line_extend}
Let $n \ge 8$ and $U_f : \Z_3^n \to \{A, B, C, X, Y, Z\}$. If $U_f$ has induced degree 1 and $U_f(0) = X$, then for all $i \in [n]$ there exists $j \in [n] \setminus \{i\}$ such that the restriction of $U_f$ to $\span(e_i, e_j)$ is canonical.
\end{Lemma}

To prove Lemma \ref{line_extend}, we give a complete characterization of $1$-skew functions for $n = 2, 3$ and use this characterization to prove a property of $1$-skew functions for $n = 4$ which is sufficient for proving Lemma \ref{line_extend}.

\subsubsection{$n = 2$}
Let $U_f : \Z_3^2 \to \{A, B, C, X, Y, Z\}$ and $R = \{(1, 1), (1, 2), (2, 1), (2, 2)\}$. If $U_f$ has induced degree 1 and $U_f(0, 0) = X$, then it is clear that $A \notin U_f(R)$ since otherwise $U_f$ will have induced degree at least 2, Similarly, $\{X, Y, Z\} \cap U_f(R) = \phi$ since $X, Y$ and $Z$ intersect with each other and each of them will force the neighbors to be their complements. Since any three elements in $R$ form a path of length 2, we cannot have three $B$ or three $C$ on $R$. It follows there are exactly two $B$ and exactly two $C$ on $R$. If the $B$ are adjacent then it is disjoint from a maximum size independent set and hence is canonical. Thus, the only 1-skew function is isomorphic to the following:

\begin{equation}
\label{D3}
\begin{tabular}{|c|c|c|}
    \hline
    $X$ & $A$ & $A$\\
    \hline
    $A$ & $B$ & $C$\\
    \hline
    $A$ & $C$ & $B$\\
    \hline
\end{tabular}
\end{equation}

The stabilizing actions for this function are generated by (i) swapping the 2nd and 3rd rows and then swapping $B$ and $C$, and (ii) swapping the row and column coordinates.

\subsubsection{$n = 3$}
\begin{Def}
Let $n > m$. We say that a function $U_f : \Z_3^n \to \{A, B, C, X, Y, Z\}$ \emph{extends} a function $V_f : \Z_3^m \to \{A, B, C, X, Y, Z\}$ if there exists a function $W_f$ that is isomorphic to $U_f$ and $W_f(x \times 0^{m-n}) = V_f(x)$ for all $x \in \Z_3^m$.
\end{Def}
We now show that up to isomorphism, there are precisely three 1-skew functions for $n = 3$. We list the $14$ distinct 1-skew functions for $n = 3$ which extend \eqref{D3} and have the first row of the second block equal to $(A, B, C)$ in Appendix \ref{sec:nequals3badfunctions}.
\begin{Thm}
    \label{bad3}
    For $n = 3$, there are 28 1-skew functions extending \eqref{D3}. Each of them is isomorphic to one of the following:
    \begin{enumerate}[(a)]
    \item 
    \begin{equation}
    \label{inde}
        \begin{tabular}{|c|c|c|}
            \hline
            $X$ & $A$ & $A$\\
            \hline
            $A$ & $B$ & $C$\\
            \hline
            $A$ & $C$ & $B$\\
            \hline
        \end{tabular} \ 
        \begin{tabular}{|c|c|c|}
            \hline
            $A$ & $B$ & $C$\\
            \hline
            $B$ & $C$ & $A$\\
            \hline
            $C$ & $A$ & $B$\\
            \hline
        \end{tabular} \ 
        \begin{tabular}{|c|c|c|}
            \hline
            $A$ & $C$ & $B$\\
            \hline
            $B$ & $A$ & $C$\\
            \hline
            $C$ & $B$ & $A$\\
            \hline
        \end{tabular}
    \end{equation}
    with stabilizing actions generated by swapping the 2nd and 3rd rows, swapping the 2nd and 3rd blocks, and then swapping $B$ and $C$.
    
    \item 
    \begin{equation}
    \label{As}
        \begin{tabular}{|c|c|c|}
            \hline
            $X$ & $A$ & $A$\\
            \hline
            $A$ & $B$ & $C$\\
            \hline
            $A$ & $C$ & $B$\\
            \hline
        \end{tabular} \ 
        \begin{tabular}{|c|c|c|}
            \hline
            $A$ & $B$ & $C$\\
            \hline
            $B$ & $C$ & $A$\\
            \hline
            $C$ & $A$ & $B$\\
            \hline
        \end{tabular} \ 
        \begin{tabular}{|c|c|c|}
            \hline
            $A$ & $C$ & $B$\\
            \hline
            $C$ & $B$ & $A$\\
            \hline
            $B$ & $A$ & $C$\\
            \hline
        \end{tabular}
    \end{equation}
    with stabilizing actions generated by (i) swapping the row and column coordinates; and (ii) swapping the 2nd and 3rd rows, swapping the 2nd and 3rd columns, and then swapping the 2nd and 3rd blocks.
    
    \item
    \begin{equation}
    \label{extra_point}
        \begin{tabular}{|c|c|c|}
            \hline
            $X$ & $A$ & $A$\\
            \hline
            $A$ & $B$ & $C$\\
            \hline
            $A$ & $C$ & $B$\\
            \hline
        \end{tabular} \ 
        \begin{tabular}{|c|c|c|}
            \hline
            $A$ & $B$ & $C$\\
            \hline
            $B$ & $C$ & $A$\\
            \hline
            $C$ & $A$ & $B$\\
            \hline
        \end{tabular} \ 
        \begin{tabular}{|c|c|c|}
            \hline
            $A$ & $C$ & $B$\\
            \hline
            $C$ & $A$ & $B$\\
            \hline
            $B$ & $B$ & $E$\\
            \hline
        \end{tabular}
    \end{equation}
    where $E \in \{A, C, Y\}$. The stabilizing actions are generated by (i) swapping the row and column coordinates; and (ii) swapping the row and block coordinates.
    \end{enumerate}
\end{Thm}
\begin{proof}
    Since the restriction of $U_f$ to $\span(e_1, e_2)$ is not canonical, by swapping $B$ and $C$ if necessary we can assume the restriction of $U_f$ to $\span(e_1, e_2)$ is identical to \eqref{D3}. Furthermore, by swapping the second and third block, we can assume the function is of the following form. Note that this also reduces the number of extensions by half. We consider the following partial function for which the values of the empty entries will be determined later.
     
    \[
    \begin{tabular}{|c|c|c|}
        \hline
        $X$ & $A$ & $A$\\
        \hline
        $A$ & $B$ & $C$\\
        \hline
        $A$ & $C$ & $B$\\
    \hline
    \end{tabular} \
    \begin{tabular}{|c|c|c|}
        \hline
        $A$ & $B$ & $C$\\
        \hline
        \ \ \ & \ \ \ & \ \ \ \\
        \hline
        \ \ \ & \ \ \ & \ \ \ \\
    \hline
    \end{tabular} \
    \begin{tabular}{|c|c|c|}
        \hline
        $A$ & $C$ & $B$\\
        \hline
        \ \ \ & \ \ \ & \ \ \ \\
        \hline
        \ \ \ & \ \ \ & \ \ \ \\
    \hline
    \end{tabular}
    \]

    We now consider all the possible values for the entries in the first column of the middle and last blocks which do not immediately create a vertex of degree at least $2$. Let $S = \{(0, 1, 1), (0, 1, 2), (0, 2, 1), (0, 2, 2)\}$ be the set of coordinates of these entries. We must have that $U_f(x) \in \{B, C\}$ for all $x \in S$ and exactly two of them are equal to $B$ as otherwise $U_f$ would have induced degree at least $2$.\\

    \begin{itemize}
    \item \textbf{Case 1: adjacent $B$ in $S$.}
    If $U_f(0, 1, 1) = U_f(0, 1, 2) = B$ (the $(0, 1)$ entries of the middle and third blocks), then the entries below them must be equal to $C$. This allows us to fill the remaining entries as follows. The subscripts in the table below indicate the order in which they are deduced.
    \[
    \begin{tabular}{|c|c|c|}
        \hline
        $X$ & $A$ & $A$\\
        \hline
        $A$ & $B$ & $C$\\
        \hline
        $A$ & $C$ & $B$\\
        \hline
    \end{tabular} \ 
    \begin{tabular}{|c|c|c|}
        \hline
        $A$ & $B$ & $C$\\
        \hline
        $B$ & $C_7$ & $A_1$\\
        \hline
        $C$ & $A_5$ & $B_3$\\
        \hline
    \end{tabular} \ 
    \begin{tabular}{|c|c|c|}
        \hline
        $A$ & $C$ & $B$\\
        \hline
        $B$ & $A_6$ & $C_4$\\
        \hline
        $C$ & $B_8$ & $A_2$\\
        \hline
    \end{tabular}.
    \]
    
    If the adjacent $B$ are in the first column of the middle block, i.e., $U_f(0, 1, 1) = U_f(0, 2, 1) = B$, then the two entries in the first column of the last block must be equal to $C$. By swapping the row coordinate with the block coordinate, it can be reduced to the case above and hence the remaining entries will be uniquely identified. The result is as follows.

    \[
    \begin{tabular}{|c|c|c|}
        \hline
        $X$ & $A$ & $A$\\
        \hline
        $A$ & $B$ & $C$\\
        \hline
        $A$ & $C$ & $B$\\
        \hline
    \end{tabular} \ 
    \begin{tabular}{|c|c|c|}
        \hline
        $A$ & $B$ & $C$\\
        \hline
        $B$ & $C$ & $A$\\
        \hline
        $B$ & $A$ & $C$\\
        \hline
    \end{tabular} \ 
    \begin{tabular}{|c|c|c|}
        \hline
        $A$ & $C$ & $B$\\
        \hline
        $C$ & $A$ & $B$\\
        \hline
        $C$ & $B$ & $A$\\
        \hline
    \end{tabular}.
    \]

    Similarly, we can uniquely deduce the functions for the remaining two cases and they are both isomorphic to one of the functions above.
    \[
    \begin{tabular}{|c|c|c|}
        \hline
        $X$ & $A$ & $A$\\
        \hline
        $A$ & $B$ & $C$\\
        \hline
        $A$ & $C$ & $B$\\
        \hline
    \end{tabular} \ 
    \begin{tabular}{|c|c|c|}
        \hline
        $A$ & $B$ & $C$\\
        \hline
        $C$ & $A$ & $B$\\
        \hline
        $B$ & $C$ & $A$\\
        \hline
    \end{tabular} \ 
    \begin{tabular}{|c|c|c|}
        \hline
        $A$ & $C$ & $B$\\
        \hline
        $C$ & $B$ & $A$\\
        \hline
        $B$ & $A$ & $C$\\
        \hline
    \end{tabular}.
    \]
    and
    \[
    \begin{tabular}{|c|c|c|}
        \hline
        $X$ & $A$ & $A$\\
        \hline
        $A$ & $B$ & $C$\\
        \hline
        $A$ & $C$ & $B$\\
        \hline
    \end{tabular} \ 
    \begin{tabular}{|c|c|c|}
        \hline
        $A$ & $B$ & $C$\\
        \hline
        $C$ & $A$ & $B$\\
        \hline
        $C$ & $B$ & $A$\\
        \hline
    \end{tabular} \ 
    \begin{tabular}{|c|c|c|}
        \hline
        $A$ & $C$ & $B$\\
        \hline
        $B$ & $C$ & $A$\\
        \hline
        $B$ & $A$ & $C$\\
        \hline
    \end{tabular}.
    \]
    The rese of the functions which are isomorphic to these cases can be generated by swapping the second and third block so there are a total of $8$ $1$-skew functions which are isomorphic to \eqref{inde}.

    \item \textbf{Case 2: no adjacent $B$ in $S$.}
    
    There are two cases for which there are no adjacent $B$ in the first column of the middle and last block. These two cases are isomorphic by swapping the 2nd and 3rd columns and then swapping $B$ and $C$. Thus, it suffices to consider the case that the $(0, 1)$ entry in the middle block is $B$.\\

    Consider the 6-cycle consisting of the cells in gray. 
    \[
    \begin{tabular}{|c|c|c|}
        \hline
        $X$ & $A$ & $A$\\
        \hline
        $A$ & $B$ & $C$\\
        \hline
        $A$ & $C$ & $B$\\
        \hline
    \end{tabular} \ 
    \begin{tabular}{|c|c|c|}
        \hline
        $A$ & $B$ & $C$\\
        \hline
        $B$ &  & \cellcolor{lightgray}\\
        \hline
        $C$ & \cellcolor{lightgray} & \cellcolor{lightgray} \\
        \hline
    \end{tabular} \ 
    \begin{tabular}{|c|c|c|}
        \hline
        $A$ & $C$ & $B$\\
        \hline
        $C$ & \cellcolor{lightgray}& \cellcolor{lightgray} \\
        \hline
        $B$ & \cellcolor{lightgray} & \\
        \hline
    \end{tabular}.
    \]

    None of these entries can be $C$. They also cannot be in $\{X, Y, Z\}$ and hence they are either $A$ or $B$. However, there cannot be a pair of adjacent $B$ in this cycle or otherwise the induced degree will be at least 2. It follows that this 6-cycle must either contain a pair of adjacent $A$s, or it consists of alternating $A$ and $B$.\\

    \begin{itemize}
    \item \textbf{Adjacent $A$ in the 6-cycle.}
    
    The values of all entries will be determined by where we place the adjacent $A$ and there are three possibilities. Suppose the adjacent $A$ are at the $(1, 2)$ entry of the middle and last block. Then the function will be as follows:

    \[
    \begin{tabular}{|c|c|c|}
        \hline
        $X$ & $A$ & $A$\\
        \hline
        $A$ & $B$ & $C$\\
        \hline
        $A$ & $C$ & $B$\\
        \hline
    \end{tabular} \ 
    \begin{tabular}{|c|c|c|}
        \hline
        $A$ & $B$ & $C$\\
        \hline
        $B$ &  & \cellcolor{lightgray}$A$\\
        \hline
        $C$ & \cellcolor{lightgray}$A$ & \cellcolor{lightgray}$B$\\
        \hline
    \end{tabular} \ 
    \begin{tabular}{|c|c|c|}
        \hline
        $A$ & $C$ & $B$\\
        \hline
        $C$ & \cellcolor{lightgray}$B$ & \cellcolor{lightgray}$A$\\
        \hline
        $B$ & \cellcolor{lightgray}$A$ & \\
        \hline
    \end{tabular},
    \]
    
    which further implies the remaining entries must be $C$:

    \[
    \begin{tabular}{|c|c|c|}
        \hline
        $X$ & $A$ & $A$\\
        \hline
        $A$ & $B$ & $C$\\
        \hline
        $A$ & $C$ & $B$\\
        \hline
    \end{tabular} \ 
    \begin{tabular}{|c|c|c|}
        \hline
        $A$ & $B$ & $C$\\
        \hline
        $B$ & $C$ & $A$\\
        \hline
        $C$ & $A$ & $B$\\
        \hline
    \end{tabular} \ 
    \begin{tabular}{|c|c|c|}
        \hline
        $A$ & $C$ & $B$\\
        \hline
        $C$ & $B$ & $A$\\
        \hline
        $B$ & $A$ & $C$\\
        \hline
    \end{tabular}.
    \]

    If the adjacent $A$ are in the last column of the middle block, then it is isomorphic to the case above by swapping the row and block coordinates. Moreover, it is also isomorphic to the case that the adjacent $A$ are in the last row of the middle block by swapping the row and column coordinates. Thus, all three cases are isomorphic to each other.

    The rest of the functions which are isomorphic to these cases can be generated by (i) swapping the 2nd and 3rd columns and then swapping $B$ and $C$ and (ii) swapping the 2nd and 3rd blocks. In total, there are $12$ 1-skew functions with which have adjacent $A$ in the corresponding 6-cycle which extend \eqref{D3} and all of them are isomorphic to \eqref{As}.\\
    
    \item \textbf{No adjacent $A$ in the 6-cycle.}
    It remains to consider the case where the $A$ and $B$ in the 6-cycle are alternating. Suppose the $(2,1)$ entry in the middle block is $A$, then we have the following partial function:

    \[
    \begin{tabular}{|c|c|c|}
        \hline
        $X$ & $A$ & $A$\\
        \hline
        $A$ & $B$ & $C$\\
        \hline
        $A$ & $C$ & $B$\\
        \hline
    \end{tabular} \ 
    \begin{tabular}{|c|c|c|}
        \hline
        $A$ & $B$ & $C$\\
        \hline
        $B$ &  & \cellcolor{lightgray}$A$\\
        \hline
        $C$ & \cellcolor{lightgray}$A$ & \cellcolor{lightgray}$B$\\
        \hline
    \end{tabular} \ 
    \begin{tabular}{|c|c|c|}
        \hline
        $A$ & $C$ & $B$\\
        \hline
        $C$ & \cellcolor{lightgray}$A$ & \cellcolor{lightgray}$B$\\
        \hline
        $B$ & \cellcolor{lightgray}$B$ & \\
        \hline
    \end{tabular}.
    \]
    
    The empty entry of the middle block must be $C$ but the empty entry of the last block can be $A, C$ or $Y$. So we must have
    \[
    \begin{tabular}{|c|c|c|}
        \hline
        $X$ & $A$ & $A$\\
        \hline
        $A$ & $B$ & $C$\\
        \hline
        $A$ & $C$ & $B$\\
        \hline
    \end{tabular} \ 
    \begin{tabular}{|c|c|c|}
        \hline
        $A$ & $B$ & $C$\\
        \hline
        $B$ & $C$ & $A$\\
        \hline
        $C$ & $A$ & $B$\\
        \hline
    \end{tabular} \ 
    \begin{tabular}{|c|c|c|}
        \hline
        $A$ & $C$ & $B$\\
        \hline
        $C$ & $A$ & $B$\\
        \hline
        $B$ & $B$ & $E$\\
        \hline
    \end{tabular}
    \]
    where $E \in \{A, C, Y\}$, which is precisely \eqref{extra_point}.

    If the $(2,1)$ entry in the middle block is $B$ then we obtain the following function which is isomorphic to the function above by swapping both the 2nd and 3rd rows, columns and then blocks.
    \[
    \begin{tabular}{|c|c|c|}
        \hline
        $X$ & $A$ & $A$\\
        \hline
        $A$ & $B$ & $C$\\
        \hline
        $A$ & $C$ & $B$\\
        \hline
    \end{tabular} \ 
    \begin{tabular}{|c|c|c|}
        \hline
        $A$ & $B$ & $C$\\
        \hline
        $B$ & $E$ & $B$\\
        \hline
        $C$ & $B$ & $A$\\
        \hline
    \end{tabular} \ 
    \begin{tabular}{|c|c|c|}
        \hline
        $A$ & $C$ & $B$\\
        \hline
        $C$ & $B$ & $A$\\
        \hline
        $B$ & $A$ & $C$\\
        \hline
    \end{tabular}.
    \]
     Similar to the case when there are adjacent $A$ in the $6$-cycle, the rest of the functions can be generated by (i) swapping the 2nd and 3rd columns and then swapping $B$ and $C$ (note that swapping $B$ and $C$ also swaps $Y$ and $Z$) and (ii) swapping the 2nd and 3rd blocks. In total, there are $8$ 1-skew functions with alternating $A$ in the corresponding 6-cycle which extend \eqref{D3} and all of them are isomorphic to \eqref{extra_point}.
    \end{itemize}
    \end{itemize}
    This covers all of the cases so the proof is complete.
\end{proof}

\subsubsection{$n = 4$}
In this case, there are a lot more 1-skew functions and we will not give a complete list of them. Instead, we prove the following lemma, which is sufficient for proving Lemma \ref{line_extend}.

\begin{Lemma}
\label{forbidden}
    Let $U_f : \Z_3^4 \to \{A, B, C, X, Y, Z\}$. If $U_f(x_1, x_2, x_3, 0)$ is identical to \eqref{inde} or \eqref{As} and $U_f(x_1, x_2, 0, x_4)$ is isomorphic to \eqref{inde} or \eqref{As} then $U_f$ has induced degree at least 2.
\end{Lemma}

The proof involves analyzing the 1-skew functions for $n = 3$ and is presented in Appendix \ref{sec:forbiddenlemmaproof}. 
\subsubsection{Proof of Lemma \ref{line_extend}}

We now prove Lemma \ref{line_extend} which says that if $U_f : \Z_3^n \to \{A, B, C, X, Y, Z\}$ is a 1-skew function then $n \le 7$. We prove this using a proof by contradiction.

Assume there exists a 1-skew function $U_f : \Z_3^n \to \{A, B, C, X, Y, Z\}$ for some $n \ge 8$. Observe that for all $i \in [3,n]$, deleting coordinate $i$ of a 1-skew function $U_f : \Z_3^n \to \{A, B, C, X, Y, Z\}$ gives a function that is also 1-skew so we can assume that $n = 8$. For $i \in [3,8]$, let $H_i = \{x \in \Z_3^8 : x_i =  0 \ \forall i \in [8] \setminus \{1, 2, i\}\}$. We claim that there are at least 5 $H_i$s for which the restriction of $U_f$ to $H_i$ is isomorphic to \eqref{extra_point}. Suppose not, then there are $H_i \ne H_j$ such that both the restriction of $U_f$ to $H_i$ and the restriction of $U_f$ to $H_j$ are not isomorphic to \eqref{extra_point}. However, by Lemma \ref{forbidden}, this implies that $U_f$ has induced degree at least $2$, contradicting the assumption that it has induced degree $1$.

Now, let $R = \{(1, 1), (1, 2), (2, 1), (2, 2)\} \times 0^{6}$. Since for each $i$ such that the restriction of $U_f$ to $H_i$ is isomorphic to \eqref{extra_point}, there exists a $x \in R$ and a neighbor $y \in H_i$ of $x$ such that $U_f(x) = U_f(y)$, by the pigeonhole principle there is a $x \in R$ and distinct neighbors $y, y'$ of $x$ such that $U_f(x) = U_f(y) = U_f(y')$, contradicting the assumption that $U_f$ has induced degree 1.

\subsection{Growing the canonical set}
We now use Lemma \ref{SAT_line_extend} to prove that for $n \ge 6$, if $U_f : \Z_3^n \to \{A, B, C, X, Y, Z\}$ has induced degree 1 and $U_f(0^n) = X$ then there is a canonical set of dimension at least $n-4$ containing $0^n$.

\begin{Thm}
\label{large_canonical}
    Let $n \ge 6$ and $U_f : \Z_3^n \to \{A, B, C, X, Y, Z\}$. If $U_f$ has induced degree 1 and $U_f(0^n) = X$, then there exists $I \subseteq [n]$ such that $|I| \ge n-4$ and the restriction of $U_f$ to $H = \{x \in \Z_3^n : x_i = 0 \ \forall i \notin I\}$ is canonical.
\end{Thm}
Before proving Theorem \ref{large_canonical}, we need a few preliminary results. First, by restricting $U_f$ to an affine subset, we obtain the following corollary of Lemma \ref{forbidden}.

\begin{Cor}
\label{line_extend_general}
Let $U_f : \Z_3^n \to \{A, B, C, X, Y, Z\}$ be a function with induced degree 1 such that $U_f(0^n) = X$. Let $I = \{i_1, \dots, i_d\} \subseteq [n]$. If $n-d \ge 5$, then for each $i \in I$, there exists $j \notin I$ such that the restriction of $U_f$ to $\span(e_i, e_j)$ is canonical.
\end{Cor}

\begin{proof}
    Without loss of generality, let $i = i_1$ and $H = \{x \in \Z_3^n : x_{i_2} = x_{i_3} = \cdots = x_{i_d} = 0\}$. The dimension of $H$ is $n-d+1 \ge 6$. Since $0^n \in H$ and the induced degree is non-increasing under restriction, by Lemma \ref{SAT_line_extend} there exists $j \in [n] \setminus I$ such that the restriction of $U_f$ to $\span(e_i, e_j)$ is canonical, as desired. 
\end{proof}

We will also need the following lemma.
\begin{Lemma}
    \label{ind_extend}
    Let $U_f : \Z_3^n \to \{A, B, C, X, Y, Z\}$ with induced degree 1. If there exists $i \in [n]$ and an affine subset $H$ such that the restriction of $U_f$ to $H$ is an independent set $I$ and $U_f(x) = U_f(x+e_i)$ for all $x \in H$ then for all $j \in [n] \setminus \{i\}$,
    \begin{enumerate}
        \item If there exists an $x \in H$ such that $U_f(x+e_j) = U_f(x+e_j+e_i)$ then the restriction of $U_f$ to $H + e_j$ is an independent set disjoint from $I$ and $U_f(y) = U_f(y+e_i)$ for all $y \in H + e_j$.

        \item If there exists an $x \in H$ such that $U_f(x+e_j) = U_f(x+2e_j)$ then the restriction of $U_f$ to $H + e_j$ is an independent set disjoint from $I$ and $U_f(y) = U_f(y+e_j)$ for all $y \in H + e_j$.
    \end{enumerate}
\end{Lemma}

\begin{proof}
    For item 1, suppose $H$ contains directions $i_1, \dots, i_d$ and let $K_m = \span(e_{i_1}, \dots, e_{i_m}) + x + e_j$. We show that $U_f(y) = U_f(y+e_i)$ for all $y \in K_m$ and $m \in [d]$ by induction on $m$.
    
    The base case $m = 0$ is trivial since $K_0 = \{x+e_j\}$ and $U_f(x+e_j) = U_f(x+e_j+e_i)$ by the assumption. For the inductive step, suppose $U_f(y) = U_f(y+e_i)$ for all $y \in K_k$ and consider $y \in K_{k+1}$. If $y \in K_k$, we are done. Otherwise, either $y+e_{k+1} \in K_k$ or $y+2e_{k+1} \in K_k$. 
    
    Suppose $y+e_{k+1} \in K_k$ (the case where $y+2e_{k+1} \in K_k$ is similar). By the inductive hypothesis, $U_f(y+e_{k+1}) = U_f(y+e_{k+1}+e_i)$. Since $U_f$ has induced degree 1, $U_f(y), U_f(y+2e_{k+1}) \ne U_f(y+e_{k+1})$. We have the following two cases:

    \begin{enumerate}
        \item If $U_f(y+e_{k+1}) \ne U_f(y-e_j)$ then  observe that $U_f(y+e_{k+1}) = U_f(y+e_{k+1}+e_i)$ and $U_f(y-e_j) = U_f(y-e_j+e_i)$ since $y-e_j \in H$. Thus, $U_f(y)$ and $U_f(y+e_i)$ must both be the element in $\{A, B, C\}$ which is not equal to $U_f(y+e_{k+1})$ or $U_f(y-e_j)$ so $U_f(y) = U_f(y+e_i)$.

        \item If $U_f(y+e_{k+1}) = U_f(y-e_j)$ then $U_f(y+e_{k+1}) \ne U_f(y+2e_{k+1}-e_j)$. Now observe that $U_f(y+e_{k+1}) = U_f(y+e_{k+1}+e_i)$ and $U_f(y + 2e_{k+1} - e_j) = U_f(y + 2e_{k+1} - e_j + e_i)$ since $y + 2e_{k+1} - e_j \in H$. Thus, $U_f(y+2e_{k+1}) = U_f(y+2e_{k+1}+e_i)$. Since $U_f(y+e_{k+1}) = U_f(y+e_{k+1}+e_i)$, this implies that $U_f(y) = U_f(y+e_i)$.
    \end{enumerate}
    
    For item 2, suppose $H$ contains directions $i_1, \dots, i_d$ and let $K_m = \span(e_{i_1}, \dots, e_{i_m}) + x + e_j$. We show that $U_f(y) = U_f(y+e_j)$ for all $y \in K_m$ and $m \in [d]$ by induction on $m$.

    The base case $m = 0$ is trivial since $K_0 = \{x+e_j\}$ and $U_f(x+e_j) = U_f(x+2e_j)$ by the assumption. For the inductive step, suppose $U_f(y) = U_f(y+e_i)$ for all $y \in K_k$ and consider $y \in K_{k+1}$. If $y \in K_k$, we are done. Otherwise, either $y+e_{k+1} \in K_k$ or $y+2e_{k+1} \in K_k$. 
    
    Suppose $y+e_{k+1} \in K_k$ (the case where $y+2e_{k+1} \in K_k$ is similar). By the inductive hypothesis, $U_f(y+e_{k+1}) = U_f(y+e_{k+1}+e_j)$. We have the following two cases:
    \begin{enumerate}
    \item If $U_f(y-e_j) \neq U_f(y+e_{k+1})$ then $U_f(y-e_j) = U_f(y-e_j+e_i) \neq U_f(y+e_{k+1}) = U_f(y+e_{k+1}+e_j)$. Thus, $U_f(y)$ and $U_f(y+e_j)$ must both be the element in $\{A, B, C\}$ which is not $U_f(y+e_{k+1})$ or $U_f(y-e_j)$ so $U_f(y) = U_f(y+e_j)$.
    \item If $U_f(y-e_j) = U_f(y+e_{k+1})$ then $U_f(y+e_{k+1} - e_j) = U_f(y+e_{k+1}-e_j+e_i) \neq U_f(y+e_{k+1}) = U_f(y+e_{k+1}+e_j)$ since $y + e_{k+1} - e_j \in H$. Thus, $U_f(y+e_{k+1}+e_i) = U_f(y+e_{k+1}+e_i+e_j)$. Now observe that $U_f(y-e_j) = U_f(y-e_j + e_i) \neq U_f(y+e_{k+1}+e_i) = U_f(y+e_{k+1}+e_i+e_j)$ as $U_f(y-e_j) = U_f(y+e_{k+1}) \neq U_f(y+e_{k+1}+e_i)$. This implies that $U_f(y+e_i) = U_f(y+e_i+e_j)$. Since $U_f(y-e_j) = U_f(y-e_j + e_i)$, this in turn implies that $U_f(y) = U_f(y+e_j)$, as needed.
    \end{enumerate} 
    In both cases, the conclusion holds so the proof is complete.
\end{proof}

We illustrate this lemma and its proof below. This lemma asserts that if there are two identical independent sets which are adjacent to each other in direction $i$ then if we consider the neighboring affine subsets in direction $j$, if there is a pair of identical elements in these neighboring affine subsets which are adjacent in direction $i$ or $j$ then this determines all of the entries of these neighboring affine subsets. Moreover, it forces these affine subsets to be identical independent sets which are adjacent in the same direction as the pair of identical elements.

In the figures below, the top two blocks are the two identical independent sets which are adjacent in direction $i$. In the figure on the left, the pair of $B$ in the first row of the middle blocks are an identical pair of elements which are adjacent in direction $i$. This pair determines the values in all of the gray cells and makes these affine subsets into identical independent sets which are adjacent in direction $i$. In the figure on the right, the pair of $B$ in the first column of the middle left and bottom left blocks are an identical pair of elements which are adjacent in direction $j$. This pair determines the values in all of the gray cells and makes these affine subsets into identical independent sets which are adjacent in direction $j$.

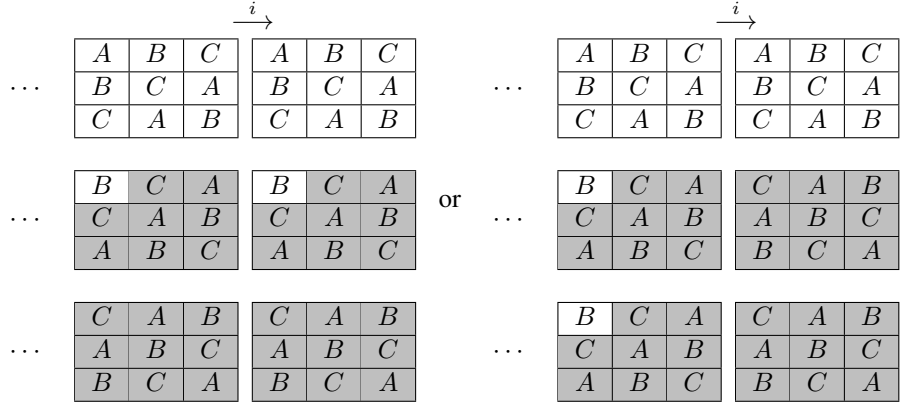
\[
    \begin{tabular}{c}
    \begin{tabular}{c c c c c c}
       & & & & $\overset{i}{\longrightarrow}$ &
    \end{tabular}\\
    \begin{tabular}{c}
        \\
        $\cdots$\\
        \\
    \end{tabular} \ 
    \begin{tabular}{|c|c|c|}
        \hline
        $A$ & $B$ & $C$\\
        \hline
        $B$ & $C$ & $A$\\
        \hline
        $C$ & $A$ & $B$\\
        \hline
    \end{tabular} \ 
    \begin{tabular}{|c|c|c|}
        \hline
        $A$ & $B$ & $C$\\
        \hline
        $B$ & $C$ & $A$\\
        \hline
        $C$ & $A$ & $B$\\
        \hline
    \end{tabular}\\
    \\
    \begin{tabular}{c}
        \\
        $\cdots$\\
        \\
    \end{tabular} \ 
    \begin{tabular}{|c|c|c|}
        \hline
        $B$ & \cellcolor{lightgray}$C$ & \cellcolor{lightgray}$A$\\
        \hline
        \cellcolor{lightgray}$C$ & \cellcolor{lightgray}$A$ & \cellcolor{lightgray}$B$\\
        \hline
        \cellcolor{lightgray}$A$ & \cellcolor{lightgray}$B$ & \cellcolor{lightgray}$C$\\
        \hline
    \end{tabular} \ 
    \begin{tabular}{|c|c|c|}
        \hline
        $B$ & \cellcolor{lightgray}$C$ & \cellcolor{lightgray}$A$\\
        \hline
        \cellcolor{lightgray}$C$ & \cellcolor{lightgray}$A$ & \cellcolor{lightgray}$B$\\
        \hline
        \cellcolor{lightgray}$A$ & \cellcolor{lightgray}$B$ & \cellcolor{lightgray}$C$\\
        \hline
    \end{tabular}\\
    \\
    \begin{tabular}{c}
        \\
        $\cdots$\\
        \\
    \end{tabular} \
    \begin{tabular}{|c|c|c|}
        \hline
        \cellcolor{lightgray}$C$ & \cellcolor{lightgray}$A$ & \cellcolor{lightgray}$B$\\
        \hline
        \cellcolor{lightgray}$A$ & \cellcolor{lightgray}$B$ & \cellcolor{lightgray}$C$\\
        \hline
        \cellcolor{lightgray}$B$ & \cellcolor{lightgray}$C$ & \cellcolor{lightgray}$A$\\
        \hline
    \end{tabular} \ 
    \begin{tabular}{|c|c|c|}
        \hline
        \cellcolor{lightgray}$C$ & \cellcolor{lightgray}$A$ & \cellcolor{lightgray}$B$\\
        \hline
        \cellcolor{lightgray}$A$ & \cellcolor{lightgray}$B$ & \cellcolor{lightgray}$C$\\
        \hline
        \cellcolor{lightgray}$B$ & \cellcolor{lightgray}$C$ & \cellcolor{lightgray}$A$\\
        \hline
    \end{tabular}
    \end{tabular} \text{ or}
    \begin{tabular}{c}
    \begin{tabular}{c c c c c c}
       & & & & $\overset{i}{\longrightarrow}$ &
    \end{tabular}\\
    \begin{tabular}{c}
        \\
        $\cdots$\\
        \\
    \end{tabular} \ 
    \begin{tabular}{|c|c|c|}
        \hline
        $A$ & $B$ & $C$\\
        \hline
        $B$ & $C$ & $A$\\
        \hline
        $C$ & $A$ & $B$\\
        \hline
    \end{tabular} \ 
    \begin{tabular}{|c|c|c|}
        \hline
        $A$ & $B$ & $C$\\
        \hline
        $B$ & $C$ & $A$\\
        \hline
        $C$ & $A$ & $B$\\
        \hline
    \end{tabular}\\
    \\
    \begin{tabular}{c}
        \\
        $\cdots$\\
        \\
    \end{tabular} \ 
    \begin{tabular}{|c|c|c|}
        \hline
        $B$ & \cellcolor{lightgray}$C$ & \cellcolor{lightgray}$A$\\
        \hline
        \cellcolor{lightgray}$C$ & \cellcolor{lightgray}$A$ & \cellcolor{lightgray}$B$\\
        \hline
        \cellcolor{lightgray}$A$ & \cellcolor{lightgray}$B$ & \cellcolor{lightgray}$C$\\
        \hline
    \end{tabular} \ 
    \begin{tabular}{|c|c|c|}
        \hline
        \cellcolor{lightgray}$C$ & \cellcolor{lightgray}$A$ & \cellcolor{lightgray}$B$\\
        \hline
        \cellcolor{lightgray}$A$ & \cellcolor{lightgray}$B$ & \cellcolor{lightgray}$C$\\
        \hline
        \cellcolor{lightgray}$B$ & \cellcolor{lightgray}$C$ & \cellcolor{lightgray}$A$\\
        \hline
    \end{tabular}\\
    \\
    \begin{tabular}{c}
        \\
        $\cdots$\\
        \\
    \end{tabular} \
    \begin{tabular}{|c|c|c|}
        \hline
        $B$ & \cellcolor{lightgray}$C$ & \cellcolor{lightgray}$A$\\
        \hline
        \cellcolor{lightgray}$C$ & \cellcolor{lightgray}$A$ & \cellcolor{lightgray}$B$\\
        \hline
        \cellcolor{lightgray}$A$ & \cellcolor{lightgray}$B$ & \cellcolor{lightgray}$C$\\
        \hline
    \end{tabular} \ 
    \begin{tabular}{|c|c|c|}
        \hline
        \cellcolor{lightgray}$C$ & \cellcolor{lightgray}$A$ & \cellcolor{lightgray}$B$\\
        \hline
        \cellcolor{lightgray}$A$ & \cellcolor{lightgray}$B$ & \cellcolor{lightgray}$C$\\
        \hline
        \cellcolor{lightgray}$B$ & \cellcolor{lightgray}$C$ & \cellcolor{lightgray}$A$\\
        \hline
    \end{tabular}
    \end{tabular}
\]
We now prove Theorem \ref{large_canonical}.

\begin{proof}[of Theorem \ref{large_canonical}]
We proceed by induction on $n$. For $n = 6$, Lemma \ref{SAT_line_extend} asserts that there exists $i \ne j$ such that $U_f$ is canonical on $\span(e_i, e_j)$.

Suppose the claim is true for functions on $\Z_3^k$ for some $k \ge 6$. Let $U_f$ be a function on $\Z_3^{k+1}$ that satisfies the assumption. Since the restriction of $U_f$ to $\{x \in \Z_3^{k+1} : x_{k+1} = 0\}\}$ is a function on $\Z_3^k$ and it also satisfies the assumption of the theorem, by the inductive hypothesis, there exists $I = \{i_1, \dots, i_d\} \subseteq [k]$ such that $d \ge k-4$ and the restriction of $U_f(x_1, \dots, x_k, 0)$ to $H = \{x \in \Z_3^{k} : x_i = 0 \ \forall i \notin I\}$ is canonical. If $d > k - 4$, we are done. However, if $d = k-4$, then $k+1 - d = 5$. Thus, by Corollary \ref{line_extend_general} for each $i \in I$, there exists $j \notin I$ such that the restriction of $U_f$ to $\span(e_i, e_j)$ is canonical. Let $i$ be the popular direction of the restriction of $U_f$ to $H$, and $j$ be the direction asserted by Corollary \ref{line_extend_general}. We claim that $U_f$ is canonical on $H' = \{x \in \Z_3^{k+1} : x_i = 0 \ \forall i \notin I'\}$ where $I' = I \cup \{j\}$.

Let $H'_{a, b} = H' \cap \{x \in \Z_3^{k+1} : x_i = a, x_j = b\}$ where $a, b \in \{0, 1, 2\}$. Since $U_f$ is canonical on $\span(e_i, e_j)$, we have either
\begin{eqnarray*}
U_f(e_i+e_j) &=& U_f(e_i+2e_j),\\
U_f(2e_i+e_j) &=& U_f(2e_i+2e_j).
\end{eqnarray*}
or
\begin{eqnarray*}
U_f(e_j+e_i) &=& U_f(e_j+2e_i),\\
U_f(2e_j+e_i) &=& U_f(2e_j+2e_i).
\end{eqnarray*}

Suppose the former case holds. By item 2 of Lemma \ref{ind_extend}, $U_f$ on $H'_{1,1}$ and $H'_{1,2}$ are a pair of identical independent sets. So $U$ is isomorphic to the following on $H'$:
\[
\begin{tabular}{c}
    \begin{tabular}{c c c}
         & $\overset{i}{\rightarrow}$ &
    \end{tabular}\\
    \begin{tabular}{c c c}
        $j \downarrow$ & \begin{tabular}{|c|c|c|} \hline $D_{d-1}$ & $A_{d-1}$ & $A_{d-1}$\\ \hline $E$ & $B_{d-1}$ & $C_{d-1}$\\ \hline $F$ & $B_{d-1}$ & $C_{d-1}$\\ \hline \end{tabular}  & \ \ \
    \end{tabular}
\end{tabular}
\]
where $B_{d-1}$ and $C_{d-1}$ are distinct independent sets of size $3^{d-2}$. Since $U$ has induced degree 1 and it is $1$-saturated, $E$ and $F$ must be disjoint from both $B_{d-1}$ and $C_{d-1}$ and have size at least $3^{d-2}$. It implies that $E = F = A_{d-1}$. Thus $U$ is canonical on $H'$.

For the latter case, by item 1 of Lemma \ref{ind_extend}, the restriction of $U$ to $H'$ is isomorphic to the following:
\[
\begin{tabular}{c}
    \begin{tabular}{c c c}
         & $\overset{i}{\rightarrow}$ &
    \end{tabular}\\
    \begin{tabular}{c c c}
        $j \downarrow$ & \begin{tabular}{|c|c|c|} \hline $D_{d-1}$ & $A_{d-1}$ & $A_{d-1}$\\ \hline $E$ & $B_{d-1}$ & $B_{d-1}$\\ \hline $F$ & $C_{d-1}$ & $C_{d-1}$\\ \hline \end{tabular}  & \ \ \
    \end{tabular}
\end{tabular}
\]
By Proposition \ref{unique_max_ind}, $E, F \in \{A_{d-1}, A_{d-1}'\}$. Since $E$ and $F$ are disjoint from $B_{d-1}$ and $C_{d-1}$ respectively, $E = A_{d-1}$ and $F = A_{d-1}$. Thus, $U_f$ is canonical on $H'$ and the proof is completed.
\end{proof}

We now prove Theorem \ref{constant_extra} which says that if $U \subseteq \Z_3^n$, $U$ has induced degree 1, and $U$ is $1$-saturated, then $|U| \le 3^{n-1} + 81$.

\begin{proof}[of Theorem \ref{constant_extra}]
    Let $U_f: \Z_3^{n-1} \to \{A, B, C, X, Y, Z\}$ be the function representation of $U$. For each $x \in \Z_3^{n-1}$ such that $U_f(x) \in \{X, Y, Z\}$, we can obtain an affine subset $H_x$ containing $x$ of dimension at least $n - 5$ on which $U_f$ is canonical by applying Theorem \ref{large_canonical} to the function $U_{f'}: \Z_3^{n-1} \to \{A, B, C, X, Y, Z\}$ where $f'(y) = f(y-x)$ (i.e., we apply a translation of $\Z_3^{n-1}$ to map $x$ to $0^{n-1}$ and then apply Theorem \ref{large_canonical}). We now observe that by Lemma \ref{canonical_path}, these affine subsets $H_x$ are disjoint for different $x$. To see this, observe that if $H_x \cap H_y \neq \phi$ for some $x,y \in \Z_3^{n-1}$ such that $U_f(x),U_f(y) \in \{X, Y, Z\}$ then by Lemma \ref{canonical_path}, the restrictions of $U_f$ to $H_x$ and $H_y$ have the same extra point so $x = y$.
    
    Since there can be at most $3^4 = 81$ disjoint affine subsets of $\Z_3^{n-1}$ of dimension at least $n - 5$, there are at most $81$ $x \in \Z_3^{n-1}$ such that $U_f(x) \in \{X, Y, Z\}$ which implies that $|U| \leq 3^{n-1} + 81$, as needed.
\end{proof}

\begin{remark}
    If we replace the use of Lemma \ref{SAT_line_extend} by Lemma \ref{line_extend} in this subsection, the same argument implies that there must be an affine subset $H_x$ of dimension at least $n-8$ on which $U_f$ is canonical which gives an upper bound of $3^{n-1} + 729$ rather than $3^{n-1} + 81$.
\end{remark}

\section{Conclusions and open problems}
In this paper, we analyzed the maximum size of a subset $U \subseteq \Z_3^{n}$ which induces a subgraph of the Hamming graph $H(n, 3)$ with maximum degree $1$. We showed that if $U$ is disjoint from a maximum size independent set then $|U| \leq \alpha(H(n, 3))+1$ but $U$ can be larger if $U$ is not disjoint from a maximum size independent set. In particular, for $n \geq 6$, there exists such a $U$ with size $\alpha(H(n, 3))+18$ and this is optimal when $n = 6$. We also showed that if $U$ is {\isat} for any $i \in [n]$ then $|U| \leq \alpha(H(n, 3)) + 729$.

The assumption of being $i$-saturated for some $i$ was motivated by the fact that it is a common property shared by all extremal subsets of $\Z_3^n$ for $n \in [6]$ and works well with SAT-solvers. We conjecture that similar results hold if we remove the assumption of being $i$-saturated for some $i \in [n]$ but this remains to be proven.
\begin{Conj}
    All induced degree 1 subsets of $\Z_3^n$ have size $\alpha(H(n,3))+O(1)$.
\end{Conj}
We can also ask what happens if we consider subsets of $\Z_3^n$ with larger induced degree.
\begin{Question}
    Given $d,n \in \mathbb{N}$, what is the largest subset of $\Z_3^n$ with induced degree at most $d$?
\end{Question}
We observe that there is a nice construction which has at least $3^{\lfloor\frac{(d-1)n}{d}\rfloor}$ extra points. That said, it is possible that there are larger constructions.
\begin{Lemma}
For all $d,n \in \mathbb{N}$, there is a subset $U$ of $\Z_3^n$ with induced degree at most $d$ such that $|U| \ge 3^{n-1} + 3^{\lfloor\frac{(d-1)n}{d}\rfloor}$ and $U$ is disjoint from a maximum size independent set of $H(n,3)$.
\end{Lemma}
\begin{proof}
We prove this lemma by induction. If $n \leq d$, we can take $U = \Z_3^n \setminus A_n$ and we will have that $|U| = 3^{n-1} + 3^{n-1} \ge 3^{n-1} + 3^{\lfloor\frac{(d-1)n}{d}\rfloor}$ and $U$ is disjoint from $A_n$.

If $n > d$ then by the inductive hypothesis, there is a subset $U_{n-d} \subseteq \Z_3^{n-d}$ of size at least $3^{n-1} + 3^{\lfloor\frac{(d-1)n}{d}\rfloor - (d-1)}$ which is disjoint from $A_{n-d}$. We can now do the following:
\begin{enumerate}
\item Start with the independent set $A_d$.
\item Replace each of the $3^{d-1}$ points in $A_d$ with a copy of $U_{n-d}$ and replace each point which is not in $A_d$ with a copy of $A_{n-d}$.
\end{enumerate}
It is not hard to verify that this subset has size at least $3^{n-1} + 3^{\lfloor\frac{(d-1)n}{d}\rfloor}$ and is disjoint from the independent set $A_n$.

We illustrate this construction for $d = 2$ below. When $d = 2$, we have that \\
$U_2 = 
    \begin{tabular}{|c|c|c|}
        \hline
        \ \ \  & $\bullet$ & $\bullet$ \\
        \hline
        $\bullet$ & $\bullet$  & \\
        \hline
        $\bullet$ &  & $\bullet$ \\
        \hline
    \end{tabular}$ and for $k \geq 1$, 
$U_{2(k+1)} = 
    \begin{tabular}{|c|c|c|}
        \hline
        $U_{2k}$ & $A_{2k}$ & $A_{2k}$ \\
        \hline
        $A_{2k}$ & $A_{2k}$ & $U_{2k}$ \\
        \hline
        $A_{2k}$ & $U_{2k}$ & $A_{2k}$ \\
        \hline
    \end{tabular}$.
\end{proof}

\section*{Acknowledgement}
We would like to thank Andrew Drucker for helpful discussions and the anonymous reviewer for comments and suggestions which helped improve the presentation of our paper. This research was supported by NSF grant CCF: 2008920.

\bibliographystyle{acm}
\bibliography{bibliography.bib}

\appendix
\newpage
\section{Uniqueness proof for $n = 4$}
\label{unique4}
In this section, we prove the following theorem.
\begin{Thm}
Up to isomorphism, there is a unique set $X \subseteq \mathbb{Z}_3^{4}$, $X$ has maximum induced degree $1$, and $|X| = 3^3 + 2$.
\end{Thm}
\begin{proof}
We say that two distinct affine subsets $H$ and $H'$ are \emph{parallel} if $H' = H \pm e_j$ for some $j \in [n]$.

\begin{Prop}\label{prop:noadjacentfours}
For any $n \geq 3$ and any set $X \subseteq \mathbb{Z}_3^{n}$, if there are two parallel affine subsets of dimension $2$ (i.e., $3 \times 3$ blocks) which each contain at least $4$ points of $X$ then the maximum induced degree of $X$ is at least $2$.
\end{Prop}
\begin{proof}
Assume that $X \subseteq \mathbb{Z}_3^{n}$, there are two parallel affine subsets of dimension $2$ (i.e., $3 \times 3$ blocks) which each contain at least $4$ points of $X$, and the maximum induced degree of $X$ is at most $1$. 

Up to isomorphism, there is only one subset of $\mathbb{Z}_3^2$ of size $4$ with maximum induced degree at most $1$ so without loss of generality we can assume that one of the affine subsets is as follows.
\[
    \begin{tabular}{|c|c|c|}
        \hline
          & $\bullet$ & $\bullet$\\
        \hline
        $\bullet$ &  & \\
        \hline
        $\bullet$ &  & \\
        \hline
    \end{tabular}
\]
Since the maximum induced degree of $X$ is $1$, in the second parallel affine subset, $X$ cannot contain any of the points shown in gray.
\[
    \begin{tabular}{|c|c|c|}
        \hline
          & \cellcolor{lightgray} \  & \cellcolor{lightgray} \ \\
        \hline
        \cellcolor{lightgray} \  &  & \\
        \hline
        \cellcolor{lightgray} \  &  & \\
        \hline
    \end{tabular}
\]
Since $X$ can contain at most three of the remaining points, $X$ cannot contain $4$ points of this affine subset which gives a contradiction. 
\end{proof}

We now consider the possible ways for the points of $X$ to be divided up when we split $\mathbb{Z}_3^4$ into nine $3 \times 3$ blocks. By Proposition \ref{prop:noadjacentfours}, if $X$ has maximum induced degree $1$ and $|X| = 3^3 + 2 = 29$ then up to isomorphism, the only possibilities are as follows.
\[
    \begin{tabular}{|c|c|c|}
        \hline
        $4$ & $3$ & $3$\\
        \hline
        $3$ & $4$ & $3$\\
        \hline
        $3$ & $3$ & $3$\\
        \hline
    \end{tabular} \text{ or }     \begin{tabular}{|c|c|c|}
        \hline
        $4$ & $3$ & $3$\\
        \hline
        $3$ & $4$ & $3$\\
        \hline
        $2$ & $3$ & $4$\\
        \hline
    \end{tabular}
\]
If we consider the possibilities for two of the $3 \times 3$ blocks with four points in $X$ then up to isomorphism, there are three possibilities.

First, these $3 \times 3$ blocks may be the same. In this case, if we look at the six neighboring blocks, $X$ cannot contain any of the points shown in gray.
    
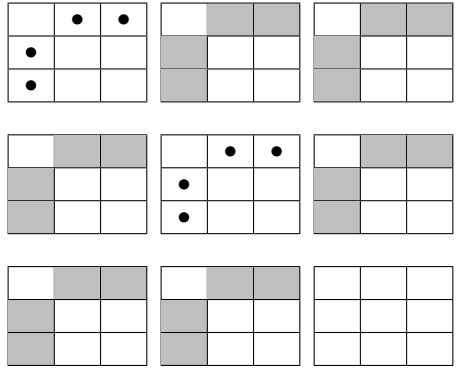
\[
    \begin{tabular}{c}
    
    \begin{tabular}{|c|c|c|}
        \hline
        \ \ \  & $\bullet$ & $\bullet$ \\
        \hline
        $\bullet$ & \ \ \  & \\
        \hline
        $\bullet$ &  & \ \ \ \\
        \hline
    \end{tabular} \ 
    \begin{tabular}{|c|c|c|}
        \hline
         \ \ \  & \cellcolor{lightgray} & \cellcolor{lightgray} \\
        \hline
        \cellcolor{lightgray} & \ \ \  & \\
        \hline
        \cellcolor{lightgray} &  & \ \ \ \\
        \hline
    \end{tabular} \ 
    \begin{tabular}{|c|c|c|}
        \hline
         \ \ \  & \cellcolor{lightgray} & \cellcolor{lightgray} \\
        \hline
        \cellcolor{lightgray} & \ \ \  & \\
        \hline
        \cellcolor{lightgray} &  & \ \ \ \\
        \hline
    \end{tabular} \\ 
    \\
    \begin{tabular}{|c|c|c|}
        \hline
         \ \ \  & \cellcolor{lightgray} & \cellcolor{lightgray} \\
        \hline
        \cellcolor{lightgray} & \ \ \  & \\
        \hline
        \cellcolor{lightgray} &  & \ \ \ \\
        \hline
    \end{tabular} \ 
    \begin{tabular}{|c|c|c|}
        \hline
        \ \ \  & $\bullet$ & $\bullet$ \\
        \hline
        $\bullet$ & \ \ \  & \\
        \hline
        $\bullet$ &  & \ \ \ \\
        \hline
    \end{tabular} \ 
    \begin{tabular}{|c|c|c|}
        \hline
         \ \ \  & \cellcolor{lightgray} & \cellcolor{lightgray} \\
        \hline
        \cellcolor{lightgray} & \ \ \  & \\
        \hline
        \cellcolor{lightgray} &  & \ \ \ \\
        \hline
    \end{tabular}\\
    \\
    \begin{tabular}{|c|c|c|}
        \hline
         \ \ \  & \cellcolor{lightgray} & \cellcolor{lightgray} \\
        \hline
        \cellcolor{lightgray} & \ \ \  & \\
        \hline
        \cellcolor{lightgray} &  & \ \ \ \\
        \hline
    \end{tabular} \ 
    \begin{tabular}{|c|c|c|}
        \hline
         \ \ \  & \cellcolor{lightgray} & \cellcolor{lightgray} \\
        \hline
        \cellcolor{lightgray} & \ \ \  & \\
        \hline
        \cellcolor{lightgray} &  & \ \ \ \\
        \hline
    \end{tabular} \ 
    \begin{tabular}{|c|c|c|}
        \hline
         \ \ \  &  &  \\
        \hline
         & \ \ \  & \\
        \hline
         &  & \ \ \ \\
        \hline
    \end{tabular}
    
    \end{tabular}
\]
Observe that $X$ can contain at most $4$ of the six upper left corners so at least two of these neighboring $3 \times 3$ blocks can only contain $2$ points of $X$. This is impossible as all but one of these $3 \times 3$ blocks must have at least $3$ points of $X$.

The second possibility is that one of the $3 \times 3$ blocks is obtained from the other by either swapping the row containing two points of $X$ with a row containing one point of $X$ or swapping the column containing two points of $X$ with a column containing one point of $X$ (but not both). In this case, if we look at the two blocks which neighbor both of these blocks, the points shown in gray cannot be in $X$ so both of these blocks contain at most $2$ points of $X$. Again, this gives a contradiction.

    
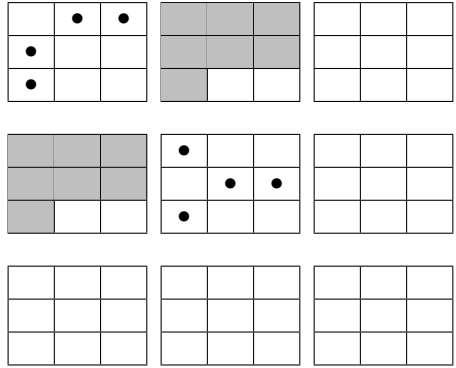
\[
    \begin{tabular}{c}
    
    \begin{tabular}{|c|c|c|}
        \hline
        \ \ \  & $\bullet$ & $\bullet$ \\
        \hline
        $\bullet$ & \ \ \  & \\
        \hline
        $\bullet$ &  & \ \ \ \\
        \hline
    \end{tabular} \ 
    \begin{tabular}{|c|c|c|}
        \hline
         \cellcolor{lightgray} \ \  & \cellcolor{lightgray} & \cellcolor{lightgray} \\
        \hline
        \cellcolor{lightgray} & \cellcolor{lightgray} & \cellcolor{lightgray} \\
        \hline
        \cellcolor{lightgray} &  \ \ \  & \ \ \ \\
        \hline
    \end{tabular} \ 
    \begin{tabular}{|c|c|c|}
        \hline
         \ \ \  &  &  \\
        \hline
         & \ \ \  & \\
        \hline
         &  & \ \ \ \\
        \hline
    \end{tabular} \\ 
    \\
    \begin{tabular}{|c|c|c|}
        \hline
         \cellcolor{lightgray} \ \  & \cellcolor{lightgray} & \cellcolor{lightgray} \\
        \hline
        \cellcolor{lightgray} & \cellcolor{lightgray} & \cellcolor{lightgray} \\
        \hline
        \cellcolor{lightgray} &  \ \ \  & \ \ \ \\
        \hline
    \end{tabular} \ 
    \begin{tabular}{|c|c|c|}
        \hline
        $\bullet$ & \ \ \  & \\
        \hline
        \ \ \  & $\bullet$ & $\bullet$ \\
        \hline
        $\bullet$ &  & \ \ \ \\
        \hline
    \end{tabular} \ 
    \begin{tabular}{|c|c|c|}
        \hline
         \ \ \  &  &  \\
        \hline
         & \ \ \  & \\
        \hline
         &  & \ \ \ \\
        \hline
    \end{tabular}\\
    \\
    \begin{tabular}{|c|c|c|}
        \hline
         \ \ \  &  &  \\
        \hline
         & \ \ \  & \\
        \hline
         &  & \ \ \ \\
        \hline
    \end{tabular} \ 
    \begin{tabular}{|c|c|c|}
        \hline
         \ \ \  &  &  \\
        \hline
         & \ \ \  & \\
        \hline
         &  & \ \ \ \\
        \hline
    \end{tabular} \ 
    \begin{tabular}{|c|c|c|}
        \hline
         \ \ \  &  &  \\
        \hline
         & \ \ \  & \\
        \hline
         &  & \ \ \ \\
        \hline
    \end{tabular}
    
    \end{tabular}
\]
The third possibility is that the one of the $3 \times 3$ blocks is obtained from the other by swapping the row containing two points of $X$ with a row containing one point of $X$ and swapping the column containing two points of $X$ with a column containing one point of $X$.

We now show that there is no way for the points of $X$ to have the division     \begin{tabular}{|c|c|c|}
        \hline
        $4$ & $3$ & $3$\\
        \hline
        $3$ & $4$ & $3$\\
        \hline
        $2$ & $3$ & $4$\\
        \hline
    \end{tabular}
and up to isomorphism, there is a unique way to have the division 
    \begin{tabular}{|c|c|c|}
        \hline
        $4$ & $3$ & $3$\\
        \hline
        $3$ & $4$ & $3$\\
        \hline
        $3$ & $3$ & $3$\\
        \hline
    \end{tabular}. To see that there is no way to have the division \begin{tabular}{|c|c|c|}
        \hline
        $4$ & $3$ & $3$\\
        \hline
        $3$ & $4$ & $3$\\
        \hline
        $2$ & $3$ & $4$\\
        \hline
    \end{tabular}, observe that if each pair of $3 \times 3$ blocks with $4$ points of $X$ satisfies the third possibility then up to ismorphism, these points must be arranged as follows.
    
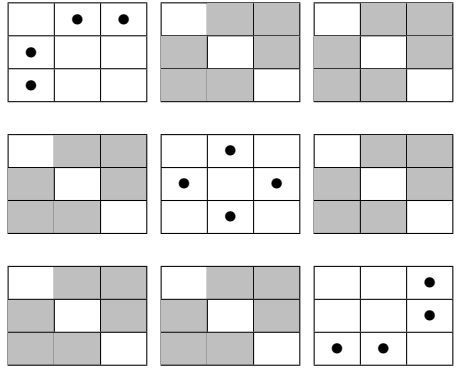
\[
    \begin{tabular}{c}
    
    \begin{tabular}{|c|c|c|}
        \hline
        \ \ \  & $\bullet$ & $\bullet$ \\
        \hline
        $\bullet$ & \ \ \  & \\
        \hline
        $\bullet$ &  & \ \ \ \\
        \hline
    \end{tabular} \ 
    \begin{tabular}{|c|c|c|}
        \hline
         \ \ \  & \cellcolor{lightgray} & \cellcolor{lightgray} \\
        \hline
        \cellcolor{lightgray} & \ \ \  & \cellcolor{lightgray} \\
        \hline
        \cellcolor{lightgray} & \cellcolor{lightgray} & \ \ \ \\
        \hline
    \end{tabular} \ 
    \begin{tabular}{|c|c|c|}
        \hline
         \ \ \  & \cellcolor{lightgray} & \cellcolor{lightgray} \\
        \hline
        \cellcolor{lightgray} & \ \ \  & \cellcolor{lightgray} \\
        \hline
        \cellcolor{lightgray} & \cellcolor{lightgray} & \ \ \ \\
        \hline
    \end{tabular} \\ 
    \\
    \begin{tabular}{|c|c|c|}
        \hline
         \ \ \  & \cellcolor{lightgray} & \cellcolor{lightgray} \\
        \hline
        \cellcolor{lightgray} & \ \ \  & \cellcolor{lightgray} \\
        \hline
        \cellcolor{lightgray} & \cellcolor{lightgray} & \ \ \ \\
        \hline
    \end{tabular} \ 
    \begin{tabular}{|c|c|c|}
        \hline
        \ \ \  & $\bullet$ &  \\
        \hline
        $\bullet$ & \ \ \  & $\bullet$\\
        \hline
         & $\bullet$ & \ \ \ \\
        \hline
    \end{tabular} \ 
    \begin{tabular}{|c|c|c|}
        \hline
         \ \ \  & \cellcolor{lightgray} & \cellcolor{lightgray} \\
        \hline
        \cellcolor{lightgray} & \ \ \  & \cellcolor{lightgray} \\
        \hline
        \cellcolor{lightgray} & \cellcolor{lightgray} & \ \ \ \\
        \hline
    \end{tabular} \\
    \\
    \begin{tabular}{|c|c|c|}
        \hline
         \ \ \  & \cellcolor{lightgray} & \cellcolor{lightgray} \\
        \hline
        \cellcolor{lightgray} & \ \ \  & \cellcolor{lightgray} \\
        \hline
        \cellcolor{lightgray} & \cellcolor{lightgray} & \ \ \ \\
        \hline
    \end{tabular} \ 
    \begin{tabular}{|c|c|c|}
        \hline
         \ \ \  & \cellcolor{lightgray} & \cellcolor{lightgray} \\
        \hline
        \cellcolor{lightgray} & \ \ \  & \cellcolor{lightgray} \\
        \hline
        \cellcolor{lightgray} & \cellcolor{lightgray} & \ \ \ \\
        \hline
    \end{tabular} \ 
    \begin{tabular}{|c|c|c|}
        \hline
         \ \ \  &  & $\bullet$ \\
        \hline
         & \ \ \  & $\bullet$ \\
        \hline
        $\bullet$ & $\bullet$ & \ \ \ \\
        \hline
    \end{tabular}
    
    \end{tabular}
\]
Observe that $X$ can only contain four of the six points in the upper left corners of the six remaining blocks. Similarly, $X$ can only contain $4$ of the $6$ middle points of these blocks and $X$ can only contain $4$ of the $6$ bottom right points of these blocks.

Finally, we consider the case where the points of $X$ are divided as     \begin{tabular}{|c|c|c|}
        \hline
        $4$ & $3$ & $3$\\
        \hline
        $3$ & $4$ & $3$\\
        \hline
        $3$ & $3$ & $3$\\
        \hline
    \end{tabular} and the pair of blocks with four points of $X$ satisfy the third possibility. In this case, up to isomorphism, these blocks must have the following points. This implies that the points shown in gray cannot be in $X$.
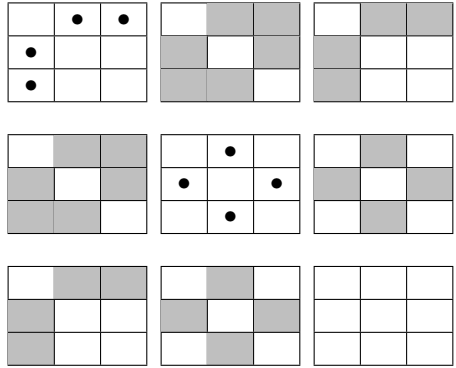
        \[
    \begin{tabular}{c}
    
    \begin{tabular}{|c|c|c|}
        \hline
        \ \ \  & $\bullet$ & $\bullet$ \\
        \hline
        $\bullet$ & \ \ \  & \\
        \hline
        $\bullet$ &  & \ \ \ \\
        \hline
    \end{tabular} \ 
    \begin{tabular}{|c|c|c|}
        \hline
         \ \ \  & \cellcolor{lightgray} & \cellcolor{lightgray} \\
        \hline
        \cellcolor{lightgray} & \ \ \  & \cellcolor{lightgray} \\
        \hline
        \cellcolor{lightgray} & \cellcolor{lightgray} & \ \ \ \\
        \hline
    \end{tabular} \  
    \begin{tabular}{|c|c|c|}
        \hline
         \ \ \  & \cellcolor{lightgray} & \cellcolor{lightgray} \\
        \hline
        \cellcolor{lightgray} & \ \ \  &  \\
        \hline
        \cellcolor{lightgray} &  & \ \ \ \\
        \hline
    \end{tabular}\\ 
    \\
    \begin{tabular}{|c|c|c|}
        \hline
         \ \ \  & \cellcolor{lightgray} & \cellcolor{lightgray} \\
        \hline
        \cellcolor{lightgray} & \ \ \  & \cellcolor{lightgray} \\
        \hline
        \cellcolor{lightgray} & \cellcolor{lightgray} & \ \ \ \\
        \hline
    \end{tabular} \ 
    \begin{tabular}{|c|c|c|}
        \hline
        \ \ \  & $\bullet$ &  \\
        \hline
        $\bullet$ & \ \ \  & $\bullet$\\
        \hline
         & $\bullet$ & \ \ \ \\
        \hline
    \end{tabular} \ 
    \begin{tabular}{|c|c|c|}
        \hline
         \ \ \  & \cellcolor{lightgray} &  \\
        \hline
        \cellcolor{lightgray} & \ \ \  & \cellcolor{lightgray} \\
        \hline
         & \cellcolor{lightgray} & \ \ \ \\
        \hline
    \end{tabular} \\
    \\
    \begin{tabular}{|c|c|c|}
        \hline
         \ \ \  & \cellcolor{lightgray} & \cellcolor{lightgray} \\
        \hline
        \cellcolor{lightgray} & \ \ \  &  \\
        \hline
        \cellcolor{lightgray} &  & \ \ \ \\
        \hline
    \end{tabular} \ 
    \begin{tabular}{|c|c|c|}
        \hline
         \ \ \  & \cellcolor{lightgray} &  \\
        \hline
        \cellcolor{lightgray} & \ \ \  & \cellcolor{lightgray} \\
        \hline
         & \cellcolor{lightgray} & \ \ \ \\
        \hline
    \end{tabular} \ 
    \begin{tabular}{|c|c|c|}
        \hline
         \ \ \  &  &  \\
        \hline
         & \ \ \  &  \\
        \hline
         &  & \ \ \ \\
        \hline
    \end{tabular}
    
    \end{tabular}
\]
For the bottom left and upper right blocks, $X$ must contain the top left point. Moreover, we must either have that the middle and bottom right points are in $X$ or the middle right and bottom middle points are in $X$. Observe that the first case is impossible as this would eliminate too many possible points from the middle right or bottom middle block. Thus, for the bottom left and upper right blocks, $X$ must contain the top left, middle right, and bottom middle points. Note that this eliminates the top left point of the middle right and bottom middle blocks.

For the middle right and bottom middle blocks, observe that $X$ cannot contain the bottom right point as this would eliminate both the top right and the bottom left points. Thus, for the the middle right and bottom middle blocks, $X$ must contain the top right, middle, and bottom left points.

Thus, the points for all of the blocks except the bottom right block must be as follows.
    
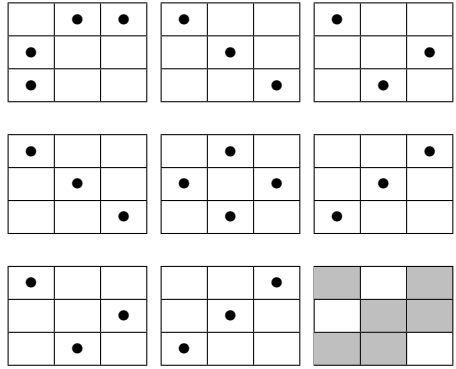
\[
    \begin{tabular}{c}
    
    \begin{tabular}{|c|c|c|}
        \hline
        \ \ \  & $\bullet$ & $\bullet$ \\
        \hline
        $\bullet$ & \ \ \  & \\
        \hline
        $\bullet$ &  & \ \ \ \\
        \hline
    \end{tabular} \ 
    \begin{tabular}{|c|c|c|}
        \hline
        $\bullet$ & \ \ \  &  \\
        \hline
         & $\bullet$ & \ \ \ \\
        \hline
        \ \ \  &  & $\bullet$ \\
        \hline
    \end{tabular} \ 
    \begin{tabular}{|c|c|c|}
        \hline
        $\bullet$ & \ \ \  &  \\
        \hline
        \ \ \  &  & $\bullet$ \\
        \hline
         & $\bullet$ & \ \ \ \\
        \hline
    \end{tabular} \\ 
    \\
    \begin{tabular}{|c|c|c|}
        \hline
        $\bullet$ & \ \ \  &  \\
        \hline
         & $\bullet$ & \ \ \ \\
        \hline
        \ \ \  &  & $\bullet$ \\
        \hline
    \end{tabular} \ 
    \begin{tabular}{|c|c|c|}
        \hline
        \ \ \  & $\bullet$ &  \\
        \hline
        $\bullet$ & \ \ \  & $\bullet$\\
        \hline
         & $\bullet$ & \ \ \ \\
        \hline
    \end{tabular} \ 
    \begin{tabular}{|c|c|c|}
        \hline
        \ \ \  &  & $\bullet$ \\
        \hline
         & $\bullet$ & \ \ \ \\
        \hline
        $\bullet$ & \ \ \  &  \\
        \hline
    \end{tabular} \\
    \\
    \begin{tabular}{|c|c|c|}
        \hline
        $\bullet$ & \ \ \  &  \\
        \hline
        \ \ \  &  & $\bullet$ \\
        \hline
         & $\bullet$ & \ \ \ \\
        \hline
    \end{tabular} \ 
    \begin{tabular}{|c|c|c|}
        \hline
        \ \ \  &  & $\bullet$ \\
        \hline
         & $\bullet$ & \ \ \ \\
        \hline
        $\bullet$ & \ \ \  &  \\
        \hline
    \end{tabular} \ 
    \begin{tabular}{|c|c|c|}
        \hline
        \cellcolor{lightgray} & \ \ \  & \cellcolor{lightgray} \\
        \hline
        \ \ \  & \cellcolor{lightgray} & \cellcolor{lightgray} \\
        \hline
        \cellcolor{lightgray} & \cellcolor{lightgray} & \ \ \ \\
        \hline
    \end{tabular}
    
    \end{tabular}
\]
For the bottom right block, the points shown in gray cannot be in $X$. There are exactly three points remaining so $X$ must contain these points. This gives the set $X_4$ shown at the beginning of Section \ref{section:lower}.
\end{proof}

The following fact can be checked directly.
\begin{Cor}
If $X \subseteq \mathbb{Z}_3^{4}$, $X$ has maximum induced degree $1$ and $|X| = 3^3 + 2$ then for each $i \in [4]$ and all $x \in \mathbb{Z}_3^{4}$, the line $l = \{x,x+e_i,x+2e_i\}$ has at least one point of $X$.
\end{Cor}
\section{Example where two canonical sets have an intersection which is not a canonical set}\label{sec:strangeintersection}

In this section, we show that when we consider general subsets of $\Z_3^n$ rather than \1sat subsets, it is possible for two canonical sets to have an intersection which is not a canonical set. We show this by giving a $U \subseteq \Z_3^{6}$ where there are two canonical subsets of $U$ which are isomorphic to $D_4$ and intersect in a $2$-dimensional affine subset where they only have two points.

The first canonical set is $S = \{x \in U: x_4 = x_5 = 0\}$ (i.e., the points in the first three blocks in the top row of blocks) and the second canonical set is $S' = \{x \in U: x_1 = 1, x_3 = 2\}$ (i.e., all of the $B$ shown in gray cells). Note that $S \cap S'$ consists of the two $B$ in the bottom row of the first set of blocks which are marked with a star.

\[
    \begin{tabular}{c c}
    \begin{tabular}{c}
    \begin{tabular}{|c|c|c|}
        \hline
        $X$ & $A$ & $A$\\
        \hline
        $A$ & $B$ & $B$\\
        \hline
        $A$ & $C$ & $C$\\
        \hline
    \end{tabular} \ 
    \begin{tabular}{|c|c|c|}
        \hline
        $A$ & $C$ & $B$\\
        \hline
        $B$ & $A$ & $C$\\
        \hline
        $C$ & \cellcolor{lightgray}$B^*$ & $A$\\
        \hline
    \end{tabular} \ 
    \begin{tabular}{|c|c|c|}
        \hline
        $A$ & $C$ & $B$\\
        \hline
        $B$ & $A$ & $C$\\
        \hline
        $C$ & \cellcolor{lightgray}$B^*$ & $A$\\
        \hline
    \end{tabular}\\
    \\
    \begin{tabular}{|c|c|c|}
        \hline
        $A$ & $C$ & $C$\\
        \hline
        $C$ & $A$ & $A$\\
        \hline
        $ $ & \cellcolor{lightgray}$B$ & \cellcolor{lightgray}$B$\\
        \hline
    \end{tabular} \ 
    \begin{tabular}{|c|c|c|}
        \hline
        $C$ & $B$ & $A$\\
        \hline
        $A$ & $C$ & $B$\\
        \hline
        \cellcolor{lightgray}$B$ & $A$ & $C$\\
        \hline
    \end{tabular} \ 
    \begin{tabular}{|c|c|c|}
        \hline
        $C$ & $B$ & $A$\\
        \hline
        $A$ & $C$ & $B$\\
        \hline
        \cellcolor{lightgray}$B$ & $A$ & $C$\\
        \hline
    \end{tabular}\\
    \\
    \begin{tabular}{|c|c|c|}
        \hline
        $A$ & $B$ & $B$\\
        \hline
        $ $ & $C$ & $C$\\
        \hline
        \cellcolor{lightgray}$B$ & $A$ & $A$\\
        \hline
    \end{tabular} \ 
    \begin{tabular}{|c|c|c|}
        \hline
        $B$ & $A$ & $C$\\
        \hline
        $C$ & $B$ & $A$\\
        \hline
        $A$ & $C$ & \cellcolor{lightgray}$B$\\
        \hline
    \end{tabular} \ 
    \begin{tabular}{|c|c|c|}
        \hline
        $B$ & $A$ & $C$\\
        \hline
        $C$ & $B$ & $A$\\
        \hline
        $A$ & $C$ & \cellcolor{lightgray}$B$\\
        \hline
    \end{tabular}
    \end{tabular} 
    
    \begin{tabular}{c}\
    \begin{tabular}{|c|c|c|}
        \hline
        $A$ & $C$ & $C$\\
        \hline
        $C$ & $A$ & $A$\\
        \hline
        $ $ & \cellcolor{lightgray}$B$ & \cellcolor{lightgray}$B$\\
        \hline
    \end{tabular} \ 
    \begin{tabular}{|c|c|c|}
        \hline
        $C$ & $B$ & $A$\\
        \hline
        $A$ & $C$ & $B$\\
        \hline
        \cellcolor{lightgray}$B$ & $A$ & $C$\\
        \hline
    \end{tabular} \ 
    \begin{tabular}{|c|c|c|}
        \hline
        $C$ & $B$ & $A$\\
        \hline
        $A$ & $C$ & $B$\\
        \hline
        \cellcolor{lightgray}$B$ & $A$ & $C$\\
        \hline
    \end{tabular}\\
    \\
    \begin{tabular}{|c|c|c|}
        \hline
        $C$ & $B$ & $B$\\
        \hline
        $Z$ & $C$ & $C$\\
        \hline
        $C$ & $A$ & $A$\\
        \hline
    \end{tabular} \ 
    \begin{tabular}{|c|c|c|}
        \hline
        $B$ & $A$ & $C$\\
        \hline
        $C$ & $B$ & $A$\\
        \hline
        $A$ & $C$ & \cellcolor{lightgray}$B$\\
        \hline
    \end{tabular} \ 
    \begin{tabular}{|c|c|c|}
        \hline
        $B$ & $A$ & $C$\\
        \hline
        $C$ & $B$ & $A$\\
        \hline
        $A$ & $C$ & \cellcolor{lightgray}$B$\\
        \hline
    \end{tabular}\\
    \\
    \begin{tabular}{|c|c|c|}
        \hline
        $ $ & $A$ & $A$\\
        \hline
        $C$ & $B$ & $B$\\
        \hline
        \cellcolor{lightgray}$B$ & $C$ & $C$\\
        \hline
    \end{tabular} \ 
    \begin{tabular}{|c|c|c|}
        \hline
        $A$ & $C$ & $B$\\
        \hline
        $B$ & $A$ & $C$\\
        \hline
        $C$ & \cellcolor{lightgray}$B$ & $A$\\
        \hline
    \end{tabular} \ 
    \begin{tabular}{|c|c|c|}
        \hline
        $A$ & $C$ & $B$\\
        \hline
        $B$ & $A$ & $C$\\
        \hline
        $C$ & \cellcolor{lightgray}$B$ & $A$\\
        \hline
    \end{tabular}
    \end{tabular}
    \end{tabular}
\]
\[
    \begin{tabular}{c}
    \begin{tabular}{|c|c|c|}
        \hline
        $A$ & $B$ & $B$\\
        \hline
        $ $ & $C$ & $C$\\
        \hline
        \cellcolor{lightgray}$B$ & $A$ & $A$\\
        \hline
    \end{tabular} \ 
    \begin{tabular}{|c|c|c|}
        \hline
        $B$ & $A$ & $C$\\
        \hline
        $C$ & $B$ & $A$\\
        \hline
        $A$ & $C$ & \cellcolor{lightgray}$B$\\
        \hline
    \end{tabular} \ 
    \begin{tabular}{|c|c|c|}
        \hline
        $B$ & $A$ & $C$\\
        \hline
        $C$ & $B$ & $A$\\
        \hline
        $A$ & $C$ & \cellcolor{lightgray}$B$\\
        \hline
    \end{tabular}\\
    \\
    \begin{tabular}{|c|c|c|}
        \hline
        $ $ & $A$ & $A$\\
        \hline
        $C$ & $B$ & $B$\\
        \hline
        \cellcolor{lightgray}$B$ & $C$ & $C$\\
        \hline
    \end{tabular} \ 
    \begin{tabular}{|c|c|c|}
        \hline
        $A$ & $C$ & $B$\\
        \hline
        $B$ & $A$ & $C$\\
        \hline
        $C$ & \cellcolor{lightgray}$B$ & $A$\\
        \hline
    \end{tabular} \ 
    \begin{tabular}{|c|c|c|}
        \hline
        $A$ & $C$ & $B$\\
        \hline
        $B$ & $A$ & $C$\\
        \hline
        $C$ & \cellcolor{lightgray}$B$ & $A$\\
        \hline
    \end{tabular}\\
    \\
\begin{tabular}{|c|c|c|}
        \hline
        $B$ & $C$ & $C$\\
        \hline
        $B$ & $A$ & $A$\\
        \hline
        $Y$ & \cellcolor{lightgray}$B$ & \cellcolor{lightgray}$B$\\
        \hline
    \end{tabular} \ 
    \begin{tabular}{|c|c|c|}
        \hline
        $C$ & $B$ & $A$\\
        \hline
        $A$ & $C$ & $B$\\
        \hline
        \cellcolor{lightgray}$B$ & $A$ & $C$\\
        \hline
    \end{tabular} \ 
    \begin{tabular}{|c|c|c|}
        \hline
        $C$ & $B$ & $A$\\
        \hline
        $A$ & $C$ & $B$\\
        \hline
        \cellcolor{lightgray}$B$ & $A$ & $C$\\
        \hline
    \end{tabular}
    \end{tabular}
\]

\section{1-skew functions for $n = 3$}\label{sec:nequals3badfunctions}
We list all 1-skew functions for $n = 3$ which extend \eqref{D3} and have the first row of the second block equal to $(A, B, C)$.\\

\noindent
\textbf{Isomorphic to \eqref{inde}:}
    \begin{enumerate}[(i)]

    \item \[
    \begin{tabular}{|c|c|c|}
        \hline
        $X$ & $A$ & $A$\\
        \hline
        $A$ & $B$ & $C$\\
        \hline
        $A$ & $C$ & $B$\\
        \hline
    \end{tabular} \ 
    \begin{tabular}{|c|c|c|}
        \hline
        $A$ & $B$ & $C$\\
        \hline
        $B$ & $C$ & $A$\\
        \hline
        $C$ & $A$ & $B$\\
        \hline
    \end{tabular} \ 
    \begin{tabular}{|c|c|c|}
        \hline
        $A$ & $C$ & $B$\\
        \hline
        $B$ & $A$ & $C$\\
        \hline
        $C$ & $B$ & $A$\\
        \hline
    \end{tabular}.
    \]

    \item \[
    \begin{tabular}{|c|c|c|}
        \hline
        $X$ & $A$ & $A$\\
        \hline
        $A$ & $B$ & $C$\\
        \hline
        $A$ & $C$ & $B$\\
        \hline
    \end{tabular} \ 
    \begin{tabular}{|c|c|c|}
        \hline
        $A$ & $B$ & $C$\\
        \hline
        $B$ & $C$ & $A$\\
        \hline
        $B$ & $A$ & $C$\\
        \hline
    \end{tabular} \ 
    \begin{tabular}{|c|c|c|}
        \hline
        $A$ & $C$ & $B$\\
        \hline
        $C$ & $A$ & $B$\\
        \hline
        $C$ & $B$ & $A$\\
        \hline
    \end{tabular}.
    \]

    \item \[
    \begin{tabular}{|c|c|c|}
        \hline
        $X$ & $A$ & $A$\\
        \hline
        $A$ & $B$ & $C$\\
        \hline
        $A$ & $C$ & $B$\\
        \hline
    \end{tabular} \ 
    \begin{tabular}{|c|c|c|}
        \hline
        $A$ & $B$ & $C$\\
        \hline
        $C$ & $A$ & $B$\\
        \hline
        $B$ & $C$ & $A$\\
        \hline
    \end{tabular} \ 
    \begin{tabular}{|c|c|c|}
        \hline
        $A$ & $C$ & $B$\\
        \hline
        $C$ & $B$ & $A$\\
        \hline
        $B$ & $A$ & $C$\\
        \hline
    \end{tabular}.
    \]

    \item[(iv)] \[
    \begin{tabular}{|c|c|c|}
        \hline
        $X$ & $A$ & $A$\\
        \hline
        $A$ & $B$ & $C$\\
        \hline
        $A$ & $C$ & $B$\\
        \hline
    \end{tabular} \ 
    \begin{tabular}{|c|c|c|}
        \hline
        $A$ & $B$ & $C$\\
        \hline
        $C$ & $A$ & $B$\\
        \hline
        $C$ & $B$ & $A$\\
        \hline
    \end{tabular} \ 
    \begin{tabular}{|c|c|c|}
        \hline
        $A$ & $C$ & $B$\\
        \hline
        $B$ & $C$ & $A$\\
        \hline
        $B$ & $A$ & $C$\\
        \hline
    \end{tabular}.
    \]
    \end{enumerate}

\noindent
\textbf{Isomorphic to \eqref{As}:}
    \begin{enumerate}    
    \item[(v)] \[
    \begin{tabular}{|c|c|c|}
        \hline
        $X$ & $A$ & $A$\\
        \hline
        $A$ & $B$ & $C$\\
        \hline
        $A$ & $C$ & $B$\\
        \hline
    \end{tabular} \ 
    \begin{tabular}{|c|c|c|}
        \hline
        $A$ & $B$ & $C$\\
        \hline
        $B$ & $C$ & $A$\\
        \hline
        $C$ & $A$ & $B$\\
        \hline
    \end{tabular} \ 
    \begin{tabular}{|c|c|c|}
        \hline
        $A$ & $C$ & $B$\\
        \hline
        $C$ & $B$ & $A$\\
        \hline
        $B$ & $A$ & $C$\\
        \hline
    \end{tabular}.
    \]

    \item[(vi)] \[
    \begin{tabular}{|c|c|c|}
        \hline
        $X$ & $A$ & $A$\\
        \hline
        $A$ & $B$ & $C$\\
        \hline
        $A$ & $C$ & $B$\\
        \hline
    \end{tabular} \ 
    \begin{tabular}{|c|c|c|}
        \hline
        $A$ & $B$ & $C$\\
        \hline
        $B$ & $C$ & $B$\\
        \hline
        $C$ & $A$ & $A$\\
        \hline
    \end{tabular} \ 
    \begin{tabular}{|c|c|c|}
        \hline
        $A$ & $C$ & $B$\\
        \hline
        $C$ & $A$ & $A$\\
        \hline
        $B$ & $B$ & $C$\\
        \hline
    \end{tabular}.
    \]

    \item[(vii)] \[
    \begin{tabular}{|c|c|c|}
        \hline
        $X$ & $A$ & $A$\\
        \hline
        $A$ & $B$ & $C$\\
        \hline
        $A$ & $C$ & $B$\\
        \hline
    \end{tabular} \ 
    \begin{tabular}{|c|c|c|}
        \hline
        $A$ & $B$ & $C$\\
        \hline
        $B$ & $C$ & $A$\\
        \hline
        $C$ & $B$ & $A$\\
        \hline
    \end{tabular} \ 
    \begin{tabular}{|c|c|c|}
        \hline
        $A$ & $C$ & $B$\\
        \hline
        $C$ & $A$ & $B$\\
        \hline
        $B$ & $A$ & $C$\\
        \hline
    \end{tabular}.
    \]

    \item[(viii)] \[
    \begin{tabular}{|c|c|c|}
        \hline
        $X$ & $A$ & $A$\\
        \hline
        $A$ & $B$ & $C$\\
        \hline
        $A$ & $C$ & $B$\\
        \hline
    \end{tabular} \ 
    \begin{tabular}{|c|c|c|}
        \hline
        $A$ & $B$ & $C$\\
        \hline
        $C$ & $A$ & $B$\\
        \hline
        $B$ & $C$ & $A$\\
        \hline
    \end{tabular} \ 
    \begin{tabular}{|c|c|c|}
        \hline
        $A$ & $C$ & $B$\\
        \hline
        $B$ & $A$ & $C$\\
        \hline
        $C$ & $B$ & $A$\\
        \hline
    \end{tabular}.
    \]

    \item[(ix)] \[
    \begin{tabular}{|c|c|c|}
        \hline
        $X$ & $A$ & $A$\\
        \hline
        $A$ & $B$ & $C$\\
        \hline
        $A$ & $C$ & $B$\\
        \hline
    \end{tabular} \ 
    \begin{tabular}{|c|c|c|}
        \hline
        $A$ & $B$ & $C$\\
        \hline
        $C$ & $C$ & $B$\\
        \hline
        $B$ & $A$ & $A$\\
        \hline
    \end{tabular} \ 
    \begin{tabular}{|c|c|c|}
        \hline
        $A$ & $C$ & $B$\\
        \hline
        $B$ & $A$ & $A$\\
        \hline
        $C$ & $B$ & $C$\\
        \hline
    \end{tabular}.
    \]

    \item[(x)] \[
    \begin{tabular}{|c|c|c|}
        \hline
        $X$ & $A$ & $A$\\
        \hline
        $A$ & $B$ & $C$\\
        \hline
        $A$ & $C$ & $B$\\
        \hline
    \end{tabular} \ 
    \begin{tabular}{|c|c|c|}
        \hline
        $A$ & $B$ & $C$\\
        \hline
        $C$ & $A$ & $B$\\
        \hline
        $B$ & $A$ & $C$\\
        \hline
    \end{tabular} \ 
    \begin{tabular}{|c|c|c|}
        \hline
        $A$ & $C$ & $B$\\
        \hline
        $B$ & $C$ & $A$\\
        \hline
        $C$ & $B$ & $A$\\
        \hline
    \end{tabular}.
    \]
    \end{enumerate}

\noindent
\textbf{Isomorphic to \eqref{extra_point}:}    \begin{enumerate}
    \item[(xi)] \[
    \begin{tabular}{|c|c|c|}
        \hline
        $X$ & $A$ & $A$\\
        \hline
        $A$ & $B$ & $C$\\
        \hline
        $A$ & $C$ & $B$\\
        \hline
    \end{tabular} \ 
    \begin{tabular}{|c|c|c|}
        \hline
        $A$ & $B$ & $C$\\
        \hline
        $B$ & $C$ & $A$\\
        \hline
        $C$ & $A$ & $B$\\
        \hline
    \end{tabular} \ 
    \begin{tabular}{|c|c|c|}
        \hline
        $A$ & $C$ & $B$\\
        \hline
        $C$ & $A$ & $B$\\
        \hline
        $B$ & $B$ & $E$\\
        \hline
    \end{tabular}
    \]
    where $E \in \{A, C, Y\}$.

    \item[(xii)] \[
    \begin{tabular}{|c|c|c|}
        \hline
        $X$ & $A$ & $A$\\
        \hline
        $A$ & $B$ & $C$\\
        \hline
        $A$ & $C$ & $B$\\
        \hline
    \end{tabular} \ 
    \begin{tabular}{|c|c|c|}
        \hline
        $A$ & $B$ & $C$\\
        \hline
        $B$ & $E$ & $B$\\
        \hline
        $C$ & $B$ & $A$\\
        \hline
    \end{tabular} \ 
    \begin{tabular}{|c|c|c|}
        \hline
        $A$ & $C$ & $B$\\
        \hline
        $C$ & $B$ & $A$\\
        \hline
        $B$ & $A$ & $C$\\
        \hline
    \end{tabular}
    \]
    where $E \in \{A, C, Y\}$.

    \item[(xiii)] \[
    \begin{tabular}{|c|c|c|}
        \hline
        $X$ & $A$ & $A$\\
        \hline
        $A$ & $B$ & $C$\\
        \hline
        $A$ & $C$ & $B$\\
        \hline
    \end{tabular} \ 
    \begin{tabular}{|c|c|c|}
        \hline
        $A$ & $B$ & $C$\\
        \hline
        $C$ & $A$ & $B$\\
        \hline
        $B$ & $C$ & $A$\\
        \hline
    \end{tabular} \ 
    \begin{tabular}{|c|c|c|}
        \hline
        $A$ & $C$ & $B$\\
        \hline
        $B$ & $C$ & $A$\\
        \hline
        $C$ & $F$ & $C$\\
        \hline
    \end{tabular}
    \]
    where $F \in \{A, B, Z\}$.

    \item[(xiv)] \[
    \begin{tabular}{|c|c|c|}
        \hline
        $X$ & $A$ & $A$\\
        \hline
        $A$ & $B$ & $C$\\
        \hline
        $A$ & $C$ & $B$\\
        \hline
    \end{tabular} \ 
    \begin{tabular}{|c|c|c|}
        \hline
        $A$ & $B$ & $C$\\
        \hline
        $C$ & $C$ & $F$\\
        \hline
        $B$ & $A$ & $C$\\
        \hline
    \end{tabular} \ 
    \begin{tabular}{|c|c|c|}
        \hline
        $A$ & $C$ & $B$\\
        \hline
        $B$ & $A$ & $C$\\
        \hline
        $C$ & $B$ & $A$\\
        \hline
    \end{tabular}
    \]
    where $F \in \{A, B, Z\}$.
    \end{enumerate}

\section{Proof of Lemma \ref{forbidden}}\label{sec:forbiddenlemmaproof}
Recall Lemma \ref{forbidden}:
\begin{Lemma}
    Let $U_f : \Z_3^4 \to \{A, B, C, X, Y, Z\}$. If $U_f(x_1, x_2, x_3, 0)$ is identical to \eqref{inde} or \eqref{As}, and $U_f(x_1, x_2, 0, x_4)$ is isomorphic to \eqref{inde} or \eqref{As}. Then $U_f$ has induced degree at least 2.
\end{Lemma}

    \begin{proof}
    We analyze the two major cases separately.
    
    \begin{itemize}
    \item \textbf{$U_f(x_1, x_2, x_3, 0)$ is identical to \eqref{inde}.}
    
    We show that when $U_f(x_1, x_2, 0, x_4)$ is isomorphic to \eqref{inde} or \eqref{As}, $U_f$ must have induced degree at least 2. We consider the following partial function, where the entries are to be determined.
\[
    \begin{tabular}{c}
    
    \begin{tabular}{|c|c|c|}
        \hline
        $X$ & $A$ & $A$\\
        \hline
        $A$ & $B$ & $C$\\
        \hline
        $A$ & $C$ & $B$\\
        \hline
    \end{tabular} \ 
    \begin{tabular}{|c|c|c|}
        \hline
        $A$ & $B$ & $C$\\
        \hline
        $B$ & $C$ & $A$\\
        \hline
        $C$ & $A$ & $B$\\
        \hline
    \end{tabular} \ 
    \begin{tabular}{|c|c|c|}
        \hline
        $A$ & $C$ & $B$\\
        \hline
        $B$ & $A$ & $C$\\
        \hline
        $C$ & $B$ & $A$\\
        \hline
    \end{tabular}\\
    \\
    \begin{tabular}{|c|c|c|}
        \hline
        $A$ & $B$ & $C$\\
        \hline
         &  & \\
        \hline
         &  & \\
        \hline
    \end{tabular} \ 
    \begin{tabular}{|c|c|c|}
        \hline
        ${\transparent{0} E}$ & ${\transparent{0} E}$ & ${\transparent{0} E}$\\
        \hline
         &  & \\
        \hline
         &  & \\
        \hline
    \end{tabular} \ 
    \begin{tabular}{|c|c|c|}
        \hline
        ${\transparent{0} E}$ & ${\transparent{0} E}$ & ${\transparent{0} E}$\\
        \hline
         &  & \\
        \hline
         &  & \\
        \hline
    \end{tabular}\\
    \\
    \begin{tabular}{|c|c|c|}
        \hline
        $A$ & $C$ & $B$\\
        \hline
         &  & \\
        \hline
         &  & \\
        \hline
    \end{tabular} \ 
    \begin{tabular}{|c|c|c|}
        \hline
        ${\transparent{0} E}$ & ${\transparent{0} E}$ & ${\transparent{0} E}$\\
        \hline
         &  & \\
        \hline
         &  & \\
        \hline
    \end{tabular} \ 
    \begin{tabular}{|c|c|c|}
        \hline
        ${\transparent{0} E}$ & ${\transparent{0} E}$ & ${\transparent{0} E}$\\
        \hline
         &  & \\
        \hline
         &  & \\
        \hline
    \end{tabular}

    \end{tabular}
\]

Now the first column must be identical to one of the cases among (i) to (x) in Appendix \ref{sec:nequals3badfunctions}. Each of case (i)-(iv) enjoys the same symmetry as \eqref{inde}. Thus we need to consider them one by one.

\begin{itemize}
    \item Case (i). The $B$ in the gray cells immediately cause $U_f$ to have induced degree 2.
    \[
    \begin{tabular}{c}

    \begin{tabular}{|c|c|c|}
        \hline
        $X$ & $A$ & $A$\\
        \hline
        $A$ & $B$ & $C$\\
        \hline
        $A$ & $C$ & \cellcolor{lightgray}$B$\\
        \hline
    \end{tabular} \ 
    \begin{tabular}{|c|c|c|}
        \hline
        $A$ & $B$ & $C$\\
        \hline
        $B$ & $C$ & $A$\\
        \hline
        $C$ & $A$ & \cellcolor{lightgray}$B$\\
        \hline
    \end{tabular} \ 
    \begin{tabular}{|c|c|c|}
        \hline
        $A$ & $C$ & $B$\\
        \hline
        $B$ & $A$ & $C$\\
        \hline
        $C$ & $B$ & $A$\\
        \hline
    \end{tabular}\\
    \\
    \begin{tabular}{|c|c|c|}
        \hline
        $A$ & $B$ & $C$\\
        \hline
        $B$ & $C$ & $A$\\
        \hline
        $C$ & $A$ & \cellcolor{lightgray}$B$\\
        \hline
    \end{tabular} \ 
    \begin{tabular}{|c|c|c|}
        \hline
        ${\transparent{0} E}$ & ${\transparent{0} E}$ & ${\transparent{0} E}$\\
        \hline
         &  & \\
        \hline
         &  & \\
        \hline
    \end{tabular} \ 
    \begin{tabular}{|c|c|c|}
        \hline
        ${\transparent{0} E}$ & ${\transparent{0} E}$ & ${\transparent{0} E}$\\
        \hline
         &  & \\
        \hline
         &  & \\
        \hline
    \end{tabular}\\
    \\
    \begin{tabular}{|c|c|c|}
        \hline
        $A$ & $C$ & $B$\\
        \hline
        $B$ & $A$ & $C$\\
        \hline
        $C$ & $B$ & $A$\\
        \hline
    \end{tabular} \ 
    \begin{tabular}{|c|c|c|}
        \hline
        ${\transparent{0} E}$ & ${\transparent{0} E}$ & ${\transparent{0} E}$\\
        \hline
         &  & \\
        \hline
         &  & \\
        \hline
    \end{tabular} \ 
    \begin{tabular}{|c|c|c|}
        \hline
        ${\transparent{0} E}$ & ${\transparent{0} E}$ & ${\transparent{0} E}$\\
        \hline
         &  & \\
        \hline
         &  & \\
        \hline
    \end{tabular}
    
    \end{tabular}
\].

    \item Case (ii). The subscripts of the entries indicate the order in which they are deduced. The entries in the gray cells certify the induced degree is at least 2.
    \[
    \begin{tabular}{c}

    \begin{tabular}{|c|c|c|}
        \hline
        $X$ & $A$ & $A$\\
        \hline
        $A$ & $B$ & \cellcolor{lightgray}$C$\\
        \hline
        $A$ & $C$ & $B$\\
        \hline
    \end{tabular} \ 
    \begin{tabular}{|c|c|c|}
        \hline
        $A_{\transparent{0} 0}$ & $B$ & $C$\\
        \hline
        $B_{\transparent{0} 0}$ & $C$ & $A$\\
        \hline
        $C_{\transparent{0} 0}$ & $A$ & $B$\\
        \hline
    \end{tabular} \ 
    \begin{tabular}{|c|c|c|}
        \hline
        $A_{\transparent{0} 0}$ & $C$ & $B_{\transparent{0} 0}$\\
        \hline
        $B_{\transparent{0} 0}$ & $A$ & \cellcolor{lightgray}$C_{\transparent{0} 0}$\\
        \hline
        $C_{\transparent{0} 0}$ & $B$ & $A_{\transparent{0} 0}$\\
        \hline
    \end{tabular}\\
    \\
    \begin{tabular}{|c|c|c|}
        \hline
        $A$ \ & $B$ & $C$\\
        \hline
        $B$ & $C$ & $A$\\
        \hline
        $B$ & $A$ & $C$\\
        \hline
    \end{tabular} \ 
    \begin{tabular}{|c|c|c|}
        \hline
        ${\transparent{0} E_0}$ & ${\transparent{0} E}$ & ${\transparent{0} E}$\\
        \hline
         &  & \\
        \hline
         &  & \\
        \hline
    \end{tabular} \ 
    \begin{tabular}{|c|c|c|}
        \hline
        ${\transparent{0} E_0}$ & ${\transparent{0} E}$ & ${\transparent{0} E_0}$\\
        \hline
         &  & \\
        \hline
         &  & \\
        \hline
    \end{tabular}\\
    \\
    \begin{tabular}{|c|c|c|}
        \hline
        $A$ & $C$ & $B$\\
        \hline
        $C$ & $A$ & $B$\\
        \hline
        $C$ & $B$ & $A$\\
        \hline
    \end{tabular} \ 
    \begin{tabular}{|c|c|c|}
        \hline
        ${\transparent{0} E}$ & ${\transparent{0} E}$ & ${\transparent{0} E}$\\
        \hline
        $A_1$ &  & \\
        \hline
         &  & \\
        \hline
    \end{tabular} \ 
    \begin{tabular}{|c|c|c|}
        \hline
        ${\transparent{0} E}$ & ${\transparent{0} E}$ & ${\transparent{0} E}$\\
        \hline
        $A_2$ &  & \cellcolor{lightgray}$C_3$\\
        \hline
         &  & \\
        \hline
    \end{tabular}
    
    \end{tabular}
\]

    \item Case (iii). The gray cell at $(0, 2, 1, 1)$ cannot be $A$, $B$ or $C$, otherwise $U_f$ will have induced degree 2. But it cannot be $X, Y, Z$ either.
    \[
    \begin{tabular}{c}
    
    \begin{tabular}{|c|c|c|}
        \hline
        $X$ & $A$ & $A$\\
        \hline
        $A$ & $B$ & $C$\\
        \hline
        $A$ & $C$ & $B$\\
        \hline
    \end{tabular} \ 
    \begin{tabular}{|c|c|c|}
        \hline
        $A_{\transparent{0} 0}$ & $B$ & $C$\\
        \hline
        $B_{\transparent{0} 0}$ & $C$ & $A$\\
        \hline
        $C_{\transparent{0} 0}$ & $A$ & $B$\\
        \hline
    \end{tabular} \ 
    \begin{tabular}{|c|c|c|}
        \hline
        $A_{\transparent{0} 0}$ & $C$ & $B$\\
        \hline
        $B_{\transparent{0} 0}$ & $A$ & $C$\\
        \hline
        $C_{\transparent{0} 0}$ & $B$ & $A$\\
        \hline
    \end{tabular}\\
    \\
    \begin{tabular}{|c|c|c|}
        \hline
        $A$ & $B$ & $C$\\
        \hline
        $C$ & $A$ & $B$\\
        \hline
        $B$ & $C$ & $A$\\
        \hline
    \end{tabular} \ 
    \begin{tabular}{|c|c|c|}
        \hline
        ${\transparent{0} E}$ & ${\transparent{0} E}$ & ${\transparent{0} E}$\\
        \hline
        $A_1$ &  & \\
        \hline
        \cellcolor{lightgray} &  & \\
        \hline
    \end{tabular} \ 
    \begin{tabular}{|c|c|c|}
        \hline
        ${\transparent{0} E}$ & ${\transparent{0} E}$ & ${\transparent{0} E}$\\
        \hline
        $A_2$ &  & \\
        \hline
         &  & \\
        \hline
    \end{tabular}\\
    \\
    \begin{tabular}{|c|c|c|}
        \hline
        $A$ & $C$ & $B$\\
        \hline
        $C$ & $B$ & $A$\\
        \hline
        $B$ & $A$ & $C$\\
        \hline
    \end{tabular} \ 
    \begin{tabular}{|c|c|c|}
        \hline
        ${\transparent{0} E}$ & ${\transparent{0} E}$ & ${\transparent{0} E}$\\
        \hline
         &  & \\
        \hline
        $A_3$ & & \\
        \hline
    \end{tabular} \ 
    \begin{tabular}{|c|c|c|}
        \hline
        ${\transparent{0} E}$ & ${\transparent{0} E}$ & ${\transparent{0} E}$\\
        \hline
         &  & \\
        \hline
        $A_4$ &  & \\
        \hline
    \end{tabular}
    
    \end{tabular}
\]

    \item Case (iv). The entries with the same subscript are deduced at the same time given the entries with smaller subscripts. There is no assignment to the gray cell will result in a \1sat function with induced degree 1.
    \[
    \begin{tabular}{c}

    \begin{tabular}{|c|c|c|}
        \hline
        $X$ & $A$ & $A$\\
        \hline
        $A$ & $B$ & $C$\\
        \hline
        $A$ & $C$ & $B$\\
        \hline
    \end{tabular} \ 
    \begin{tabular}{|c|c|c|}
        \hline
        $A_{\transparent{0} 0}$ & $B_{\transparent{0} 0}$ & $C$\\
        \hline
        $B_{\transparent{0} 0}$ & $C_{\transparent{0} 0}$ & $A$\\
        \hline
        $C_{\transparent{0} 0}$ & $A_{\transparent{0} 0}$ & $B$\\
        \hline
    \end{tabular} \ 
    \begin{tabular}{|c|c|c|}
        \hline
        $A_{\transparent{0} 0}$ & $C$ & $B$\\
        \hline
        $B_{\transparent{0} 0}$ & $A$ & $C$\\
        \hline
        $C_{\transparent{0} 0}$ & $B$ & $A$\\
        \hline
    \end{tabular}\\
    \\
    \begin{tabular}{|c|c|c|}
        \hline
        $A$ & $B$ & $C$\\
        \hline
        $C$ & $A$ & $B$\\
        \hline
        $C$ & $B$ & $A$\\
        \hline
    \end{tabular} \ 
    \begin{tabular}{|c|c|c|}
        \hline
        $C_3$ & $A_{4}$ & ${\transparent{0} E}$\\
        \hline
        $A_1$ & $B_{6}$ & \\
        \hline
        $B_2$ & $C_{5}$ & \\
        \hline
    \end{tabular} \ 
    \begin{tabular}{|c|c|c|}
        \hline
        $C_{3}$ & ${\transparent{0} E}$ & ${\transparent{0} E}$\\
        \hline
        $A_1$ &  & \\
        \hline
        $B_2$ &  & \\
        \hline
    \end{tabular}\\
    \\
    \begin{tabular}{|c|c|c|}
        \hline
        $A$ & $C$ & $B$\\
        \hline
        $B$ & $C$ & $A$\\
        \hline
        $B$ & $A$ & $C$\\
        \hline
    \end{tabular} \ 
    \begin{tabular}{|c|c|c|}
        \hline
        $B_{3}$ & $A_{4}$ & ${\transparent{0} E}$\\
        \hline
        $C_2$ & $B_{5}$ & \cellcolor{lightgray}\\
        \hline
        $A_1$ &  & \\
        \hline
    \end{tabular} \ 
    \begin{tabular}{|c|c|c|}
        \hline
        $B_{3}$ & ${\transparent{0} E}$ & ${\transparent{0} E}$\\
        \hline
        $C_2$ &  & \\
        \hline
        $A_1$ &  & \\
        \hline
    \end{tabular}
    
    \end{tabular}
\]
\end{itemize}

For case (v) to (x), the stabilizing action for \eqref{inde} partitions them into three orbits $\{\text{(v)}, \text{(viii)}\}$, $\{\text{(vi)}, \text{(ix)}\}$ and $\{\text{(vii)}, \text{(x)}\}$. So it suffices to consider case (v), (vi) and (vii).

\begin{itemize}
    \item Case (v). The three $B$ in the gray cells immediately certify that the function has induced degree at least 2.
    \[
    \begin{tabular}{c}
    
    \begin{tabular}{|c|c|c|}
        \hline
        $X$ & $A$ & $A$\\
        \hline
        $A$ & $B$ & $C$\\
        \hline
        $A$ & $C$ & \cellcolor{lightgray}$B$\\
        \hline
    \end{tabular} \ 
    \begin{tabular}{|c|c|c|}
        \hline
        $A$ & $B$ & $C$\\
        \hline
        $B$ & $C$ & $A$\\
        \hline
        $C$ & $A$ & \cellcolor{lightgray}$B$\\
        \hline
    \end{tabular} \ 
    \begin{tabular}{|c|c|c|}
        \hline
        $A$ & $C$ & $B$\\
        \hline
        $B$ & $A$ & $C$\\
        \hline
        $C$ & $B$ & $A$\\
        \hline
    \end{tabular}\\
    \\
    \begin{tabular}{|c|c|c|}
        \hline
        $A$ & $B$ & $C$\\
        \hline
        $B$ & $C$ & $A$\\
        \hline
        $C$ & $A$ & \cellcolor{lightgray}$B$\\
        \hline
    \end{tabular} \ 
    \begin{tabular}{|c|c|c|}
        \hline
        ${\transparent{0} E}$ & ${\transparent{0} E}$ & ${\transparent{0} E}$\\
        \hline
         &  & \\
        \hline
         &  & \\
        \hline
    \end{tabular} \ 
    \begin{tabular}{|c|c|c|}
        \hline
        ${\transparent{0} E}$ & ${\transparent{0} E}$ & ${\transparent{0} E}$\\
        \hline
         &  & \\
        \hline
         &  & \\
        \hline
    \end{tabular}\\
    \\
    \begin{tabular}{|c|c|c|}
        \hline
        $A$ & $C$ & $B$\\
        \hline
        $C$ & $B$ & $A$\\
        \hline
        $B$ & $A$ & $C$\\
        \hline
    \end{tabular} \ 
    \begin{tabular}{|c|c|c|}
        \hline
        ${\transparent{0} E}$ & ${\transparent{0} E}$ & ${\transparent{0} E}$\\
        \hline
         &  & \\
        \hline
         &  & \\
        \hline
    \end{tabular} \ 
    \begin{tabular}{|c|c|c|}
        \hline
        ${\transparent{0} E}$ & ${\transparent{0} E}$ & ${\transparent{0} E}$\\
        \hline
         &  & \\
        \hline
         &  & \\
        \hline
    \end{tabular}
    
    \end{tabular}
\]

    \item Case (vi).
    \[
    \begin{tabular}{c}
    
    \begin{tabular}{|c|c|c|}
        \hline
        $X$ & $A$ & $A$\\
        \hline
        $A$ & $B$ & $C$\\
        \hline
        $A$ & $C$ & \cellcolor{lightgray}$B$\\
        \hline
    \end{tabular} \ 
    \begin{tabular}{|c|c|c|}
        \hline
        $A_{\transparent{0} 0}$ & $B_{\transparent{0} 0}$ & $C_{\transparent{0} 0}$\\
        \hline
        $B_{\transparent{0} 0}$ & $C_{\transparent{0} 0}$ & $A_{\transparent{0} 0}$\\
        \hline
        $C_{\transparent{0} 0}$ & $A_{\transparent{0} 0}$ & \cellcolor{lightgray}$B_{\transparent{0} 0}$\\
        \hline
    \end{tabular} \ 
    \begin{tabular}{|c|c|c|}
        \hline
        $A_{\transparent{0} 0}$ & $C_{\transparent{0} 0}$ & $B$\\
        \hline
        $B_{\transparent{0} 0}$ & $A_{\transparent{0} 0}$ & $C$\\
        \hline
        $C_{\transparent{0} 0}$ & $B_{\transparent{0} 0}$ & $A$\\
        \hline
    \end{tabular}\\
    \\
    \begin{tabular}{|c|c|c|}
        \hline
        $A$ & $B$ & $C$\\
        \hline
        $B$ & $C$ & $B$\\
        \hline
        $C$ & $A$ & $A$\\
        \hline
    \end{tabular} \ 
    \begin{tabular}{|c|c|c|}
        \hline
        ${\transparent{0} E_0}$ & ${\transparent{0} E_0}$ & ${\transparent{0} E_0}$\\
        \hline
         &  & \\
        \hline
         &  & \\
        \hline
    \end{tabular} \ 
    \begin{tabular}{|c|c|c|}
        \hline
        ${\transparent{0} E_0}$ & ${\transparent{0} E_0}$ & ${\transparent{0} E}$\\
        \hline
         &  & \\
        \hline
         &  & \\
        \hline
    \end{tabular}\\
    \\
    \begin{tabular}{|c|c|c|}
        \hline
        $A$ & $C$ & $B$\\
        \hline
        $C$ & $A$ & $A$\\
        \hline
        $B$ & $B$ & $C$\\
        \hline
    \end{tabular} \ 
    \begin{tabular}{|c|c|c|}
        \hline
        ${\transparent{0} E}$ & ${\transparent{0} E}$ & ${\transparent{0} E}$\\
        \hline
         &  & \\
        \hline
        $A_1$ & $C_3$ & \cellcolor{lightgray}$B_5$\\
        \hline
    \end{tabular} \ 
    \begin{tabular}{|c|c|c|}
        \hline
        ${\transparent{0} E}$ & ${\transparent{0} E}$ & ${\transparent{0} E}$\\
        \hline
         &  & \\
        \hline
        $A_2$ & $C_4$ & \\
        \hline
    \end{tabular}
    
    \end{tabular}
\]

    \item Case (vii).
    \[
    \begin{tabular}{c}
    
    \begin{tabular}{|c|c|c|}
        \hline
        $X$ & $A$ & $A$\\
        \hline
        $A$ & $B$ & $C$\\
        \hline
        $A$ & $C$ & $B$\\
        \hline
    \end{tabular} \ 
    \begin{tabular}{|c|c|c|}
        \hline
        $A$ & $B$ & $C_{\transparent{0} 0}$\\
        \hline
        $B$ & $C$ & $A_{\transparent{0} 0}$\\
        \hline
        $C$ & $A$ & $B_{\transparent{0} 0}$\\
        \hline
    \end{tabular} \ 
    \begin{tabular}{|c|c|c|}
        \hline
        $A$ & $C$ & $B_{\transparent{0} 0}$\\
        \hline
        $B$ & $A$ & $C_{\transparent{0} 0}$\\
        \hline
        $C$ & $B$ & $A_{\transparent{0} 0}$\\
        \hline
    \end{tabular}\\
    \\
    \begin{tabular}{|c|c|c|}
        \hline
        $A$ & $B$ & $C$\\
        \hline
        $B$ & $C$ & $A$\\
        \hline
        $C$ & $B$ & $A$\\
        \hline
    \end{tabular} \ 
    \begin{tabular}{|c|c|c|}
        \hline
        ${\transparent{0} E}$ & ${\transparent{0} E}$ & \cellcolor{lightgray}$B_6$\\
        \hline
         & & \cellcolor{lightgray}$B_5$\\
        \hline
         &  & $C_1$\\
        \hline
    \end{tabular} \ 
    \begin{tabular}{|c|c|c|}
        \hline
        ${\transparent{0} E}$ & ${\transparent{0} E}$ & ${\transparent{0} E}$\\
        \hline
         &  & \cellcolor{lightgray}$B_7$\\
        \hline
         &  & \\
        \hline
    \end{tabular}\\
    \\
    \begin{tabular}{|c|c|c|}
        \hline
        $A$ & $C$ & $B$\\
        \hline
        $C$ & $A$ & $B$\\
        \hline
        $B$ & $A$ & $C$\\
        \hline
    \end{tabular} \ 
    \begin{tabular}{|c|c|c|}
        \hline
        ${\transparent{0} E}$ & ${\transparent{0} E}$ & $A_4$\\
        \hline
         &  & $C_3$\\
        \hline
         &  & $A_2$\\
        \hline
    \end{tabular} \ 
    \begin{tabular}{|c|c|c|}
        \hline
        ${\transparent{0} E}$ & ${\transparent{0} E}$ & ${\transparent{0} E_0}$\\
        \hline
         &  & \\
        \hline
         &  & \\
        \hline
    \end{tabular}
    
    \end{tabular}
    \]
\end{itemize}

    \item \textbf{$U_f(x_1, x_2, x_3, 0)$ is identical to \eqref{As}.}

    If $U_f(x_1, x_2, 0, x_4)$ is isomorphic to \eqref{inde}, then it is isomorphic to one of the cases we have analyzed above by swapping the row and column indices of the blocks. So it remains to consider the another case.
    
    Since $U_f(x_1, x_2, 0, x_4)$ is isomorphic to \eqref{As}, it is identical to one of the functions among (v) to (x) we listed above. We start with the following partial function, where $E$ are the entries to be determined.
\[
    \begin{tabular}{c}
    
    \begin{tabular}{|c|c|c|}
        \hline
        $X$ & $A$ & $A$\\
        \hline
        $A$ & $B$ & $C$\\
        \hline
        $A$ & $C$ & $B$\\
        \hline
    \end{tabular} \ 
    \begin{tabular}{|c|c|c|}
        \hline
        $A$ & $B$ & $C$\\
        \hline
        $B$ & $C$ & $A$\\
        \hline
        $C$ & $A$ & $B$\\
        \hline
    \end{tabular} \ 
    \begin{tabular}{|c|c|c|}
        \hline
        $A$ & $C$ & $B$\\
        \hline
        $C$ & $B$ & $A$\\
        \hline
        $B$ & $A$ & $C$\\
        \hline
    \end{tabular}\\
    \\
    \begin{tabular}{|c|c|c|}
        \hline
        $A$ & $B$ & $C$\\
        \hline
         &  & \\
        \hline
         &  & \\
        \hline
    \end{tabular} \ 
    \begin{tabular}{|c|c|c|}
        \hline
        ${\transparent{0} E}$ & ${\transparent{0} E}$ & ${\transparent{0} E}$\\
        \hline
         &  & \\
        \hline
         &  & \\
        \hline
    \end{tabular} \ 
    \begin{tabular}{|c|c|c|}
        \hline
        ${\transparent{0} E}$ & ${\transparent{0} E}$ & ${\transparent{0} E}$\\
        \hline
         &  & \\
        \hline
         &  & \\
        \hline
    \end{tabular}\\
    \\
    \begin{tabular}{|c|c|c|}
        \hline
        $A$ & $C$ & $B$\\
        \hline
         &  & \\
        \hline
         &  & \\
        \hline
    \end{tabular} \ 
    \begin{tabular}{|c|c|c|}
        \hline
        ${\transparent{0} E}$ & ${\transparent{0} E}$ & ${\transparent{0} E}$\\
        \hline
         &  & \\
        \hline
         &  & \\
        \hline
    \end{tabular} \ 
    \begin{tabular}{|c|c|c|}
        \hline
        ${\transparent{0} E}$ & ${\transparent{0} E}$ & ${\transparent{0} E}$\\
        \hline
         &  & \\
        \hline
         &  & \\
        \hline
    \end{tabular}
    
    \end{tabular}
\]

Recall that the stabilizing actions for \eqref{As} includes (a) swapping the row and column coordinates, and (b) swapping 2nd and 3rd rows, columns and then blocks. These actions together with the action of swapping the 2nd and 3rd blocks partition cases (v) to (x) into four orbits: $\{\text{(v)}\}$, $\{\text{(vi)}, \text{(vii)}\}$, $\{\text{(viii)}\}$, $\{\text{(ix), \text{(x)}}\}$. Thus it suffices to analyze case (v), (vi), (viii) and (ix) respectively.

\begin{itemize}
    \item Case (v). The three $B$ in the gray cells certify the induced degree is at least 2 immediately.
    \[
    \begin{tabular}{c}
    
    \begin{tabular}{|c|c|c|}
        \hline
        $X$ & $A$ & $A$\\
        \hline
        $A$ & $B$ & $C$\\
        \hline
        $A$ & $C$ & \cellcolor{lightgray}$B$\\
        \hline
    \end{tabular} \ 
    \begin{tabular}{|c|c|c|}
        \hline
        $A$ & $B$ & $C$\\
        \hline
        $B$ & $C$ & $A$\\
        \hline
        $C$ & $A$ & \cellcolor{lightgray}$B$\\
        \hline
    \end{tabular} \ 
    \begin{tabular}{|c|c|c|}
        \hline
        $A$ & $C$ & $B$\\
        \hline
        $C$ & $B$ & $A$\\
        \hline
        $B$ & $A$ & $C$\\
        \hline
    \end{tabular}\\
    \\
    \begin{tabular}{|c|c|c|}
        \hline
        $A$ & $B$ & $C$\\
        \hline
        $B$ & $C$ & $A$\\
        \hline
        $C$ & $A$ & \cellcolor{lightgray}$B$\\
        \hline
    \end{tabular} \ 
    \begin{tabular}{|c|c|c|}
        \hline
        ${\transparent{0} E}$ & ${\transparent{0} E}$ & ${\transparent{0} E}$\\
        \hline
         &  & \\
        \hline
         &  & \\
        \hline
    \end{tabular} \ 
    \begin{tabular}{|c|c|c|}
        \hline
        ${\transparent{0} E}$ & ${\transparent{0} E}$ & ${\transparent{0} E}$\\
        \hline
         &  & \\
        \hline
         &  & \\
        \hline
    \end{tabular}\\
    \\
    \begin{tabular}{|c|c|c|}
        \hline
        $A$ & $C$ & $B$\\
        \hline
        $C$ & $B$ & $A$\\
        \hline
        $B$ & $A$ & $C$\\
        \hline
    \end{tabular} \ 
    \begin{tabular}{|c|c|c|}
        \hline
        ${\transparent{0} E}$ & ${\transparent{0} E}$ & ${\transparent{0} E}$\\
        \hline
         &  & \\
        \hline
         &  & \\
        \hline
    \end{tabular} \ 
    \begin{tabular}{|c|c|c|}
        \hline
        ${\transparent{0} E}$ & ${\transparent{0} E}$ & ${\transparent{0} E}$\\
        \hline
         &  & \\
        \hline
         &  & \\
        \hline
    \end{tabular}
    
    \end{tabular}
\]

    \item Case (vi).
    \[
    \begin{tabular}{c}
    
    \begin{tabular}{|c|c|c|}
        \hline
        $X$ & $A$ & $A$\\
        \hline
        $A$ & \cellcolor{lightgray}$B$ & $C$\\
        \hline
        $A$ & $C$ & $B$\\
        \hline
    \end{tabular} \ 
    \begin{tabular}{|c|c|c|}
        \hline
        $A$ & $B_{\transparent{0} 0}$ & $C$\\
        \hline
        $B$ & $C_{\transparent{0} 0}$ & $A$\\
        \hline
        $C$ & $A_{\transparent{0} 0}$ & $B$\\
        \hline
    \end{tabular} \ 
    \begin{tabular}{|c|c|c|}
        \hline
        $A$ & $C_{\transparent{0} 0}$ & $B$\\
        \hline
        $C$ & \cellcolor{lightgray}$B_{\transparent{0} 0}$ & $A$\\
        \hline
        $B$ & $A_{\transparent{0} 0}$ & $C$\\
        \hline
    \end{tabular}\\
    \\
    \begin{tabular}{|c|c|c|}
        \hline
        $A$ & $B$ & $C$\\
        \hline
        $B$ & $C$ & $B$\\
        \hline
        $C$ & $A$ & $A$\\
        \hline
    \end{tabular} \ 
    \begin{tabular}{|c|c|c|}
        \hline
        ${\transparent{0} E}$ & ${\transparent{0} E_0}$ & ${\transparent{0} E}$\\
        \hline
         &  & \\
        \hline
         &  & \\
        \hline
    \end{tabular} \ 
    \begin{tabular}{|c|c|c|}
        \hline
        ${\transparent{0} E}$ & ${\transparent{0} E_0}$ & ${\transparent{0} E}$\\
        \hline
         &  & \\
        \hline
         &  & \\
        \hline
    \end{tabular}\\
    \\
    \begin{tabular}{|c|c|c|}
        \hline
        $A$ & $C$ & $B$\\
        \hline
        $C$ & $A$ & $A$\\
        \hline
        $B$ & $B$ & $C$\\
        \hline
    \end{tabular} \ 
    \begin{tabular}{|c|c|c|}
        \hline
        ${\transparent{0} E}$ & ${\transparent{0} E}$ & ${\transparent{0} E}$\\
        \hline
         &  & \\
        \hline
         & $C_1$ & \\
        \hline
    \end{tabular} \ 
    \begin{tabular}{|c|c|c|}
        \hline
        ${\transparent{0} E}$ & ${\transparent{0} E}$ & ${\transparent{0} E}$\\
        \hline
         & \cellcolor{lightgray}$B_3$ & \\
        \hline
         & $C_2$ & \\
        \hline
    \end{tabular}
    
    \end{tabular}
\]

    \item Case (viii).
    \[
    \begin{tabular}{c}
    
    \begin{tabular}{|c|c|c|}
        \hline
        $X$ & $A$ & $A$\\
        \hline
        $A$ & $B$ & $C$\\
        \hline
        $A$ & $C$ & $B$\\
        \hline
    \end{tabular} \ 
    \begin{tabular}{|c|c|c|}
        \hline
        $A$ & $B_{\transparent{0} 0}$ & $C_{\transparent{0} 0}$\\
        \hline
        $B$ & $C_{\transparent{0} 0}$ & $A_{\transparent{0} 0}$\\
        \hline
        $C$ & $A_{\transparent{0} 0}$ & $B_{\transparent{0} 0}$\\
        \hline
    \end{tabular} \ 
    \begin{tabular}{|c|c|c|}
        \hline
        $A$ & $C_{\transparent{0} 0}$ & $B_{\transparent{0} 0}$\\
        \hline
        $C$ & $B_{\transparent{0} 0}$ & $A_{\transparent{0} 0}$\\
        \hline
        $B$ & $A_{\transparent{0} 0}$ & \cellcolor{lightgray}$C_{\transparent{0} 0}$\\
        \hline
    \end{tabular}\\
    \\
    \begin{tabular}{|c|c|c|}
        \hline
        $A$ & $B$ & $C$\\
        \hline
        $C$ & $A$ & $B$\\
        \hline
        $B$ & $C$ & $A$\\
        \hline
    \end{tabular} \ 
    \begin{tabular}{|c|c|c|}
        \hline
        ${\transparent{0} E}$ & ${\transparent{0} E_0}$ & ${\transparent{0} E}$\\
        \hline
         &  & \\
        \hline
         & $B_2$ & \cellcolor{lightgray}$C_4$\\
        \hline
    \end{tabular} \ 
    \begin{tabular}{|c|c|c|}
        \hline
        ${\transparent{0} E}$ & ${\transparent{0} E}$ & ${\transparent{0} E}$\\
        \hline
         &  & \\
        \hline
         & $B_1$ & \cellcolor{lightgray}$C_3$\\
        \hline
    \end{tabular}\\
    \\
    \begin{tabular}{|c|c|c|}
        \hline
        $A$ & $C$ & $B$\\
        \hline
        $B$ & $A$ & $C$\\
        \hline
        $C$ & $B$ & $A$\\
        \hline
    \end{tabular} \ 
    \begin{tabular}{|c|c|c|}
        \hline
        ${\transparent{0} E}$ & ${\transparent{0} E_0}$ & ${\transparent{0} E_0}$\\
        \hline
         &  & \\
        \hline
         &  & \\
        \hline
    \end{tabular} \ 
    \begin{tabular}{|c|c|c|}
        \hline
        ${\transparent{0} E}$ & ${\transparent{0} E_0}$ & ${\transparent{0} E_0}$\\
        \hline
         &  & \\
        \hline
         &  & \\
        \hline
    \end{tabular}
    
    \end{tabular}
    \]

    \item Case (ix).
    \[
    \begin{tabular}{c}
    
    \begin{tabular}{|c|c|c|}
        \hline
        $X$ & $A$ & $A$\\
        \hline
        $A$ & $B$ & $C$\\
        \hline
        $A$ & $C$ & $B$\\
        \hline
    \end{tabular} \ 
    \begin{tabular}{|c|c|c|}
        \hline
        $A$ & $B_{\transparent{0} 0}$ & $C_{\transparent{0} 0}$\\
        \hline
        $B$ & $C_{\transparent{0} 0}$ & $A_{\transparent{0} 0}$\\
        \hline
        $C$ & $A_{\transparent{0} 0}$ & $B_{\transparent{0} 0}$\\
        \hline
    \end{tabular} \ 
    \begin{tabular}{|c|c|c|}
        \hline
        $A_{\transparent{0} 0}$ & $C_{\transparent{0} 0}$ & $B_{\transparent{0} 0}$\\
        \hline
        $C_{\transparent{0} 0}$ & $B_{\transparent{0} 0}$ & $A_{\transparent{0} 0}$\\
        \hline
        $B_{\transparent{0} 0}$ & $A_{\transparent{0} 0}$ & $C_{\transparent{0} 0}$\\
        \hline
    \end{tabular}\\
    \\
    \begin{tabular}{|c|c|c|}
        \hline
        $A$ & $B$ & $C$\\
        \hline
        $C$ & $C$ & $B$\\
        \hline
        $B$ & $A$ & $A$\\
        \hline
    \end{tabular} \ 
    \begin{tabular}{|c|c|c|}
        \hline
        ${\transparent{0} E}$ & ${\transparent{0} E_0}$ & ${\transparent{0} E_0}$\\
        \hline
         &  & \\
        \hline
         &  & \\
        \hline
    \end{tabular} \ 
    \begin{tabular}{|c|c|c|}
        \hline
        ${\transparent{0} E_0}$ & ${\transparent{0} E_0}$ & ${\transparent{0} E_0}$\\
        \hline
         &  & \\
        \hline
         &  & \\
        \hline
    \end{tabular}\\
    \\
    \begin{tabular}{|c|c|c|}
        \hline
        $A$ & $C$ & $B$\\
        \hline
        $B$ & $A$ & $A$\\
        \hline
        $C$ & $B$ & $C$\\
        \hline
    \end{tabular} \ 
    \begin{tabular}{|c|c|c|}
        \hline
        ${\transparent{0} E}$ &  & \\
        \hline
         & $B_3$ & \\
        \hline
         & $C_4$ & \cellcolor{lightgray}$A_1$ \\
        \hline
    \end{tabular} \ 
    \begin{tabular}{|c|c|c|}
        \hline
         &  & \\
        \hline
         & $C_2$ & \\
        \hline
        \cellcolor{lightgray}$A_6$ & $B_5$ & \cellcolor{lightgray}$A_7$\\
        \hline
    \end{tabular}
    
    \end{tabular}
    \]
\end{itemize}
\end{itemize}

In all cases $U_f$ must have induced degree at least 2 and this completes the proof.
\end{proof}
\end{document}